\documentclass[11pt]{article}

\usepackage{amsthm}
\usepackage{amsmath} 
\usepackage{amssymb}
\usepackage[margin=1in]{geometry}
\usepackage{enumerate}
\usepackage{hyperref}
\usepackage{mathrsfs}
\usepackage{color}
\usepackage{bm}
\usepackage{bbm}
\usepackage[font={small,it}]{caption}
\usepackage{subcaption}
\usepackage{overpic}
\usepackage{graphicx}
\usepackage{epic}
\usepackage{enumitem}
\usepackage{epstopdf}
\usepackage{authblk}
\DeclareMathOperator{\spn}{span}
\usepackage{mathtools}

\title{Detecting random bifurcations via rigorous enclosures of large deviations rate functions}
 \author[1]{Alexandra Blessing (Neam\c tu)}
\author[2]{Alex Blumenthal}
\author[3]{Maxime Breden}
\author[4,5]{Maximilian Engel}
\date{\today}

\affil[1]{Department of Mathematics and Statistics,
Universität Konstanz,
78457 Konstanz, Germany}

\affil[2]{School of Mathematics,
Georgia Institute of Technology, 686 Cherry Street, Atlanta, GA 30332-0160 USA}

\affil[3]{CMAP, CNRS, Ecole polytechnique, Institut Polytechnique de Paris, 91120 Palaiseau, France}

\affil[4]{Institute of Mathematics, Freie Universit\"at Berlin, Arnimallee 7, 14195 Berlin, Germany}

\affil[5]{KdV Institute, University of Amsterdam, Science Park 105-107, 1098 XG Amsterdam, The Netherlands}

\theoremstyle{plain}
\newtheorem{thm}{Theorem}[section]
\newtheorem{cor}[thm]{Corollary}
\newtheorem{lem}[thm]{Lemma}
\newtheorem{prop}[thm]{Proposition}

\theoremstyle{definition}

\newtheorem{defn}[thm]{Definition}
\newtheorem{rmk}[thm]{Remark}
\newtheorem{ass}[thm]{Assumption}

\newtheorem{exam}[thm]{Example}

\newcommand{\C}{\mathbb{C}}

\newcommand{\E}{\mathbb{E}}

\newcommand{\N}{\mathbb{N}}
\renewcommand{\P}{\mathbb{P}}
\newcommand{\R}{\mathbb{R}}

\newcommand{\Z}{\mathbb{Z}}

\newcommand{\Bc}{\mathcal{B}}
\newcommand{\Fc}{\mathcal{F}}

\newcommand{\Lc}{\mathcal{L}}

\newcommand{\Ec}{\mathcal E}
\newcommand{\Xc}{\mathcal X}
\newcommand{\Dc}{\mathcal D}

\newcommand{\Ic}{\mathcal{I}}

\newcommand{\barf}{\bar{f}}
\newcommand{\barl}{\bar{\lambda}}
\newcommand{\barX}{\bar{X}}

\renewcommand{\a}{\alpha}
\renewcommand{\b}{\beta}

\renewcommand{\l}{\lambda}

\newcommand{\T}{\mathbb T}
\newcommand{\Cc}{\mathcal C}

\newcommand{\rmd}{\mathrm{d}}
\newcommand{\rmD}{\mathrm{D}}
\renewcommand{\epsilon}{\varepsilon}

\renewcommand{\l}{\lambda}
\newcommand{\ol}{\overline{\l}}
\newcommand{\ul}{\underline{\l}}

\numberwithin{equation}{section}

\DeclareMathOperator{\spec}{spec}

\newcommand{\parahead}[1]{\bigskip 

\noindent {\bf #1 } 

\medskip 
}

\begin{document}

\maketitle


\abstract{
The main goal of this work is to provide a description of transitions from uniform to non-uniform snychronization in diffusions based on large deviation estimates for finite time Lyapunov exponents. These can be characterized in terms of moment Lyapunov exponents which are principal eigenvalues of the generator of the tilted (Feynman-Kac) semigroup. Using a computer assisted proof, we demonstrate how to determine these eigenvalues and investigate the rate function which is the Legendre-Fenchel transform of the moment Lyapunov function. 
We apply our results to two case studies: the pitchfork bifurcation and a two-dimensional toy model,  also considering the transition to a positive asymptotic Lyapunov exponent.
}

\section{Introduction}
  \label{sec:intro}


Local bifurcations in ordinary differential equations result in qualitative changes to the long-time behavior of the dynamics through the creation or destruction of equilibria or periodic orbits. 
As past work has uncovered, however, the situation is more subtle in the presence of noise, e.g., for stochastic differential equations forced by white-in-time noise. 
    
    To illustrate the point and motivate the present manuscript, consider the normal form of a pitchfork bifurcation forced by additive, white-in-time noise: 
     \begin{equation}
        \label{eq:pitchforkIntro}
        \rmd X_t = (\alpha X_t - X^3_t) \rmd t + \sigma \rmd W_t, \ X_0 \in \mathbb{R} \, . 
        \end{equation}
    Absent noise ($\sigma = 0$), the resulting ODE undergoes a pitchfork bifurcation at $\alpha = 0$, resulting in the creation of a pair of stable equilibria and a change of stability for the equilibrium at $0$. 
    When $\sigma > 0$, however, it was shown in \cite{CrauelFlandoli} that some features of the pitchfork are `destroyed' in the following sense: if $\varphi^t$ denotes the stochastic flow generated by equation \eqref{eq:pitchforkIntro} -- the analogue of the solution flow to an ODE -- then for all fixed initial $X_0, Y_0 \in \R$, one has that 
    \begin{align} \label{eq:syncIntro}| \varphi^t(X_0) - \varphi^t(Y_0)| \to 0 \qquad \text{ as } t \to \infty  \end{align}
    with probability 1, irrespective of the value of $\alpha$. Contrast this to the standard pitchfork ($\sigma=0$), where \emph{synchronization} as in \eqref{eq:syncIntro} also occurs unconditionally when $\alpha< 0$, but only if $X_0, Y_0$ are nonzero and share the same sign when $\alpha>0$.
    
    However, it was pointed out in Callaway et al. \cite{Callawayetal} that one can recover some ``signature'' of the original pitchfork by examining quantitatively the synchronization property \eqref{eq:syncIntro}. Their argument is framed around the statistics of \emph{finite-time Lyapunov exponents} (FTLE)
   \begin{align} \label{eq:FTLEintro} \lambda_t(\alpha; X_0) = \frac{1}{t} \log | (\varphi^t)'(X_0) |\end{align}
    for fixed initial $X_0$. 
    When $t \to \infty$, the ergodic theorem implies $\lambda_t(\alpha; X_0)$ converges, with probability 1 and for all $X_0 \in \R$, to the \emph{asymptotic} Lyapunov exponent $\lambda(\alpha) \in \R$. For equation \eqref{eq:pitchforkIntro}, it holds that $\lambda(\alpha) < 0$ for all $\alpha \in \R$, compatible with the synchronization property \eqref{eq:syncIntro}. 
    
    On the other hand, the finite-time versions $\lambda_t(\alpha; X_0)$ undergo a qualitative transition through the bifurcation value $\alpha = 0$: 
    \begin{itemize}
        \item[(i)] For $\alpha < 0$, it holds that $\lambda_t(\alpha; X_0) < 0$ for all $X_0 \in \R$ and with probability 1; while 
        \item[(ii)] for $\alpha > 0$ the event $\{ \lambda_t(\alpha; X_0) > 0\}$ has \emph{positive probability for all $t > 0, X_0 \in \R$.}
    \end{itemize}
    Statement (ii) holds because, with positive probability, trajectories $X_t = \varphi^t(X_0)$ can linger near the unstable region at the origin for arbitrarily long amounts of time. 
    Viewed another way, for $\alpha > 0$ synchronization as in \eqref{eq:syncIntro} is only true asymptotically, as $t \to \infty$, while transient derivative growth can persist for arbitrarily long times with positive probability, preventing the merging of $\varphi^t(X_0)$ and  $\varphi^t(Y_0)$. For additional discussion on the connection between FTLE and synchronization, see \cite{Callawayetal}. 

    \subsubsection*{Summary of this paper}




    The purpose of this paper is to study quantitatively the statistics of $\lambda_t$ and the transition from negative FTLE to positive FTLE with positive probability, focusing on the case of stochastic flows generated by SDE. 
    Our results are framed around the \emph{large deviations rate function} tracking the statistics of deviations of the FTLE from its asymptotic value. For stochastic flows on $\R$, and for $r > \lambda(\alpha)$, this rate function takes the form
    \[\Ic_\alpha(r) := - \lim_{t \to \infty} \frac1t \log \P(\lambda_t(\alpha; X_0) > r) \in (0,\infty] \, ,  \]
    where the RHS limit definition\footnote{Here we use the convention that $\log 0 = -\infty$.} exists and is constant over all $X_0 \in \R$ and under mild conditons \cite{ArnoldKliemann1987}. 
    In particular, when $\lambda(\alpha)<0$, the rate $\Ic_\alpha(0)$ quanitifies the likelihood of transient nonsynchronization in the sense that 
    \[\P(\lambda_t(\alpha; X_0) > 0) \approx e^{- t \Ic_\alpha(0)}\]
    at long times $t$.


    
This paper contains: 
\begin{itemize}
    \item[(a)] Conditions guaranteeing the existence of the rate function $\mathcal{I}$ for a broad class of stochastic flows (Section \ref{sec:prelim}); 
    \item[(b)] Some general results on transitions from negative to positive FTLE (Section \ref{sec:randBifurc}), including: 
        \begin{itemize}
            \item[(i)] A general result for stochastic flows connecting transitions to positive FTLE with the LDP rate function (Theorem \ref{thm:main} and Proposition \ref{prop:LDPsignChange}); 
            \item[(ii)] Some necessary and some sufficient conditions for transitions to positive FTLE in special cases; 
        \end{itemize}
    \item[(c)] Computer-assisted estimates of the rate function $\Ic$ for two sets of worked examples: 
    \begin{itemize}
        \item[(i)] the stochastic pitchfork \eqref{eq:pitchforkIntro}; and 
        \item[(ii)] a toy model given by a 2d linear SDE. 
    \end{itemize}

\end{itemize}

The 2d linear SDE model is not ``dynamical'', i.e., originating as the linearization of some nonlinear system, but it can be realized as an appropriate scaling limit of a prototypical system exhibiting shear-induced chaos \cite{LinYoung,BreEng23,ChemnitzEngel}. We include it here to demonstrate how a computer-assisted proof can be carried out for the estimation of LDP rate functions in higher dimensions, entailing 
 a treatment of the \emph{projective process}, a Markov process tracking tangent directions of the stochastic flow. 
 
 We note as well that the methods involved in the computer-assisted analysis for models (c)(i), (c)(ii) are substantially different: roughly speaking, (c)(i) amounts to finding eigenvalues of a self-adjoint operator on an unbounded space, and (c)(ii) to finding eigenvalues of a non-self-adjoint operator on a compact space. 

\subsubsection*{Validated computer-assisted enclosures of $\Ic_\a(0)$ for the stochastic pitchfork}


Figure~\ref{fig:pitchfork_various_alphas_intro} depicts a computer-assisted estimate of the rate function $\Ic_\alpha(0)$ for the stochastic pitchfork \eqref{eq:pitchforkIntro} as $\alpha$ is varied. 
Notable here is the nonlinear dependence of $\alpha \mapsto \Ic_\alpha(0)$, which we certify to exhibit a local minimum in the interval $\alpha \in [1.225,1.229]$ (see Theorem~\ref{th:pitchfork_minimum}). 
For comparison, we depict in Figure~\ref{fig:pitchfork_various_alphas_LE}
approximate values for $\lambda(\alpha)$ obtained from the expression $\lambda(\alpha) = \int_\R (\alpha - 3x^2) \rho(x) \rmd x$ (see Theorem~\ref{thm:mlp} and Theorem~\ref{thm:noncompactLDP} for more background on this so called Furstenberg-Khasminskii formula, and~\eqref{eq:defnStatDenstiyPitchforkIntro} for the formula of the stationary density $\rho$).

In view of the connection between synchronization and FTLE, the enclosures presented in Figure \ref{fig:pitchfork_various_alphas_intro} suggest that synchronization is at its `weakest' at the minimum of $\alpha \mapsto \Ic_\alpha(0)$. Heuristically, the existence of this minimum should have to do with the tradeoff, between increased expansion near $0$, which increases FTLE, and a shortening of the time spent in the expanding region, which decreases LE and FTLE. While this heuristic applies to FTLE as well as to LE, it is notable that the LE and FTLE exhibit their local extrema at distinct values of the parameter $\alpha$: the argmax for $\alpha \mapsto \lambda(\alpha)$ at $\alpha \approx 0.25$ occurs appreciably below the argmin $\alpha \approx 1.2$ for $\alpha \mapsto \Ic_\alpha(0)$.

\begin{figure}[h!]
    \includegraphics[width=0.49\linewidth]{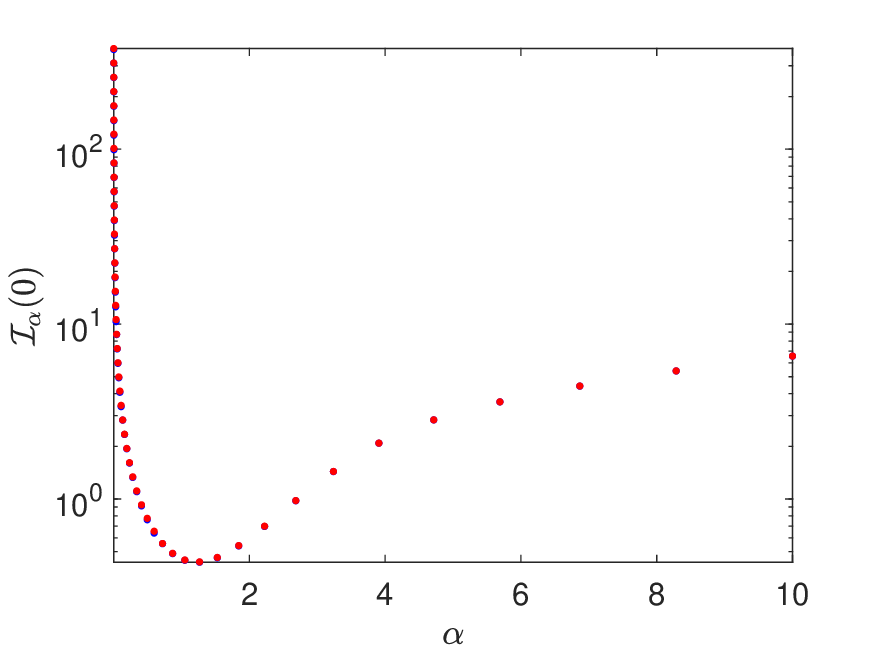}
    \hfill
    \includegraphics[width=0.49\linewidth]{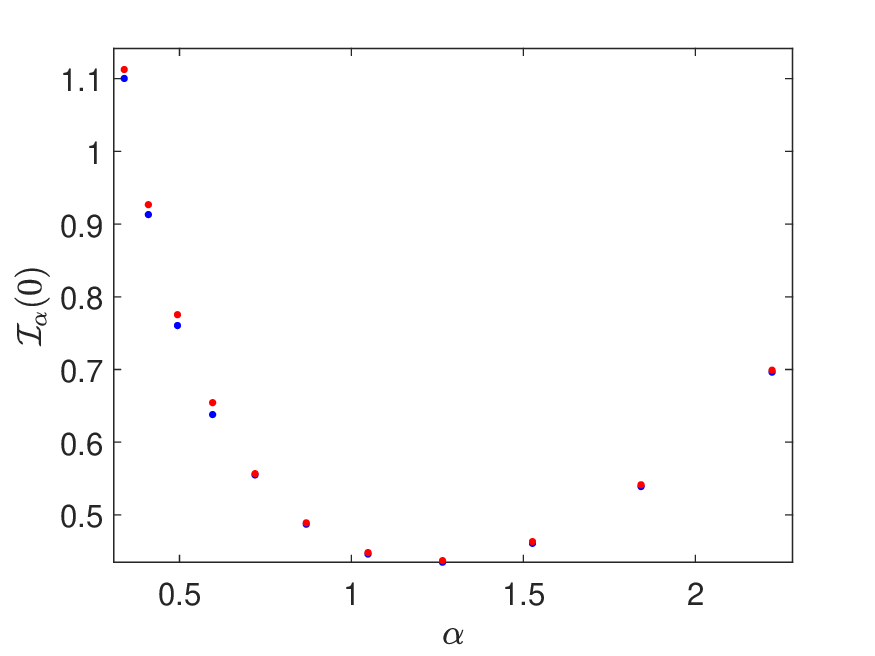}
    \caption{Rigorous enclosures of $\Ic_\alpha(0)$ for different values of $\alpha$, and $\sigma=1$ in~\eqref{eq:pitchforkIntro}. Upper-bounds are shown in red, and lower bounds in blue, but they are close enough to be hard to distinguish. The right figure is zoomed in near a local minimum at $\alpha \approx 1.2$.}
    \label{fig:pitchfork_various_alphas_intro}
\end{figure}

\begin{figure}[h!]
\centering
    \includegraphics[width=0.5\linewidth]{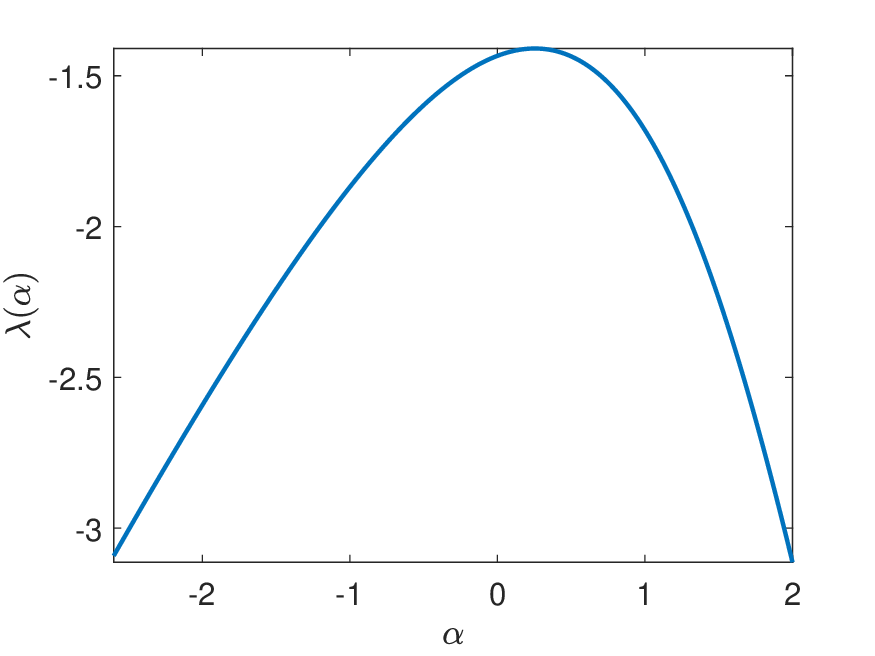}
    \caption{Numerical computation of $\lambda(\alpha)$ for different values of $\alpha$, and $\sigma=1$ in~\eqref{eq:pitchforkIntro}. 
    \label{fig:pitchfork_various_alphas_LE}}
\end{figure}



\subsection*{Previous work}

\noindent {\bf Stochastic bifurcations. } Work in this field has historically distinguished between {\it phenomenological bifurcations}, to do with changes in the stationary measure, e.g., unimodal to bimodal, versus dynamical bifurcations, to do with the change in the stability type of typical sample paths signalled by a change in sign of the (asymptotic) Lyapunov exponent of the system. For a general discussion, see, e.g., the overviews \cite{arnold1996toward} or \cite[Chapter 9]{Arnoldbook} or the representative works \cite{blumenthal2022positive, EngelLambRasmussen, GalMonNis20, lamb2015topological, teramae2004robustness, zmarrou2007bifurcations}.

 {The additive pitchfork model \eqref{eq:pitchforkIntro} exhibits a phenomenological bifurcation, as can be seen from the form of the stationary density 
\begin{align} \label{eq:defnStatDenstiyPitchforkIntro}\rho_{\sigma, \alpha}(x) = \frac{1}{Z} \operatorname{exp} \frac{2}{\sigma^2}\left(\frac12 \alpha x^2 - \frac14 x^4\right) \, , \end{align}
with $Z > 0$ a normalization constant. At a fixed noise amplitude $\sigma > 0$, the stationary density for \eqref{eq:pitchforkIntro} goes from unimodal at $\alpha < 0$ to bimodal at $\alpha > 0$, reflecting the transition of stationary mass away from the now unstable region near $0$ to the new stable regions near $\pm \sqrt{\alpha}$.   } 

{This co-incidence of P-bifurcations and the nonuniform synchronization scenario is likely to be common for a variety of one-dimensional stochastic flows. We caution, however, that a priori there is no reason to expect the two to coincide, especially for higher-dimensional systems: a P-bifurcation is a property of the stationary measure alone and has no bearing on FTLE, which have more to do with the fine statistics of tangent directions. 
  }

The transition of FTLE from negative to positive has been established in a variety of settings, including: near a Hopf bifurcation with additive noise \cite{DoanEngeletal}; as noise amplitude is increased in the logistic map family in a periodic window \cite{sato2018dynamical}; and for noised pitchfork bifurcations in Chafee-Infante and Swift-Hohenberg SPDEs \cite{BN23, BlumenthalEngelNeamtu2021}. We also refer the reader to Section \ref{sec:outlook} for additional situations where the FTLE transition might take place.

\bigskip

\noindent {\bf Large deviations. } Our work leans on existing large deviations principles for Lyapunov exponents and their connection to the spectral theory of tilted (i.e.~Feynman-Kac) semigroups \cite{arnold1984formula, ArnoldKliemann1987, Stroock86}. More broadly, there is a long literature of work applying large deviations principles and related ideas to describe metastable states and other transient behaviors in random systems \cite{freidlin1998random}. 

Particularly relevant to us is the work of Berglund and Gentz \cite{berglund2002pathwise} which considers \eqref{eq:pitchforkIntro} in a small-noise regime ($\sigma \ll 1$) taking the parameter $\alpha$ to vary slowly through the bifurcation. The work \cite{berglund2002pathwise} seeks to understand the role of small additive noise on \emph{bifurcation delay} -- a tendency for slowly-varying parameters to `delay' the transition of an equilibrium from stable to unstable. It is shown that additive noise reduces the order of the bifurcation delay present in the deterministic system. See also the related studies \cite{berglund2002metastability, berglund2002sample} and, for a broader perspective on the role of noise in slowly-varying systems, the book \cite{berglund2006noise}.

A crucial difference in our case is that we work beyond perturbative regimes, with (albeit time-independent) macroscopic values of the bifurcation parameter $\alpha$ as well as the noise amplitude $\sigma$. This regime is harder to treat analytically, and computer-assisted proof is a natural tool, allowing for instance to discover nonperturbative features such as the discrepancy between the argmin of $\alpha \mapsto \Ic_\alpha(0)$ and the argmax of $\alpha \mapsto \lambda(\alpha)$ depicted in Figures \ref{fig:pitchfork_various_alphas_intro}, \ref{fig:pitchfork_various_alphas_LE}.


\bigskip

\noindent {\bf Computer-assisted proof. } Computer-assisted proofs in dynamical systems go back, at least, to the pioneering work of Lanford on the Feigenbaum conjecture~\cite{Lan82}. The usage of rigorous numerics for stochastic systems is more recent; lately, computed-assisted proofs have been leveraged in order to establish the existence of some forms of \emph{transition} in random dynamical systems outside of perturbative regimes~\cite{GalMonNis20,BreEng23,BreChuLamRas24}. However, these works study transitions characterized by an asymptotic Lyapunov exponent, which, as already mentioned, is too crude of a tool for the kind of problems we focus on in this paper. 

As will be explained in Section~\ref{sec:prelim}, the central quantities for our study of the large deviation rate function $\Ic_\alpha$ are moment Lyapunov exponents $\Lambda_\alpha(p)$. These can be characterized as principal eigenvalues of $p$-dependent elliptic equations (tilted generators), and the computer-assisted part of our argument consists in getting tight and rigorous enclosures for such eigenvalues. 
For self-adjoint operators, there exist computer-assisted techniques specifically devised for rigorously enclosing eigenvalues (see e.g.~\cite{Plu90,Liu15} and \cite[Chapter 10]{NakPluWat19} for a broader overview). More generally, one can write down a zero-finding problem for an eigenpair, and then use more generic computer-assisted techniques for enclosing zeros based on the Newton-Kantorovich Theorem~\cite{BerLes15}. As will also be highlighted in Section~\ref{sec:prelim}, having precise but only punctual information on $\Lambda_\alpha$ (that is, knowing $\Lambda_\alpha(p)$ for a few specific values of $p$) is already enough to precisely enclose $\Ic_\alpha(0)$. However, even more information becomes available if one can control, at least locally, the map $p\mapsto\Lambda_\alpha(p)$ and some of its derivatives. This will be achieved thanks to recent developments on rigorous continuation techniques from~\cite{Bre23} (see also~\cite{AriGazKoc21}). We refer to Section~\ref{sec:CAP} for more details on the different computer-assisted techniques we employ depending on the specifics of our two examples, and to Section~\ref{sec:outlook} for open problems related to computer-assisted proofs suggested by our investigations.

\bigskip

\noindent {\bf Potential applications. } 
FTLEs have encountered growing interest and usage in different subfields of climate and Earth science, serving as indicators of transitions between metastable regimes. A finite-time framework is particularly valuable here in view of the simple fact that both empirical and numerical data provide only finite time series, and so in principle the asymptotic Lyapunov exponent need not be directly observed.

Numerical FTLE and other dynamical observables of the quasi-geostrophic model of Marschall and Molteni \cite{marshall1993toward} have been studied in \cite{lucarini2020new}. In this context, it is argued that the emergence and disappearance of \emph{atmospheric blocking} events corresponds to an increase in the number of linearly independent unstable directions for the finite-time dynamics, i.e., an increase in the dimension of unstable finite-time \emph{covariant Lyapunov vectors} -- see, e.g., \cite{ginelli2007characterizing}.

A data-based method \cite{metzner2012analysis} has been used in \cite{quinn2021dynamical} to yield a linear model that exhibits FTLEs and CLVs, indicating stable and unstable directions corresponding with atmospheric regimes in the North Atlantic Oscillation. Such methods and their reliability have been tested in \cite{viennet2022guidelines}.

Another field of application concerns glacial cycles, i.e.~the life time and recurrence patterns of glacial periods. They have been studied under quasi-periodic astronomic forcing \cite{mitsui2014dynamics}, yielding non-chaotic strange attractors. The transitions between ordered periods and chaotic transients have been investigated also in regimes with additive noise \cite{mitsui2016effects}.
A theoretical quantitative understanding, as presented here, may well be relevant for such models.

\section{Preliminaries}
  \label{sec:prelim}

This section lays out necessary preliminaries and assumptions required for the definition of stochastic flows, Lyapunov exponents and rate functions. Section \ref{subsec:compact2} addresses the case of a compact state space $M$, already covered in the literature (see \cite{BaxendaleStroock88} or, for linear diffusions on $\mathbb{R}^d$ driven by a process on $M$, see \cite{ArnoldKliemann1987}), while Section \ref{subsec:noncompact2} discusses what further assumptions are necessary in the noncompact case.

\subsection{Large deviations and FTLE on a compact state space} \label{subsec:compact2}

Let $M$ be a compact, connected, smooth Riemannian manifold of dimension $d$ and 
consider the general form of a stochastic differential equation 
\begin{equation}
\label{eq:SDE_gen}
\rmd X_t = f_0(X_t)\,\rmd t + \sum_{j=1}^m f_j(X_t) \circ \rmd W_t^j\,,  \ X_0 = x \in M,
\end{equation}
where the $W_t^j$, $j\geq 1$ are $m$ independent Brownian motions and $\circ$ denotes the Stratonovich differential. Here, the vector fields $f_0, f_1, \dots, f_m$ are assumed to be smooth, hence \eqref{eq:SDE_gen} admits unique strong solutions for all initial data. 
We write
\begin{equation}
\label{eq:G}
G = f_0 + \frac{1}{2} \sum_{j=1}^m f_j^2
\end{equation}
for the generator of $(X_t)$ in H\"ormander form. 

For vector fields $f, g: M \to TM$, we denote by $[f,g]$ the Lie bracket between the two vector fields.

\begin{ass}\label{ass:statDensity} \ 
    \begin{itemize}
        \item[(a)] Given the vector fields $(f_j)$ in equation~\eqref{eq:SDE_gen}, we consider the collection of vector fields $\mathcal V_k$ given by
$$ \mathcal{V}_0 = \{ f_j \,:\, j > 0 \}, \quad \mathcal V_{k+1} = \mathcal V_k \cup \{ [g, f_j] \,:\, g \in \mathcal V_k, \, j \geq 0 \},$$ 
and the vector spaces, for $x \in M$,
$$ \mathcal V_k(x) = \spn \{ f(x) \,:\, f \in \mathcal V_k \}.$$ 
     We assume that~\eqref{eq:SDE_gen}     satisfies the \emph{parabolic H\"ormander condition} 
        \begin{align}\label{eq:parHormander}
            \cup_{k \geq 1} \mathcal V_k(x) = T_x M, \quad \text{for all } x \in M.
        \end{align}
        \item[(b)] The SDE \eqref{eq:SDE_gen} satisfies the following \emph{path controllability condition}: for any $X_0 \in M$ and $T > 0$, the law of $(X_t)_{0 \leq t \leq T}$ is locally positive in $C_{X_0}([0,T], M)$, i.e.~attributes positive measure to any non-empty open subset of $C_{X_0}([0,T], M)$ with respect to the compact-open topology, where $C_x([0,T], M)$ is the space of continuous paths $\gamma : [0,T] \to M$ with $\gamma(0) = x \in M$. 

        
    \end{itemize}
\end{ass}






Assumption \ref{ass:statDensity}(a) implies absolute continuity of the transition kernels $P^t(x, \cdot)$ for all fixed $x \in M$ \cite{hormander1967hypoelliptic}, while (b) can be checked for SDEs using control theory and the Stroock-Varadhan Theorem \cite[Theorem 6.1]{bellet2006ergodic}. Assumption \ref{ass:statDensity}(b) implies  the following strong irreducibility of the Markov process $(X_t)$: 
 \begin{align} \label{eq:irred} P^t(x, U) > 0 \quad \text{ for all } \quad U \subset M \text{ open.} \end{align}
 Here, for $t > 0$ and $K \subset M$ Borel, we write $P^t(x, K) = \P(X_t \in K |X_0 = x)$ for the Markov kernel of $(X_t)$.
Note that assumption \ref{ass:statDensity}(b) may be relaxed to a condition on the end points $X_T$ of the paths for the results of this section. However, we will need the full assumption on the whole path for the proof of Theorem~\ref{thm:main}. 
In order to work with the same set of conditions throughout the following, we have chosen this formulation of assumption \ref{ass:statDensity}(b).

Taken together, it is standard that Assumptions \ref{ass:statDensity}(a) and (b) imply that $(X_t)$ admits a unique stationary measure $\nu$ with smooth, strictly positive density $\rho$. 

In the following, we write $\langle, \rangle_x$ for the metric on $T_x M$ at $x$ and  $\| \cdot \|_x$ for the corresponding norm, suppressing $x$ in the notation when it is clear from the context.
Furthermore, for a vector field $f: M \to TM$, we write $\rmD f$ for the covariant derivative of the \emph{Levi-Civita connection}, viewed for fixed $x \in M$ as a linear operator $\rmD f(x) : T_x M \to T_x M$. 
Finally, for a linear operator $A : T_x M \to T_x M$, we will write $A^\top : T_x M \to T_x M$ for its transpose with respect to the inner product, i.e., the unique linear operator for which $\langle A^\top v, w \rangle = \langle v, A w\rangle$ for all $v, w \in T_x M$.



We further consider the extended variational process $(X_t, Y_t)$ on $TM$ given by~\eqref{eq:SDE_gen} and
\begin{equation}
\label{eq:var_process}
\rmd Y_t = \rmD f_0(X_t) Y_t\,\rmd t + \sum_{j=1}^m \rmD f_j(X_t) Y_t \circ \rmd W_t^j\,,  \ Y_0 \in T_x M\, .
\end{equation}

We can now proceed similarly to \cite{ArnoldKliemann1987}, working more generally on the tangent bundle $T M$, see \cite{BaxendaleStroock88}.
We introduce the polar coordinates $\|Y_t\| \in \mathbb{R}_{> 0}$ and $s_t = Y_t/ \|Y_t\|$. Identifying $s_t = - s_t$ we obtain the \emph{projective process} $(X_t,s_t)$ on the projective bundle $P M$ with fibers $P_x M = P(T_x M) \cong P^{d-1}$: 
\begin{equation}
\label{eq:projective_process}
\rmd s_t =  h_0(X_t, s_t) \rmd t 
+ \sum_{j=1}^m  h_j(X_t, s_t) \circ \rmd W_t^j, \quad  s_0 = \frac{Y_0}{\|Y_0\|},
\end{equation} 
where
\begin{align*}
h_j (X,s) =  \rmD f_j(X) s - \langle s, \rmD f_j(X) s \rangle s.
\end{align*}
In particular, $(X_t, s_t)$ is a diffusion process on $PM$ whose generator is given by
\begin{equation}
\label{eq:generator_bundle}
 L = (f_0 + h_0) + \frac{1}{2} \sum_{j=1}^m (f_j + h_j)^2.
\end{equation}
Finally, for $(X,s) \in PM$ we write $\mathbb{P}_{(X, s)}, \mathbb{E}_{(X,s)}$ for probabilities and expectations conditioned on $(X_0, s_0) = (X, s)$. The notation $\mathbb{P}_X, \mathbb{E}_X$ 
is defined analogously.

\begin{ass}\label{ass:projStat} \ 
    \begin{itemize}
        \item[(a)] The vector fields $\begin{pmatrix}
            f_j \\ h_j
        \end{pmatrix}$ on $PM$ satisfy the analogue of the parabolic H\"ormander condition \eqref{eq:parHormander} on $PM$. 
        \item[(b)] The $(X_t, s_t)$ process satisfies an analogue of the path controllability condition on $PM$. 
    \end{itemize}
    Like before, (a) and (b) taken together imply the existence of a unique stationary measure $\mu$ on $PM$ with smooth, positive density on $PM$. Observe that Assumption \ref{ass:statDensity} is implied by Assumption \ref{ass:projStat}, and that in this setting the stationary measure $\mu$ on $PM$ projects to the unique stationary measure $\nu$ on $M$ \cite{ArnoldKliemannOel, ArnoldOelPardoux}.
\end{ass}

In addition, one may readily deduce that the radial component of $Y_t$, which measures the growth of the variational process, satisfies, in Stratonovich form,
\begin{equation*}
 \|Y_t\| = \|Y_0\| \exp \left( \int_0^t  q_0(X_\tau, s_\tau) \rmd \tau + \sum_{j=1}^m \int_0^t q_j(X_\tau, s_\tau)  \circ \rmd W_\tau^j \right) \, , 
\end{equation*}
which in It\^o form reads
\begin{equation}
\label{eq:linear_growth}
 \|Y_t\| = \|Y_0\| \exp \left( \int_0^t  Q(X_\tau, s_\tau) \rmd \tau + \sum_{j=1}^m \int_0^t q_j(X_\tau, s_\tau)  \rmd W_\tau^j \right) \, . 
\end{equation}
In particular, when $\| Y_0\| = 1$ is fixed (as we will do implicitly throughout what follows), $\| Y_t\|$ depends only on the representative $s_0 \in P_x M$. Here, we have written 
\begin{align}\label{eq:defnQ2}\begin{aligned}
q_j(X,s) &= \langle s, \rmD f_j(X) s \rangle, \\
Q(X,s) &= q_0(X, s) + \sum_{j=1}^m  (f_j + h_j) q_j(X,s) \,, 
\end{aligned}\end{align}
where the $(f_j+h_j)$'s are treated as derivations acting on functions on $PM$.


\subsubsection*{Lyapunov exponents and moment Lyapunov exponents}

In this paper, we study the asymptotic statistics of the finite time Lyapunov exponents (FTLEs)
\begin{equation}
\label{eq:FTLE}
\lambda_t(X_0, s_0) = \frac{1}{t} \log \|Y_t\| \, .
\end{equation}
When the initial data $(X_0, s_0)$ is clear from context, we will simply write $\lambda_t$. 

One can describe the asymptotic statistics of the FTLEs via all typical probabilistic categories, including a law of large numbers governing the asymptotic limit 
\[\lambda(X_0, s_0) = \lim_{t \to \infty} \lambda_t(X_0, s_0),  \]
and a large deviations principle concerning the rate at which this limit is realized. 
The large deviations prinicple is expressible in terms of the \emph{moment Lyapunov exponents}
\begin{align}\label{eq:defnMLE2} \Lambda(p; (X, s)) =   \lim_{t \to \infty} \frac{1}{t} \log \mathbb{E}_{(X,s)} \| Y_t \|^p \, , \end{align}
when this limit exists, where here $p \in \mathbb{R}$ is a fixed parameter. 

Below, we will obtain $\Lambda(p)$ from spectral properties of the \emph{tilted semigroup} 
$S^t_p$ acting on bounded measurable observables $\psi: PM \to \mathbb{R}$ by
\begin{equation}
\label{eq:tilted_semigroup}
S^t_p \psi(X,s) = \mathbb{E}_{(X,s)} \left[ \psi(X_t, s_t) \exp \left\{p\left( \int_0^t  Q(X_\tau, s_\tau) \rmd \tau + \sum_{j=1}^m \int_0^t q_j(X_\tau, s_\tau)  \rmd W_\tau^j \right) \right\}  \right]  \,. 
\end{equation} 

The following is standard -- see, e.g., \cite{ArnoldKliemann1987}. 
\begin{thm}\label{thm:mlp}
Under Assumption \ref{ass:projStat}, the following holds: 

\begin{enumerate}[label=(\roman*)]

\item 
For all $(X_0, s_0) \in PM$, the limit  
\begin{align}\label{eq:asymptoticLE2} \lambda(X_0, s_0) = \lim_{t \to \infty} \frac{1}{t} \log  \|Y_t\| = \lambda \end{align}
exists with probability 1,
 and is almost-surely independent of $(X_0, s_0)$. The asymptotic \emph{Lyapunov exponent} $\lambda$ is given by the \emph{Furstenberg-Khasminskii formula}
\begin{equation}
\label{eq:FK_general}
\lambda = \int_{PM} Q(X,s)  \, \rmd \mu(X,s).
\end{equation}

\item  
For all $(X_0, s_0) \in PM$ and $p \in \mathbb{R}$, the limit  
\begin{align}\label{eq:convMLE2} \Lambda(p; X_0, s_0) = \lim_{t \to \infty} \frac{1}{t} \log \mathbb{E}_{(X_0, s_0)} \|Y_t\|^p 
\end{align}
exists with probability 1 and is independent of $(X_0, s_0)$. 

\item For all $t \geq 0$, it holds that $e^{t \Lambda(p)}$ is the simple dominant eigenvalue of $S^t_p$, treated as a bounded linear operator on the space of continuous functions on $PM$ with the uniform norm.


\item $\Lambda: \mathbb{R} \to \mathbb{R}$ is convex and analytic with $\Lambda(0) =0$ and $\Lambda'(0) = \lambda$.

\item The function $\gamma: \mathbb{R} \to \mathbb{R}$ defined by 
\begin{align}\label{eq:defineGamma}\gamma(p) = \begin{cases}
    \lambda & p = 0 \\ \frac{\Lambda(p)}{p} & p \neq 0 
\end{cases}\end{align}
is nondecreasing and analytic.

\end{enumerate}

\end{thm}
We set 
$$ \gamma_- = \lim_{p \to - \infty} \gamma(p), \quad \gamma_+ = \lim_{p \to \infty} \gamma(p).$$
Note that either $\gamma(p) \equiv \lambda$ for all $p$, or $\gamma$ is 
strictly increasing (as is typically the case) and satisfies 
$$ \gamma_- < \lambda = \gamma(0) < \gamma_+.$$
For additional comments, see Remark \ref{rmk:gammas} below.

\begin{proof}[Comments on the proof]
    Item (i) is a standard consequence of the Birkhoff ergodic theorem. Items (ii), (iii) are the majority of the work, as (iv) and (v) follow from this -- see \cite{arnold1984formula} for details. 
    
    For (ii), (iii), 
    set ${\bf 1} : PM \to \R$ to be the function identically equal to 1. According to~\eqref{eq:linear_growth}, it holds that \[ S^t_p {\bf 1}(X,s) = \mathbb{E}_{(X,s)} \| Y_t\|^p, \] 
where we recall that $s = Y_0/ \|Y_0\|\in P_x M$.
One needs to prove that $S^t_p$ admits a simple dominant eigenvalue $e^{t \Lambda(p)}$ and a spectral gap. This can be shown by using that, for $t > 0$ arbitrary, $S^t_p$ is \emph{uniformly positive}: for any $\varphi \geq 0$ continuous, $\varphi$ not identically $0$, $S^t_p \varphi \geq c > 0$ pointwise for some $c = c(\varphi)$. The uniform positivity of $S^t_p$ can be checked using Assumption \ref{ass:projStat}(b) and compactness of $M$. From here, \cite[Theorem 3]{birkhoff1957extensions} yields that $S^t_p$ admits a simple dominant eigenvalue $e^{t \Lambda(p)}$ with spectral gap, where the associated eigenpair $(\psi_p, \nu_p)$ is such that $\psi_p : PM \to \R_{> 0}$ is continuous and $\nu_p$ is a probability measure\footnote{Note that at $p = 0$, $\nu_0$ is the unique stationary measure for the $(X_t, s_t)$ process on $PM$, and $\psi_0$ is the constant function identically equal to 1.} on $PM$ with $\int \psi_p d \nu_p = 1$. \end{proof}

We now turn to large deviations principles, which can be read off of the \emph{rate function} $\Ic : \mathbb{R} \to \mathbb{R}_{\geq 0} \cup \{ \infty\}$ defined as the Legendre-Fenchel transform of the moment Lyapunov exponents
\begin{align} \label{eq:legendre-fenichel2} \Ic(r) = \sup_{p \in \mathbb{R}} \left(r p - \Lambda(p) \right) \, . \end{align}
The following basic properties are straightforward from Theorem \ref{thm:mlp} (iv), (v) and standard facts about Legendre transforms: if $\gamma_- < \gamma_+$, 
\begin{enumerate}
\item $\Ic(r)$ is finite if $r \in (\gamma_-, \gamma_+)$ and infinite if $ r \notin [\gamma_-, \gamma_+]$,
\item $\Ic$ is strictly convex, analytic, and nonnegative on $(\gamma_-, \gamma_+)$,
\item $\Ic(r) = 0$ iff $r=\lambda$, and $\Ic'(\lambda) = 0$. In particular, $\Ic(r)$ has its strict global minimum of $0$ at $r = \lambda$. 
\item $\Ic$ is strictly decreasing on $(\gamma_-, \lambda)$ and strictly increasing on $(\lambda, \gamma_+)$.
\end{enumerate}

With this in place, we are now ready to state the large deviations principle. 
\begin{thm} Suppose that $M$ is compact, the SDE \eqref{eq:SDE_gen} on $M$ satisfies Assumption \ref{ass:projStat}, and lastly that $\gamma_- < \gamma_+$ as in Theorem \ref{thm:mlp}.
\label{thm:LDPs}
\begin{enumerate}[label=(\roman*)]

\item 
We have the following: 
\begin{align*}
\limsup_{t \to \infty} \sup_{(X_0, s_0) \in PM} \frac{1}{t} \log \mathbb{P}_{(X_0, s_0)} \left( r_- \leq \lambda_t  \leq r_+ \right) &\leq - \inf_{\tilde r \in [r_-, r_+]}  \Ic(\tilde r), \\
\liminf_{t \to \infty} \inf_{(X_0, s_0) \in PM} \frac{1}{t} \log \mathbb{P}_{(X_0, s_0)} \left( r_- < \lambda_t  < r_+ \right) &\geq - \inf_{ \tilde r \in (r_-, r_+)}  \Ic(\tilde r) .
\end{align*} 

\item For every  $(X_0, s_0) \in PM$ and  $\varepsilon>0$ it holds that
\[ \lim\limits_{t\to\infty} \frac{1}{t} \log \P_{(X_0, s_0)}(|\lambda_t-\lambda|\geq \varepsilon ) =c(\varepsilon), \]
where $c(\varepsilon)=-\min \{\Ic(\lambda+\varepsilon), \Ic(\lambda-\varepsilon) 
\}<0$.
\end{enumerate}
\end{thm}
Theorem \ref{thm:LDPs} is a consequence of the G\"artner-Ellis Theorem (Theorem \ref{thm:ge}), which roughly-speaking implies a large deviations principle assuming convergence of the limit \eqref{eq:defnMLE2} to a differentiable function -- see Appendix~\ref{ldp} for further discussion. 

Below we record a mild strengthening of Assumption \ref{ass:projStat} allowing to ignore the distinction between open and closed intervals in Theorem \ref{thm:LDPs}. 
We refer below to the following condition on the process $(X_t, Y_t)$. 

\begin{ass}\label{ass:nondegenFTLE2}
    For any $t > 0$ fixed and $(X_0, s_0) \in PM$, it holds that 
    \[\P_{(X_0, s_0)} ( \lambda_t = c) = 0 \]
    for all $c \in \R$. 
\end{ass}
Assumption \ref{ass:nondegenFTLE2} holds, e.g., if the variational process $(X_t, Y_t)$ itself satisfies the parabolic H\"ormander condition (c.f. Assumption \ref{ass:statDensity}(a)). For further discussion of Assumption \ref{ass:nondegenFTLE2}, see Section \ref{subsubsec:conditionG}. 

The following is immediate.

\begin{cor}
\label{cor:LDP}
    In addition to the setting of Theorem \ref{thm:LDPs}, let Assumption \ref{ass:nondegenFTLE2} hold. Then, one has that 
    \[\lim_{t \to \infty} \frac1t \log \mathbb{P}_{(X_0, s_0)}( r_- < \lambda_t < r_+) = -\inf_{\tilde r \in [r_-, r_+]} \Ic(\tilde r) \,, \] 
    where the limit is uniform over $(X_0, s_0) \in PM$. 
    In particular, for $r > \lambda$ and for all $(X_0, s_0) \in PM$, it holds that 
    \begin{align}\label{eq:simplerLDPFormula}
        \Ic(r) = -\lim_{t \to \infty} \frac1t \log \mathbb{P}_{(X_0, s_0)} (\lambda_t > r) \,, 
    \end{align}
\end{cor}

\begin{rmk}\label{rmk:gammas} 
    \begin{itemize}
        \item[(a)] We note that $\Lambda''(0)$ has been shown to coincide with the \emph{asymptotic variance} appearing in the central limit theorem governing convergence in distribution of $\sqrt{t} (\lambda_t -  \lambda)$. For further details, see \cite{ArnoldKliemann1987}. 
        
        \item[(b)] The scenario $\gamma_- = \gamma_+$ implies immediately that $\Lambda(p) \equiv  p \lambda$ \cite{ArnoldKliemann1987}. This behavior is quite degenerate and straightforward to rule out in many practical cases. For instance, our computer-assisted enclosures of $p \mapsto \Lambda(p)$ preclude linear dependence on $p$, implying $\gamma_- < \gamma_+$ immediately.


        \item[(c)] It is also typically the case that $\gamma_\pm = \pm \infty$, and indeed that $\Lambda(p)$ scales like $p^2$ as $|p| \to \infty$. For further discussion, see Section \ref{subsubsec:infIsom3}. 
        

        \item[(d)] 
        One may expect that $\Lambda(p)$ can be obtained as the principle eigenvalue of the infinitesimal generator $L_p$, which in our setting should take the form 
        \begin{equation}
            \label{eq:tilted_generator_general_prelim}
            L_p = L + p \sum_{j=1}^m q_j h_j + p Q + \frac{p^2}{2} \sum_{j=1}^m q_j^2 \,. 
        \end{equation}
        Doing so will form the basis of our computer-assisted enclosures of $\Lambda(p)$ presented in Sections \ref{sec:test_problems} and \ref{sec:CAP}, to which we refer the reader for further details on this point. 
        

    \end{itemize}

\end{rmk}




\subsection{The noncompact case} \label{subsec:noncompact2}




Here, we specialize to the case 
\begin{gather*} M = \R^d \, , \qquad \text{ and } \\ 
\text{the vector fields $f_i, i\geq 1,$ are constant vectors in $\mathbb{R}^d$} \, , 
\end{gather*} 
and will refer here and below to the $d \times m$ matrix $\sigma = (f_1 f_2 \dots f_m)$ with columns $f_j, j \geq 1$. Note that in this case, we identify $PM$ with the Cartesian product $\R^d \times P^{d-1}$. 


In this setting, further assumptions will be required to ensure that the equations defining $(X_t)$ and $(X_t, Y_t)$ admit global-in-time strong solutions and stationary measures. Both issues are addressed by the following assumption. 


\begin{ass}\label{ass:drift}
    There is a smooth function $W : \R^d \to [1,\infty)$ with compact sublevel sets such that $G$ defined in~\eqref{eq:G} satisfies the \emph{drift condition} 
    \begin{align}\label{eq:driftCond}
        G W \leq - \alpha W + \beta ,
    \end{align}
    for some constants $\alpha, \beta > 0$. 
\end{ass}
It is straightforward to check that Assumption \ref{ass:drift} ensures unique existence of strong solutions to \eqref{eq:SDE_gen} for $(X_t)$ and \eqref{eq:var_process} for $(Y_t)$, hence also \eqref{eq:projective_process} for $(s_t)$. 
Regarding stationary measures: it is standard that Assumptions \ref{ass:statDensity} and \ref{ass:drift} imply unique existence of a stationary measure $\nu$ for $(X_t)$ on $\R^d$ with smooth, everywhere-positive density. 
Since the $P^{d-1}$ component is compact, Assumptions \ref{ass:projStat} and \ref{ass:drift} similarly imply unique existence of a stationary measure $\mu$ for $(X_t, s_t)$ on $\R^d \times P^{d-1}$ with smooth, everywhere-positive density. For further details, see, e.g., \cite{hairer2011yet}.

\parahead{Large deviations}

    As described in Section \ref{subsec:compact2}, to prove a large deviations estimate for finite-time Lyapunov exponents it is enough to establish a spectral theory as in Theorem \ref{thm:mlp} for the tilted semigroup $(S^t_p)$, $p \in \R$, see \eqref{eq:tilted_semigroup}. 
    The main challenge is that $S^t_p$ is not a bounded operator on the space of bounded continuous functions, since frequently in applications the exponent in \eqref{eq:linear_growth} is unbounded at infinity.
    
    Instead, an appropriate functional framework is to treat $S^t_p$ as an operator on $C_W = C_W(PM)$, the space of continuous functions with finite weighted norm 
    \[\| \varphi\|_{C_W} = \sup_{(X, s) \in PM} \frac{|\varphi(X, s)|}{W(X)},  \]
    for some appropriate function $W$  to be chosen shortly. 
    A natural choice would be the function $W$ appearing in Assumption \ref{ass:drift}; however, the drift condition \eqref{eq:driftCond} is not enough to ensure $S^t_p$ is bounded on $C_W$: some stronger condition is required. 
    This problem has been treated in some special cases  (e.g., \cite{bedrossian2022almost}); general frameworks addressing this include \cite{FStoltz20, kontoyiannis2005large}.

The following framework is a version of the results of \cite{FStoltz20} adapted to observables of the projective process $(X_t, s_t)$ on $\R^d \times P^{d-1}$.

\smallskip 

\noindent {\bf Notation.} Here and below, for nonnegative functions $f, g$ on $\mathbb{R}^d$  with $f > 0$, we write $g \ll f$ if $f / g$ has compact sublevel sets. Note that if $g$ is continuous, then $g \ll f$ implies that for any $a > 0$ there exists $R > 0$ such that 
\[g \leq a f + R \,. \]

 Moreover, for real-valued (and possibly sign-indefinite) functions $f, g$, we will write $f \sim g$ if there exist constants 
$c, c', R, R' > 0$ such that 
\[c' g - R' \leq f \leq c g + R \qquad \text{pointwise on $\mathbb{R}^d$.}\]
 
\begin{ass}\label{ass:noncompactLDP2} \ 
    \begin{itemize}
        \item[(1)] There exists a function $V$ belonging to the Schwartz class\footnote{Recall that $f$ is of Schwartz class if it is smooth and it and its partial derivatives have strictly polynomial growth at infinity. 
        } such that:
        \begin{itemize}
            \item[(a)] $V$ has compact sublevel sets;
            \item[(b)] $|\sigma^T \nabla V|$ has compact sublevel sets; and
            \item[(c)] for all $\eta \in (0,1)$, it holds that
           \[G V + \frac{\eta}{2} |\sigma^T \nabla V|^2 \sim - |\sigma^T \nabla V|^2 \,. \] 
        \end{itemize}
        \item[(2)]  There exists $\theta \in (0,1)$ such that the function $x \mapsto \| \rmD f_0 (x) \|$ (where the norm refers to the operator norm for $\rmD f_{0}(x): \mathbb{R}^{d} \rightarrow \mathbb{R}^{d}$) 
        satisfies the bounds 
        \begin{gather}
                \label{eq:boundingDf02} \| \rmD f_0 \|_{C_W} < \infty \, ,  \\ 
                \label{eq:strongerBoundingDf02} \| \rmD f_0\| \ll | \sigma^T \nabla V|^2 \,, 
        \end{gather}
        where 
        \[W = e^{\theta V} \, . \] 
    \end{itemize}
\end{ass}





\begin{thm}\label{thm:noncompactLDP}
    In the setting of this section (additive noise, $M = \R^d$) let Assumptions \ref{ass:projStat} and \ref{ass:noncompactLDP2} hold. Then: 
    \begin{itemize}
    \item[(a)] 
    \begin{itemize}
        \item[(i)] The limit \eqref{eq:asymptoticLE2} definining the asymptotic Lyapunov exponent $\lambda$ exists with probability 1 and is constant over $(X_0, s_0)$. Moreover, the Furstenberg-Khasminskii formula \eqref{eq:FK_general} holds (in particular, $Q \in L^1(\mu)$ where $ Q(x, s) = \langle x, \rmD f_0(x) s \rangle$. 
        \item[(ii)] For all $p \in \mathbb{R}$, the limit \eqref{eq:convMLE2} exists with probability 1 and is independent of $(X_0, s_0)$. 
        \item[(iii)] For all $t > 0$ and $p \in \mathbb{R}$, the operators $S^t_p$ are bounded in the $C_W$-norm,  where $W$ is as in Assumption \ref{ass:noncompactLDP2}(2). Moreover, for all $t > 0$ the operator $S^t_p : C_W \to C_W$ admits a simple, isolated, dominant eigenvalue $e^{t \Lambda(p)}$. 
        \item[(iv)] $\Lambda : \R \to \R$ is convex, analytic, and satisfies $\Lambda(0) = 0, \Lambda'(0) = \lambda$. 
        \item[(v)] The function $\gamma$ as in \eqref{eq:defineGamma} is nondecreasing and analytic. 
    \end{itemize}
    \item[(b)] If $\gamma_- < \gamma_+$, then the conclusions of Theorem \ref{thm:LDPs} hold\footnote{To be precise: the convergences listed in (i) and (ii) of Theorem \ref{thm:LDPs} are uniform over compact subsets of $PM$.} with $\Ic$ the Legendre-Fenchel transform of $\Lambda$ as in \eqref{eq:legendre-fenichel2}.
    \end{itemize}
\end{thm}

\begin{rmk}
    It is straightforward to check that Assumption \ref{ass:noncompactLDP2} (1) implies Assumption \ref{ass:drift} for $W = e^{\eta V}$ for any $\eta \in (0,1)$,  on computing \[\frac{GW}{W} = \eta \left( G V + \frac{\eta}{2} |\sigma^T \nabla V|^2 \right) \sim - \eta |\sigma^T \nabla V|^2 \,. \]  The bound \eqref{eq:boundingDf02} ensures that $Q$ is integrable with respect to $\mu$, while 
    \eqref{eq:strongerBoundingDf02} implies a sufficient level of control on $Q(X_t, s_t)$ so that $S^t_p : C_W \to C_W$ is bounded  for the particular choice of $\eta = \theta$, following results in \cite{FStoltz20}. See  Appendix \ref{ldp} for more details. 
\end{rmk}

\begin{rmk}
    While ellipticity of the generators $L, G$ is not explicitly assumed, assumption \ref{ass:noncompactLDP2}(1)(b) is strong and essentially forces the noise vectors $f_j$ to span $\R^d$. 
    The authors are unaware of a more general approach, even in this additive setting, that would allow for truly hypoelliptic noise. We refer the reader to Section \ref{sec:outlook} for further discussion on open questions in this area. 
\end{rmk}




\subsubsection*{Special case: gradient systems}

Below we consider the special case of gradient systems
\begin{align} \label{eq:genGradSystem2}\begin{aligned}
        \rmd X_t & = - \nabla V(X_t) \rmd t + \sigma \sum_{1}^d e_j \rmd W_t^j \, , \\ 
        \rmd Y_t & = - \rmD^2 V(X_t) Y_t \rmd t 
\end{aligned}
\end{align}
where the $(e_j)_{j = 1}^d$ form the standard basis for $\R^d$, and where the potential function $V : \mathbb{R}^d \to \mathbb{R}$ is smooth and satisfies the following: 

\begin{ass}\label{ass:gradient}  Assumption \ref{ass:projStat} holds, and additionally:
    \begin{itemize}
        \item[(a)] The function $V$ is of Schwartz class and has compact sublevel sets $\{ V \leq C\}$, $C \in \mathbb{R}$. 
        \item[(b)] It holds that 
        \begin{align}
            |\Delta V| \ll |\nabla V|^2 \qquad \text{ and } \qquad 
            \| \rmD^2 V \| \ll |\nabla V|^2 \,. 
        \end{align}
    \end{itemize}
\end{ass}

\begin{cor} \label{cor:gradient}
Consider the system \eqref{eq:genGradSystem2} and let Assumption \ref{ass:gradient} hold. Then: 
    \begin{itemize}
        \item[(a)] Equation \eqref{eq:genGradSystem2} for $(X_t, Y_t)$ admits unique strong solutions. The process $(X_t)$ admits the unique stationary density \[\rho = \frac{1}{Z} \exp \bigl(-2 V(x) / \sigma^2\bigr)\] and the corresponding $(X_t, s_t)$ process on $\mathbb{R}^d \times P^{d-1}$ admits a unique stationary measure $\mu$ with smooth, positive density. 
        \item[(b)] Assumption \ref{ass:noncompactLDP2} is satisfied for $V$ and with any value $\theta \in (0,1)$. In particular, if $\gamma_- < \gamma_+$ then the conclusions of Theorem \ref{thm:LDPs} hold with $\Ic$ the Legendre-Fenchel transform of $\Lambda$. 
    \end{itemize}

\end{cor}

\begin{rmk} Assumption \ref{ass:gradient}(b) is mild and holds automatically for a broad class of potentials with superlinear growth at infinity, e.g., $V(x) = c |x|^p + v(x)$ where $c > 0$, $p > 1$ is an integer, and $v : \mathbb{R}^d \to \mathbb{R}$ is of Schwartz class and satisfies $v / |x|^{p -\epsilon} \to 0$ as $|x| \to \infty$. 
\end{rmk}

\section{Random bifurcations via large deviation principles}
  \label{sec:randBifurc}


We now turn to the definition and properties of finite-time Lyapunov exponent transitions. 

\subsubsection*{Settings}

In what follows, we will be working under one of two cases, (1) $M$ compact and (2) $M = \R^d$, the standing assumptions of which are listed below. 
Throughout, $f_j^\alpha, j = 0, \dots, m$ are smooth vector fields on $M$ depending smoothly on a parameter $\alpha \in I$, where $I \subset \R$ is an open interval, and we write $X_t^\alpha, Y_t^\alpha$ and $s_t^\alpha$ for the corresponding processes solving \eqref{eq:SDE_gen}, \eqref{eq:var_process} and \eqref{eq:projective_process}, respectively. 

\begin{itemize}
	\item[(1)] \emph{The case $M$ is compact.} Here, $M$ is a fixed compact Riemannian manifold. We impose Assumption \ref{ass:projStat} for all $\alpha \in I$, and with $\gamma_\pm = \gamma_\pm^\alpha$ as in Theorem \ref{thm:mlp}, we will assume that $\gamma_-^\alpha < \gamma_+^\alpha$ for all $\alpha \in I$. Finally, we will impose that 
	\begin{align} \label{eq:forceByIsom3} q_j^\alpha \equiv 0 \, , \ j \geq 1\,,
	 \end{align}
	where $q_j^\alpha(X, s) = \langle s, \rmD f_j^\alpha (X) s \rangle$ as in \eqref{eq:defnQ2} (c.f. Remark \ref{rmk:forceIsom3} below).
	\item[(2)] \emph{The case $M = \R^d$.} Here, the vector fields $f_j^\alpha$ are constant vectors. We impose Assumption \ref{ass:projStat}, and moreover, that Assumption \ref{ass:noncompactLDP2} holds with $\sigma^\alpha = (f_1^\alpha \cdots f_m^\alpha)$ and for some Schwartz function $V^\alpha$ for all $\alpha \in I$. Lastly, we impose that $\gamma_-^\alpha < \gamma_+^\alpha$ with $\gamma^\alpha_\pm$ as in Theorem \ref{thm:noncompactLDP}. 
\end{itemize}
At times, in either setting, we may additionally impose Assumption \ref{ass:nondegenFTLE2}.

{Note that in either case, we have that (i) the SDEs defining $(X_t^\alpha, Y_t^\alpha)$ and $(X_t^\alpha, s_t^\alpha)$ admit unique strong solutions for all time, and (ii) the asymptotic Lyapunov exponent $\lambda(\a) = \lim_t \lambda_t^\alpha(X_0, s_0)$ exists with probability 1 and is constant over fixed initial $(X_0, s_0) \in PM$. 
Moreover, in either case it holds that the FTLEs are given by 
\begin{gather*} 
	\lambda_t^\alpha(X_0, s_0) = \frac1t \log \| Y_t^\alpha\| = \frac1t \int_0^tQ^\alpha(X_\tau^\alpha, s_\tau^\alpha) \, \rmd \tau
	\, , \\  \text{ where } \quad Q^\alpha(X, s) = \langle s, D f_0^\alpha(X) s \rangle \,, 
\end{gather*}
	interpreting, as usual, $Y_0 = Y_0^\alpha$ to be an arbitrary unit vector representative for $s_0 = s_0^\alpha \in P_{X_0} M$. In either case, let $\mathcal{I}_\alpha$ denote the corresponding rate function. 
}

\begin{rmk}\label{rmk:forceIsom3}
	Condition \eqref{eq:forceByIsom3} says, essentially, that the noise vector fields $f_j^\alpha$ are infinitesimal isometries w.r.t. the Riemannian metric $g$ on $M$. This assumption is conspicuous, but necessary to witness interesting transitions in the statistics of FTLE. For more details, see Section \ref{subsubsec:infIsom3}. 
\end{rmk}


\subsection{FTLE transitions in terms of LDP}

Assume either of settings (1) or (2). 

\begin{defn}[FTLE transition]
	\label{def:random_bif}
	{
	We say that a \emph{finite-time Lyapunov exponent transition} occurs at $\a = \a_0 \in I$ if there exists $\epsilon > 0$ such that $(\alpha_0 - \epsilon , \alpha_0 + \epsilon) \subset I$ and 
	\begin{itemize}
		\item[(a)] for all $\a \in (\a_0 - \epsilon, \a_0)$, $(X_0, s_0) \in PM$ and $t \geq 0$, 
		\[\lambda_t^\alpha(X_0, s_0) < 0 \qquad \text{ with probability 1};\]
		 and
		\item[(b)] for all $\alpha \in (\a_0, \a_0 + \epsilon)$ 
		there is some $(X_0, s_0) \in PM$ and $t_0 > 0$ such that 
		\begin{equation*}
			\mathbb{P}(\lambda_{t_0}^\alpha(X_0,s_0) > 0 ) > 0 \,. 
		\end{equation*} 
	\end{itemize}
	}
	\end{defn}

The following is a characterization of the FTLE transition in terms of the large deviations principles laid out in Section \ref{sec:prelim}. 

\begin{thm}\label{thm:main}
{  Assume either of settings (1) or (2). 
\begin{itemize}
	\item[(a)] Suppose that an FTLE transition occurs at $\alpha_0$, and let $\epsilon > 0$ be as in Definition \ref{def:random_bif}. Then, $\Ic_\alpha(0) = \infty$ for $\alpha \in (\alpha_0 - \epsilon, \alpha_0)$ and $\Ic_\alpha(0) < \infty$ for $\alpha \in (\alpha_0, \alpha_0 + \epsilon)$. 
	\item[(b)] Let Assumption \ref{ass:nondegenFTLE2} hold. Suppose there is an interval $(\alpha_0 - \epsilon, \alpha_0 + \epsilon) \subset I, \epsilon > 0,$ such that $\Ic_\alpha(0) = \infty$ for $\alpha \in (\alpha_0 - \epsilon, \alpha_0)$ and $\Ic_\alpha(0) < \infty$ for $\alpha \in (\alpha_0, \alpha_0 +\epsilon)$. Then, an FTLE transition occurs at $\alpha_0$. 
\end{itemize}
}
\end{thm}




\begin{rmk} 
		We emphasize that there is no reason a priori to expect that the `soft', nonquantitative lower bound in Definition \ref{def:random_bif}(b) implies a quantitative, exponential-in-time lower bound. That this implication holds is a consequence of the Markov property and controllability as in Assumption \ref{ass:projStat}(b).


\end{rmk}



We now present another characterization of FTLE transitions useful for constructing and diagnosing particular examples. 
\begin{prop}\label{prop:LDPsignChange}
	Assume either of settings (1) or (2). 
	\begin{itemize}
		\item[(a)] Let Assumption \ref{ass:nondegenFTLE2} hold. Let $(\alpha_0 - \epsilon, \alpha_0 + \epsilon) \subset I, \epsilon > 0$ and assume $\{Q^\alpha > 0\}$ is empty for $\alpha \in (\alpha_0 - \epsilon, \alpha_0)$ and nonempty for $\alpha \in (\alpha_0, \alpha_0 + \epsilon)$. Then, an FTLE transition occurs at $\alpha_0$. 

		\item[(b)] Conversely, suppose that an FTLE transition occurs at $\alpha_0$, and let $\epsilon > 0$ be as in Definition \ref{def:random_bif}. Then, $\{Q^\alpha > 0\}$ is empty for $\alpha \in (\alpha_0 - \epsilon, \alpha_0)$ and nonempty for $\alpha \in (\alpha_0, \alpha_0 + \epsilon)$. 
	\end{itemize}
\end{prop}

\begin{exam}\label{exam:equilibrium3}
	Let $(\alpha_0 - \epsilon, \alpha_0 + \epsilon) \subset I, \epsilon > 0$ and suppose that the drift vector field $f_0^\alpha$ is such that 
	\begin{itemize}
		\item $Q^\alpha < 0$ for $\alpha \in (\alpha_0 - \epsilon, \alpha_0)$; and 
		\item for $\alpha \in (\alpha_0, \alpha_0 + \epsilon)$ there exists a \emph{linearly unstable} equilibrium $\mathfrak{p}(\alpha)$ for $f_0^\alpha$, i.e., \[\spec\big((Df_0^\alpha)_{\mathfrak{p}(\alpha)}\big) \cap \{z \in \mathbb{C} :  \Re(z) > 0\} \neq \emptyset \, , \] where $\Re(z)$ denotes the real part of $z \in \mathbb{C}$. 
	\end{itemize}
	Then, $Q^\alpha > 0$ along the unstable eigenspace of $Df_0(\mathfrak{p}(\alpha))$, and it follows that an FTLE transition occurs at $\alpha_0$. This is the case, for instance, when $f_0$ undergoes a  pitchfork bifurcation in $M = \R$ as in Section \ref{sec:test_problems}. 
\end{exam}

\begin{rmk}
	Observe that $Q^\alpha$ itself depends on the metric. 
	On the other hand, it is not hard to see that moment Lyapunov exponents, hence LDP rate functions, \emph{do not depend on the metric}. 
Hence, Proposition \ref{prop:LDPsignChange}  relates $\Ic_\alpha(0)$, a metric-independent quantity, to the metric-dependent set $\{ Q^\alpha > 0\}$. 
	Note, however, that we are in the special settings of (1) and (2), where the metric-dependent quantities $q_j^{\alpha}$, $j\geq 1$, are taken to be zero, with the effect of essentially fixing the metric.
\end{rmk}

\subsubsection{Conditions for Assumption \ref{ass:nondegenFTLE2}}\label{subsubsec:conditionG}

Before continuing, we give some comments on Assumption \ref{ass:nondegenFTLE2} in the settings laid out at the beginning of Section \ref{sec:randBifurc}. Below, the parameter $\alpha \in I$ is fixed, but since $\alpha$ plays no role we drop it from the notation. 

Below, $r_t = t \lambda_t(X_0, s_0) = \log \| Y_t\|$ is the solution to the auxiliary random ODE 
\[\rmd r_t = Q(X_t, s_t) \rmd t \, ,  \]
so that $(X_t, s_t, r_t)$ is the solution of an SDE with drift vector field 
\[\tilde f_0(X, s, r) = \begin{pmatrix}
	f_0(X) \\ h_0(X, s) \\ Q(X, s)
\end{pmatrix}\]
and diffusion vector fields 
\[\tilde f_j(X, s, r) = \begin{pmatrix}
	f_j(X, s) \\ h_j(X, s) \\ 0 
\end{pmatrix} \, , \]
all viewed on $PM \times \R$. 

\begin{lem}
	Impose either of settings (1) or (2). Assume moreover that $\{ \tilde f_j\}_{j = 0}^m$ satisfies the parabolic H\"ormander condition (c.f. Assumption \ref{ass:statDensity}). Then, Assumption \ref{ass:nondegenFTLE2} holds, i.e., for all $(X_0, s_0) \in PM$ and $t > 0, c \in \R$ and $\alpha \in I$, it holds that 
	\[\P(\lambda_t(X_0, s_0) = c) = 0  \,. \]
\end{lem}
This follows immediately from H\"ormander's Theorem applied to the transition kernels of the $(X_t, s_t, r_t)$ process (see e.g.~\cite{Hairer2011}), 
where we make use of the diffemorphic relation between the $(X_t, Y_t)$ coordinates and the $(X_t, s_t, r_t)$ coordinates for the variational process.

We caution that for a given sample path (i.e., for a fixed Brownian realization $(W_t^j)_{j = 1}^m$), it is possible that $\lambda_t(X_0, s_0) = 0$ along some set of random times $t$. 
Assumption \ref{ass:nondegenFTLE2} ensures that the likelihood of $\{\lambda_t(X_0, s_0) = 0\}$ at a \emph{given time} $t$ is zero. 

For diffusions on $\R$ with additive noise, the following simple sufficient condition can be used to check Assumption \ref{ass:nondegenFTLE2}. 
Note that in one dimension the $s_t$ coordinate is trivial, and so for Assumption \ref{ass:nondegenFTLE2} it is enough to have bracket spanning for the process $(X_t, r_t)$. 

\begin{lem}
	Let 
	\[\rmd X_t = f_0(X_t) \rmd t + \sigma \rmd W_t\]
	be a diffusion on $M = \T$ or $\R$ for some $\sigma > 0$. Assume that for all $X \in M$, there is some $k \in \N$ such that $f^{(k)}(X) \neq 0$, where $f^{(k)}$ is the $k$-th derivative of $f$. 
	  Then, the process $(X_t, r_t)$ satisfies the parabolic H\"ormander condition, and in particular, Assumption \ref{ass:nondegenFTLE2} holds. 
\end{lem}

\begin{proof}
	The relevant vector fields in the spanning condition for $(X_t, r_t)$ are
	\[V_0 = \begin{pmatrix}
		f_0 \\ f_0'
	\end{pmatrix} \, , \qquad V_1 = \begin{pmatrix}
		1 \\ 0 
	\end{pmatrix} \, . \]
	Since $V_1$ always spans the first coordinate, it suffices, for each $X \in \R$, to determine brackets of $V_0, V_1$ which have nonvanishing second coordinate (note that there is no dependence of these vector fields on the $r$ coordinate). 
	
	For this, 
	 observe that for all $\ell \geq 1$, the $\ell$-fold iterated bracket of $V_1$ with $V_0$ gives
	\[\underbrace{[V_1, [ V_1, \cdots , [V_1 , [V_1}_{\ell \text{ times}}, V_0]]  \cdots ]]  = \frac{d^\ell}{dX^{\ell}}\begin{pmatrix}
		f_0 \\ f_0'
	\end{pmatrix} = \begin{pmatrix}
		f_0^{(\ell)} \\ f_0^{(\ell + 1)}
	\end{pmatrix} \,. \]
	Thus, the parabolic H\"ormander condition holds for $(X_t, r_t)$ at a point $(X, r)$ if for some $k \in \N$ one has that $f_0^{(k)}(X)$ does not  vanish. 
\end{proof}

\subsubsection{Driving by infinitesimal isometries}\label{subsubsec:infIsom3}

 We conclude with a discussion of the role played by condition \eqref{eq:forceByIsom3} in the case when $M$ is compact. Put in other words, this condition requires that the vector fields $f_j^\alpha$ are infinitesimal generators of isometries on $M$. The following addresses what happens without this condition. 
	
	Since $\alpha$ plays no role, we drop the $\alpha$ dependence and focus on a single collection of vector fields $(f_j)_{0 \leq j \leq m}$. We will impose Assumption \ref{ass:projStat}, but we will not impose \eqref{eq:forceByIsom3}. Instead, we ask only for the mild strengthening 
	\begin{align}
		T_{(X, s)} PM = \operatorname{Lie}_{(X, s)}\left\{ \begin{pmatrix}
			f_1 \\ h_1 
		\end{pmatrix} , \dots, \begin{pmatrix}
			f_m \\ h_m 
		\end{pmatrix} \right\}
	\end{align}
	at all $(X, s) \in PM$. Let $\Lambda(p)$ be as in Theorem \ref{thm:mlp}. 

	\begin{thm}[{\cite[Theorem 2.15]{BaxendaleStroock88}}] \label{thm:baxRigidity3}
		The following are equivalent. 
		\begin{itemize}
			\item[(i)] $\lim_{p \to \infty} \frac{\Lambda(p)}{p^2} = 0$. 
			\item[(ii)] $\lim_{p \to -\infty} \frac{\Lambda(p)}{p^2} = 0$. 
			\item[(iii)] There exists a Riemannian metric $\tilde{g}$ on $M$ with respect to which the vector fields $f_j, 1 \leq j \leq m $ are all infinitesimal isometries. In particular, with respect to the metric $\tilde{g}$, it holds that 
			\[\tilde q_j(X, s) := \langle \rmD f_j(X) s, s \rangle_{\tilde g}\]
			is identically equal to zero. 
		\end{itemize}
	\end{thm}

	The following is immediate, c.f., \cite[Section 1]{BaxendaleStroock88}. 
	\begin{cor}\label{cor:isomGen3}
		Assume the setting of Theorem \ref{thm:baxRigidity3}. Suppose there is no Riemannian metric with respect to which all $f_j, 1 \leq j \leq m$ are infinitesimal isometries. Then, $\gamma_\pm = \pm \infty$. In particular, 
		\begin{itemize}
			\item[(a)] $\mathcal{I}(r) < \infty$ for all $r \in \R$; and 
			\item[(b)] There exists $T > 0$ such that for all $(X_0, s_0) \in PM$, one has 
			\[\P_{(X_0, s_0)}(\lambda_t > 0) > 0 \]
			for all $t \geq T$. 
		\end{itemize}
	\end{cor}

	Corollary \ref{cor:isomGen3} leads to the following dichotomy: 
	either the noise vector fields are all infinitesimal isometries with respect to some metric $\tilde{g}$, in which case we can set $g = \tilde{g}$ and proceed as in \eqref{eq:forceByIsom3}, or there is no such metric, in which case we should not expect FTLE transitions. In this paper, we have chosen to deal with this state of affairs by assuming $q_j \equiv 0$ from the start. 

\subsection{Proofs}

Throughout the proofs of Theorem \ref{thm:main} and Proposition \ref{prop:LDPsignChange} it will be enough to work with one $\alpha$ at a time. From here on, $\alpha$ is fixed and dependence on it will be dropped.

\subsubsection*{Proof of Theorem \ref{thm:main}}

For (a), it suffices to prove the following. For now, Assumption \ref{ass:nondegenFTLE2} is not used. 

\begin{itemize}
	\item[(I)] If $\P(\lambda_t(X_0, s_0) < 0) = 1$ for all $(X_0, s_0) \in PM$ and $t > 0$, then $\Ic(0) = \infty$. 
	\item[(II)] If $\P(\lambda_{t_0}(X_0, s_0) > 0) > 0$ for some $(X_0, s_0) \in PM$ and $t_0 > 0$, then $\Ic(0) < \infty$. 
\end{itemize}
(I) is evident from the definitions. The proof of (II) will deferred to the end. For Theorem \ref{thm:main}(b) it suffices to check the following. 

\begin{itemize}
	\item[(III)] $\Ic(0) < \infty$ implies that there exists $(X_0, s_0) \in PM$ and $t_0 > 0$ such that $\lambda_{t_0}(X_0, s_0) > 0$ with positive probability. 
	\item[(IV)] Let Assumption \ref{ass:nondegenFTLE2} hold. Then, $\Ic(0) = \infty$ implies that for all $(X_0, s_0) \in PM$ and $t > 0$ it holds that $\lambda_t(X_0, s_0) < 0$ with full probability. 
\end{itemize}

(III) is also evident from the definitions. Given (II), let us check that (IV) holds. For this, statement (II) is the contrapositive of the assertion that $\Ic(0) = \infty$ implies 
	\[\P(\lambda_t(X_0, s_0) \leq 0) = 1\]
for all $(X_0, s_0) \in PM$ and $t > 0$. If Assumption \ref{ass:nondegenFTLE2} holds, then $\P(\lambda_t(X_0, s_0) = 0) = 0$ for all $(X_0, s_0)$ and $t$, from which (IV) follows. 

It remains to prove statement (II). We will require two preliminary results. Below, for $ (X, s) \in PM$ we write ${\bf P}_{(X, s)}$ for the law of the process $(X_\bullet, s_\bullet) := (X_t, s_t)_{t \in [0,1]}$ in the path space $C([0,1], PM)$. In particular, observe that ${\bf P}_{(X, s)}$ is a probability measure on the set of $(X_\bullet, s_\bullet) \in C([0,1], PM)$ for which $(X_0, s_0) = (X, s)$. Throughout, the path space $C([0,1], PM)$ has the uniform topology. 

The following is a consequence of strong uniqueness and nonexplosion of the SDE defining $(X_t, s_t)$. 

\begin{lem}\label{lem:weakStarContSamplePath}
	Assume either of the settings (1) or (2). Then, it holds that $(X, s) \mapsto {\bf P}_{(X, s)}$ is weak-$*$ continuous. 
\end{lem}
A proof sketch is included in Appendix \ref{app:weakStarCty}. 

Given a set $S \subset PM$ let \[\underline{S} = \{ (X_\bullet, s_\bullet) \in C([0,1], PM) : (X_t, s_t) \in S \text{ for all } t \in [0,1]\} \,, \] and for $t \in [0,1]$ fixed let \[S_{(t)} = \{ (X_\bullet, s_\bullet) \in C([0,1], PM) : (X_t, s_t) \in S\} \,. \] Observe that if $S$ is open (resp. closed) then $\underline{S}$ and $S_{(t)}$ are open (resp. closed) in $C([0,1], PM)$. 

\begin{prop}
	Assume ${\bf P}_{(X, s)}$ is fully supported in $C([0,1], PM)$ for all $(X, s) \in PM$. Let $U \subset PM$ be an open set and $K \subset U$ a compact set with nonempty interior. Then, 
	\[\inf_{(X, s) \in K} {\bf P}_{(X, s)} (\underline{U} \cap K_{(1)}) > 0 \,. \]
\end{prop}
\begin{proof}
	To start, fix a nonempty open subset $U' \subset K$ and observe that $\underline{U} \cap U'_{(1)} \subset C([0,1], M)$ is open. Now, by Lemma \ref{lem:weakStarContSamplePath} it holds that $(X, s) \mapsto {\bf P}_{(X, s)}(\underline{U} \cap U'_{(1)})$ is lower semi-continuous. It is a standard fact that lower semicontinuous functions achieve their minima along compact sets, hence 
	\[\inf_{(X, s) \in K} {\bf P}_{(X, s)} (\underline{U} \cap U_{(1)}') = P_{z_*} (\underline{U} \cap U_{(1)}')\]
	for some $(X_*, s_*) \in K$. By the hypothesis, ${\bf P}_{(X_*, s_*)}$ is fully supported in $C([0,1], PM)$ such that 
	\[\inf_{(X, s) \in K} {\bf P}_{(X, s)}(\underline{U} \cap K_{(1)}) \geq \inf_{(X, s) \in K} {\bf P}_{(X, s)}(\underline{U} \cap U'_{(1)}) = {\bf P}_{(X_*, s_*)}(\underline{U} \cap U_{(1)}') > 0 \,, \]
	which completes the proof.
\end{proof}

We now turn to the proof of statement (II) above. In the proof below, we write $(\Omega, \Fc, \P)$ for the probability space on which our SDEs are posed. Without loss, it will suffice to work with 
\[\Omega = C([0,\infty) , \R^m) \, , \quad \Fc = \operatorname{Bor}(C([0,\infty) , \R^m)) \, , \]
where $C([0,\infty), \R^m)$ has the compact-open topology, and taking $\mathbb{P}$ as standard Wiener measure. With $(t, \omega) \mapsto W^j_t(\omega), 1 \leq j \leq m$ the time-$t$ coordinate functions on $\Omega$, we write $\Fc_{s, t} = \sigma(W^j_\tau - W^j_s : \tau \in [s,t], 1 \leq j \leq m)$ for the $\sigma$-algebra of events depending on increments between $s$ and $t$, and $\Fc_t := \Fc_{0, t}$.  Finally, we write $\theta^t : \Omega \to \Omega$ for the measure-preserving flow given by the time shift, i.e., 
\[
W_s^j(\theta^t \omega) = W_{s+t}^j(\omega) - W_t^j(\omega) \,.\]

\begin{proof}[Proof of (II)]
	Assertion (II) clearly implies $\{ Q > 0\}$ is nonempty. Fix $U$ open, $K$ compact, and $c > 0$ such that (i) $K \subset U \subset \{ Q \geq c\}$ and (ii) $K$ has nonempty interior. Set 
	\[c_K = \inf_{(X, s) \in K} {\bf  P}_{(X, s)} (\underline{U} \cap K_{(1)}) > 0 \,. \]
	For fixed initial $(X_0, s_0) \in K$, define 
	\[\Ec_{(X_0, s_0)} = \{ (X_t, s_t) \in U \text{ for all } t \in [0,1] \, , \text{ and } (X_1, s_1) \in K \} \subset \Omega \, , \]
	and note (i) $\Ec_{(X_0, s_0)} \in \Fc_{0,1}$ and (ii) it holds that $\P(\Ec_{(X_0, s_0)}) \geq c > 0$ for all $(X_0, s_0) \in K$. Now, for integer $n \geq 1$ it holds that 
	\[
		\{ \lambda_n(X_0, s_0) \geq c\}  \supset \Ec^{(n)}_{(X_0, s_0)},
	\] 
	where 
	\begin{align*} \Ec^{(n)}_{(X_0, s_0)} & := \bigcap_{m = 0}^{n-1} \theta^{-m} \Ec_{(X_m, s_m)}  \\ 
		& = \{ (X_t, s_t) \in U \text{ for all } t \in [0,n] \, , \text{ and } (X_m, s_m) \in K \text{ for all } m \in [1,n] \cap \Z \} \, . 
	\end{align*}
	Note that $\Ec^{(n)}_{(X_0, s_0)}$ is $\Fc_n$-measurable for all $(X_0, s_0)$. 

	Now, 
	\begin{align*}
		\P(\Ec^{(n)}_{(X_0, s_0)}) & = \E \left[ \E \left(  {\bf 1}_{\Ec^{(n-1)}_{(X_0, s_0)}} {\bf 1}_{\theta^{-{(n-1)}} \Ec^{(1)}_{(X_{n-1}, s_{n-1})}} \bigg| \Fc_{n-1} \right)  \right] \\ 
		& = \E \left[ {\bf 1}_{\Ec^{(n-1)}_{(X_0, s_0)}} \E \left(   {\bf 1}_{\theta^{-{(n-1)}} \Ec^{(1)}_{(X_{n-1}, s_{n-1})}} \bigg| \Fc_{n-1} \right)  \right] \\ 
		& = \E \left[ {\bf 1}_{\Ec^{(n-1)}_{(X_0, s_0)}} {\bf P}_{(X_{n-1}, s_{n-1})} (\underline{U} \cap K_{(1)})  \right] \\ 
		& \geq  c_K \P(\Ec^{(n-1)}_{(X_0, s_0)})
	\end{align*}
	where above ${\bf 1}_{E}$ denotes the indicator function of a set $E \subset \Omega$, and we have used the Markov property in passing from the second to the third lines. Bootstrapping, we conclude 
	\[\P \{ \lambda_n(X_0, s_0) \geq c\} \geq c_K^n \, . \]
	From this and the characterization in Theorem \ref{thm:LDPs}(a) (or Theorem \ref{thm:noncompactLDP}(b)), it follows that $\Ic(0) \leq - \log c_K < \infty$. 
\end{proof}

\subsubsection*{Proof of Proposition \ref{prop:LDPsignChange}}

To prove Proposition \ref{prop:LDPsignChange}(a), we need to show the following two points under Assumption \ref{ass:nondegenFTLE2}. 

\begin{itemize}
	\item[(V)] $\{ Q > 0\}$ empty implies $\P(\lambda_t(X_0, s_0) < 0) = 1$ for all $(X_0, s_0) \in PM$ and for all $t > 0$. 
	\item[(VI)] $\{ Q > 0\}$ nonempty implies $\P(\lambda_t(X_0, s_0) > 0) > 0$ for some $(X_0, s_0) \in PM$ and for some $t > 0$.  
\end{itemize}

It is evident that $\{ Q > 0\}$ empty implies $\P(\lambda_t(X_0, s_0) \leq 0) = 1$ for all $t > 0, (X_0, s_0) \in PM$, and point (V) follows on imposing condition (G) to rule out zero FTLE at the fixed time $t$. For (VI), fix $(X_0, s_0) \in \{ Q > 0\}$, and observe that positivity of $\lambda_t(X_0, s_0)$ follows on taking $t$ sufficiently small so that $(X_\tau, s_\tau) \in \{ Q > 0\}$ for all $\tau \in [0,t]$. From this, it is clear that $\lambda_t(X_0, s_0) > 0$ has positive probability for all $t$ sufficiently small. Note that (VI) did not use Assumption \ref{ass:nondegenFTLE2}. 

For Proposition \ref{prop:LDPsignChange}(b), we will check that 

\begin{itemize}
	\item[(VII)] if $\P(\lambda_t(X_0, s_0) < 0) = 1$ for all $(X_0, s_0) \in PM$ and for all $t > 0$, then $\{ Q > 0\}$ is empty; and 
	\item[(VIII)] if $\P(\lambda_t(X_0, s_0) > 0) > 0$ for some $t > 0, (X_0, s_0) \in PM$, then $\{ Q > 0\}$ is nonempty. 
\end{itemize}

 To start, the contrapositive of (VI) states that if $\P(\lambda_t(X_0, s_0) \leq 0) = 1$ for all $t > 0$, then $\{ Q > 0\}$ is empty. This statement is clearly stronger than (VII), hence (VII) follows. Finally, we note that statement (VIII) is evident from the integral form of $\lambda_t$.

\section{Two worked examples}
  \label{sec:test_problems}

We pivot in this and the next section to a treatment with computer-assisted proof of the large deviations rate functions for finite-time Lyapunov exponents associated to two concrete models: in Section \ref{sec:pitchfork} the pitchfork bifurcation discussed in Section \ref{sec:intro} and in Section \ref{sec:ToyModel} a 2d linear SDE serving as a toy model of shear-induced chaos. Considerations involving computer-assisted work are deferred to Section \ref{sec:CAP}. 


\subsection{Pitchfork bifurcation with additive noise}
\label{sec:pitchfork}


Our first test problem is an additively-forced version of the standard form of a pitchfork bifurcation in one dimension: 
\begin{gather}
	\label{eq:pitchfork}
	\begin{gathered}
		\rmd X_t = - U'(X_t) \rmd t + \sigma \rmd W_t, \ X_0 \in \mathbb{R} \, , \\ 
		\text{ where } U(x) := \frac14 x^4 - \frac12 \alpha x^2 \, ,  
	\end{gathered}
\end{gather}
for $\alpha \in \mathbb{R}$ and $\sigma >0$. As shown in \cite{Callawayetal}, this family of SDE admits an FTLE transition (Definition \ref{def:random_bif}). 
Our aim is to quantify this in terms of large deviations principles (see Corollary~\ref{cor:LDP}) as stated in Theorem~\ref{thm:main}. Note that this theorem applies here as Assumption~\ref{ass:gradient} and, by that, Corollary~\ref{cor:gradient} are clearly satisfied. As explained in Section~\ref{sec:prelim}, the rate function $\mathcal{I}_\alpha$ appearing in the large deviations principles is obtained via the moment Lyapunov exponent $\Lambda_\alpha(p)$ obtained from the spectral theory of the twisted semigroups $S^t_{p}$ on $C_W$, where $W = e^{\theta U}$ and $\theta \in (0,1)$ is arbitrary.

Our computer-assisted estimate of $\Lambda(p)$ in the next section will rely on identifying it as the dominant eigenvalue of the tilted generator $L_p$ (see \eqref{eq:tilted_generator_general_prelim}), which in this case can be written as 
\begin{equation}
	\label{eq:generator_pitchfork}
	L_{p} g(x) = \frac{\sigma^2}{2} \partial_{xx} g(x) + (\alpha x - x^3) \partial_x g(x) +   p (\alpha - 3x^2) g(x) \,. 
\end{equation} 

To start, we observe that $L_p$ is (formally) self-adjoint as an operator on $L^2(\rho)$, where $\rho = Z^{-1} e^{- 2 U / \sigma^2}$ is the stationary density for \eqref{eq:pitchfork}  and satisfies $L^*\rho = 0$. Here, $L^2(\rho)$ is the weighted $L^2$ space with inner product
\[\langle f, g \rangle_\rho = \int \overline f g ~ \rho \, \rmd x \, . \] 

For computational reasons it will be more convenient to work with a flat $L^2$ space. To this end, the isometry $L^2(\rho) \to L^2$ given by
\[ g \mapsto \rho^{1/2} g \]
conjugates
$L_p$ to an operator $\tilde H_p$ of Schr\"odinger type on $L^2$ with the same spectrum as $L_p$:
\begin{equation}
	\label{eq:Schroedinger_pitchfork}
	\tilde H_p f(x) = \frac{\sigma^2}{2} \partial_{xx} f(x) - V_p(x) f(x) \, , 
\end{equation}
where 
\[
V_p(x):= \left(  \frac{(x^3 - \alpha x)^2}{2 \sigma^2} - (3 x^2 - \alpha) \left( \frac{1}{2} - p \right) \right) \,. 
\]
For convenience and to conform to standard notation, we primarily work with $H_p = - \tilde H_p$, especially in Section \ref{sec:CAP}. 
Throughout, we view $\tilde H_p$ and $H_p$ as unbounded operators on flat $L^2$ with domain $\mathcal D (\tilde H_p) \supset C_c^\infty(\R)$. Since the potential $V_p$ blows up at infinity, it is standard that $H_p$ (hence $\tilde H_p$ and $L_p$) is (i) essentially self-adjoint\footnote{We recall that an unbounded, densely defined and symmetric operator is said to be essentially
self-adjoint if its closure is self-adjoint.} and (ii) has purely discrete spectrum \cite{BerShu12}.

Thus, our method in the next section will be to estimate the dominant eigenvalue of $\tilde H_p$ viewed as an operator on $L^2$. However, it is not clear a priori that the spectrum of $L_p$ on $L^2(\rho)$ (hence that of $\tilde H_p$ on $L^2$) has anything to do with the twisted semigroup $(S^t_p)$ acting on $C_W$ as
\begin{equation*}
S^t_p f(x) = \mathbb{E}_{x} \left[ f(X_t) \exp \left\{p\left( \int_0^t  (\alpha - 3 X_\tau^2) \rmd \tau \right)\right\}  \right]  \,. 
\end{equation*} 
We address this discrepancy with the following lemma.
	
\begin{lem} \label{lem:spectraCoincide4}
	The dominant eigenvalue of $L_p$ on $L^2(\rho)$ coincides with the moment Lyapunov exponent $\Lambda(p)$.  
\end{lem}

Before giving the proof of Lemma \ref{lem:spectraCoincide4}, we observe that $C_W \subset L^2(\rho)$ for $\theta$ small enough:
let us fix $\theta < \frac{1}{\sigma^2}$, so that for all $f \in C_W$ it holds that 
    \begin{equation}
    \label{eq:CWinL2}
     \|f\|^2_{L^2(\rho)} \leq C \sup_{x\in\R} \Big(\frac{f(x)}{W(x)}\Big)^2 \int_\R e^{U(x) (2\theta-\frac{2}{\sigma^2}) } ~ \rmd x \leq C' \| f\|_{C_W}^2 \,, 
     \end{equation}
	where $C, C' > 0$ are constants. Thus, it holds that $C_W \subset L^2(\rho)$ and that $\| \cdot\|_{L^2(\rho)} \leq C'' \| \cdot\|_{C_W}$, where $C'' = \sqrt{C'}$.
	
	Furthermore, for the proof of Lemma~\ref{lem:spectraCoincide4}, we	record the following preliminary lemmas.

\begin{lem}
	\label{lem:semigroupToGenerator_pitchfork}
		The tilted generator\footnote{Technically, it is the closure of $L_p$ that generates a $C^0$ semigroup. Note that $L_p$ automatically admits a unique closure since it is essentially self-adjoint and densely defined, c.f. \cite[Proposition X.1.6]{conway2019course}.} $L_p$~\eqref{eq:generator_pitchfork}, 
		treated as an operator on $L^2(\rho)$,
		generates a $C_0$-semigroup $\tilde S_p^t$ on $L^2(\rho)$.
\end{lem}
\begin{proof}
	It suffices to prove that $\tilde{H}_p$ as in~\eqref{eq:Schroedinger_pitchfork} generates a $C^0$ semigroup on the flat $L^2$ space.  
	To start, we observe that there exists a constant $\tilde{c}>0$ such that $\tilde{H}_p + \tilde{c}\text{Id}$ is \emph{bounded from above} in the sense that 
	\[\langle f, (\tilde H_p + \tilde{c}) f \rangle \leq a \| f \|_{L^2}^2\]
	for some uniform constant $a > 0$ and for all $f \in \mathcal D (\tilde{H}_p)$.  Since $\tilde{H}_p$ is also essentially self-adjoint, it follows from \cite[Example 3.26]{EngelNagel} that $\tilde H_p$ generates a strongly-continuous semigroup on $L^2$. 
\end{proof}
	
	\begin{lem}\label{coincide}
		Assume that $\theta < \frac{1}{\sigma^2}$ such that $C_W \subset L^2(\rho)$. 
		Then the semigroups  $(S^t_p)_{t\geq 0}$ and $(\tilde{S}^t_p)_{t\geq 0}$ coincide on $C_W$. 
	\end{lem}
	\begin{proof}
		The proof relies on the following steps.
		\begin{enumerate}
			\item First of all we establish that $(S^t_p)_{t\geq 0}$ and $(\tilde{S}^t_p)_{t\geq 0}$ coincide along \(C^\infty_c\), i.e.
			\[
			S^t_p \varphi (x) = \tilde{S}^t_p \varphi(x) \quad \text{for Leb. a.e. } x \in \mathbb{R}.
			\]
			This follows directly from the Feynman-Kac formula~\cite[Theorem 8.2.1]{Oksendal}, applied to $L_p$ as a backward Kolmogorov operator.
			Note that there  exists, in particular, a unique continuous extension of $\tilde S^t_p \varphi$ (which is a priori only defined in $L^2(\rho)$, i.e.~almost everywhere).
			\item The next step is to check that $(S^t_p)_{t\geq 0}$ and $(\tilde{S}^t_p)_{t\geq 0}$ coincide on $C^0_c(\R)$.~To this aim we fix $\varphi\in C^0_c(\R)$ and convolve it with a smooth mollifier, i.e. $\varphi_\varepsilon:=\varphi\star \eta_\varepsilon$, where $\eta$ is a smooth density and $\eta_\varepsilon=\varepsilon^{-1}\eta(\varepsilon^{-1}x)$. Since $\varphi$ is uniformly continuous, we have that for all $\varepsilon>0$ sufficiently small $|\varphi_\varepsilon|\leq |\varphi|+1$. Consequently, by the dominated convergence theorem we obtain that $S^t_p \varphi_\varepsilon\to S^t_p\varphi$ pointwise. Moreover, by the continuity of the translation operators in $L^2(\rho)$ it holds that $\varphi_\varepsilon\to \varphi$ in $L^2(\rho)$ implying that $\tilde{S}^t_p\varphi_\varepsilon \to \tilde{S}^t_p\varphi$ in $L^2(\rho)$. 
The convergence $\tilde{S}^t_p\varphi_\varepsilon \to \tilde{S}^t_p\varphi$ in $L^2(\rho)$ implies by Chebyshev's inequality the convergence in probability of $\tilde{S}^t_p\varphi_\varepsilon \to \tilde{S}^t_p\varphi$ (with respect to $\rho(x)\rmd x$). Therefore there exists a sequence $(\varepsilon_k)_{k\geq 0}$ such that $\tilde{S}^t_p\varphi_{\varepsilon_k}(x) \to \tilde{S}^t_p\varphi(x)$ holds Lebesgue almost surely as $\varepsilon_k\to 0$.~In conclusion $\tilde{S}^t_p\varphi=S^t_p\varphi$ Lebesgue almost surely for all $\varphi\in C^0_c(\R)$. 
			\item We now proceed to show that $S^t_p \varphi = \tilde{S}^t_p \varphi$ for all $\varphi \in C_W$. To this aim we employ a truncation argument. For a function $\varphi \in C_W$, we let $\varphi_n := \chi_n \varphi$, where $(\chi_n)$ is a sequence of smooth functions with $0 \leq \chi_n \leq 1$ such that 
			\[\chi_n (x) = \begin{cases}
			1 & x \in [-n, n] \\ 
			0 & x \notin [-(n+1), n+1].
			\end{cases}\]
			Obviously $\chi_n \to 1$ pointwise on $\mathbb{R}$ as $n\to\infty$. As in step 2, we obtain by the dominated convergence theorem  that $S^t_p \varphi_n \to S^t_p \varphi$ pointwise. Since $\varphi \in L^2(\rho)$, we have, again by dominated convergence, that $\varphi_n \to \varphi$ in $L^2(\rho)$ which implies that $\tilde{S}^t_p \varphi_n \to \tilde{S}^t_p \varphi$  in $L^2(\rho)$. 
Again, as in step 2, the convergence of $\tilde{S}^t_p \varphi_n \to \tilde{S}^t_p \varphi$ implies by Chebyshev's inequality the convergence in probability (with respect to $\rho(x)\rmd x$). Therefore it follows that there exists a sequence $(n_k)_{k\geq 0}$ such that  \(\tilde{S}^t_p \varphi_{n_k}(x) \to \tilde{S}^t_p \varphi(x)\) holds Lebesgue almost surely as $n_k\to\infty$.
In conclusion \(S^t_p \varphi = \tilde{S}^t_p \varphi\) holds pointwise almost surely for $\varphi\in C_W$, and, in particular, $S^t_p \varphi$ is the unique continuous extension of  \(\tilde{S}^t_p \varphi\).
		\end{enumerate}
		This finishes the proof.
	\end{proof}

In addition, we use the following general result, the proof of which is given in Appendix \ref{app:proofSpecRelation}. 
	\begin{lem}\label{lem:specRelationGeneral4}
		Let $(\Bc, |\cdot|)$ be a Banach space, and let $\Bc_0 \subset \Bc$ be a dense subspace equipped with its own norm $|\cdot|_0$ with respect to which $(\Bc_0, |\cdot|_0)$ is also Banach, and for which $|x| \leq C |x|_0$ for all $x \in \Bc_0$ and for some $C > 0$. 
		
		Let $A : \Bc \to \Bc$ be a bounded linear operator for which $A(\Bc_0) \subset \Bc_0$ and $A_0 := A|_{\Bc_0} : (\Bc_0, |\cdot|_0) \circlearrowleft$ is a bounded linear operator. Finally, assume $A : (\Bc, |\cdot|) \circlearrowleft$ and $A_0 : (\Bc_0, |\cdot|_0) \circlearrowleft$ are both compact on their respective spaces. Then, 
		\[\rho(A : \Bc \to \Bc) = \rho(A_0 : \Bc_0 \to \Bc_0) \, , \]
		i.e., the spectral radius of $A$ regarded on $\Bc$ coincides with that of $A_0 = A|_{\Bc_0}$ regarded on $\Bc_0$. 
	\end{lem}

\begin{proof}[Proof of Lemma \ref{lem:spectraCoincide4}]
	Let $\tilde S^t_p : L^2(\rho) \circlearrowleft$ denote the $C^0$ semigroup generated by $L_p$. By the spectral mapping theorem in this context, 
	\[\spec(\tilde{S}^t_p) \setminus \{ 0 \} = e^{t \spec(L_p)}  \]
	for all $t > 0$. Thus, to prove Lemma \ref{lem:spectraCoincide4}, it suffices to check, for some fixed $t > 0$, that the spectral radius of $\tilde S^t_p$ on $L^2(\rho)$ coincides with that of $S^t_p$ on $C_W$. 

	For this, we will check the hypotheses of Lemma \ref{lem:specRelationGeneral4}:
fixing $\theta < \frac{1}{\sigma^2}$, we have that $C_W \subset L^2(\rho)$ by~\eqref{eq:CWinL2}.
	The compactness of $S^t_p$ on $C_W$ follows\footnote{In \cite{FStoltz20}, they establish compactness of $S^t_p$ in the space $B^\infty_W$ of measurable functions for which the $C_W$-norm is finite. Compactness of $S^t_p$ in $C_W$ follows from closedness of $C_W \subset B^\infty_W$ in the $C_W$-norm.} along the lines of \cite[Lemma 6.6]{FStoltz20}, while compactness of $\tilde S^t_p$ on $L^2(\rho)$ follows from well-known results in the theory of Schr\"odinger operators \cite[Thm. XIII.16]{reed1978iv}. 
Due to Lemma~\ref{coincide} the operators $\tilde S^t_p $ and $ S^t_p$ coincide on $C_W$ such that Lemma \ref{lem:specRelationGeneral4} applies. 
We conclude that $\Lambda(p)$ is the dominant eigenvalue of $L_p$, or equivalently, that of $\tilde H_p$.
\end{proof}

\begin{rmk}\ 

	\begin{itemize}
		\item[(1)] The above symmetrization procedure applies to  FTLE of any scalar gradient SDE 
		\[\rmd X_t = - U'(X_t) \rmd t + \sigma \rmd W_t \, , \]
		for which the tilted generator $L_p$ on $L^2(\rho)$, $\rho = Z^{-1} \exp(- 2 U / \sigma^2)$, will read
		\[L_p g = \frac{\sigma^2}{2} \partial_{xx} g - U' \partial_x g - p U'' g,\]
		and for which the symmetrization $\tilde H_p$ on $L^2$ reads as 
		\[\tilde H_p f =  \frac{\sigma^2}{2} \partial_{xx} f + \left(\left(\frac12 - p\right) U''  - \frac{(U')^2}{2 \sigma^2} \right)f \,.  \]
		For further details on the symmetrization procedure for tilted semigroups associated to gradient SDE, see \cite[Section IV.C]{Touchette}. 
		Under reasonable conditions on the potential $U$, this allows to use the computer-assisted scheme in Section \ref{sec:homotopy} to enclose the corresponding rate function $\mathcal{I}$ for FTLE for this class of SDE. 
		
		\item[(2)] We caution, however, that this symmetrization procedure cannot generally be applied to FTLE of gradient SDE in dimension $\geq 2$, as there is no reason a priori to expect to be able to symmetrize the generator for the projective process $(X_t, s_t)$ on $P \R^d$. 
	\end{itemize}

\end{rmk}

\subsection{A two-dimensional toy model}
\label{sec:ToyModel}


We consider the following linear stochastic differential equation
\begin{equation} \label{eq:specvarEqusim_new_v2}
	\rmd Y_t = A Y_t \,\rmd t + B Y_t \circ \rmd W_t\,, \ A= \begin{pmatrix}
		- \alpha & 0 \\ b & - \alpha
	\end{pmatrix}\,, \ B = \begin{pmatrix}
		0 & \sigma \\ - \sigma & 0
	\end{pmatrix}, \ Y_0 \in \mathbb{R}^2,
\end{equation}
where $\alpha > 0$ is the linear contraction parameter, $b > 0$ represents some shearing, $W_t$ is a one-dimensional Brownian motion and $\sigma >0$ expresses noise intensity. 


A related model for shear-induced chaos \cite{LinYoung} was investigated in \cite{EngelLambRasmussen}. More recently, it was shown that the linearization of this model arises as a scaling limit for the linearization of the Hopf normal form with additive noise \cite{ChemnitzEngel}. There are, however, two key differences between the linearizations considered in \cite{ChemnitzEngel, EngelLambRasmussen} and the linear model~\eqref{eq:specvarEqusim_new_v2}: 
\begin{itemize}
	\item[(1)] The entry $A_{2,2} = - \alpha$ replaces $0$ from the model in \cite{EngelLambRasmussen} to make it so that a transition from uniform to nonuniform synchronization takes place and is easier to describe. 
    \item[(2)] For technical reasons we have replaced a zero entry in $B$ by $B_{2,1} = - \sigma$ to make the diffusion for the projective process uniformly elliptic instead of hypoelliptic. This greatly simplifies the computer-assisted work in Section \ref{sec:CAP_Toy}. See Section \ref{sec:outlook} for a brief additional discussion on this point. 
\end{itemize}


Trivially, the assumptions of Section~\ref{subsec:compact2} are satisfied here such that Theorems~\ref{thm:mlp} and~\ref{thm:LDPs} hold.
As discussed in Section~\ref{sec:prelim}, we introduce the process $s_t = Y_t/ \|Y_t\| \in \mathbb S^{1}$, which can be described on the projective space ${P}^{1}$ by identifying $s_t = - s_t$.
The dynamics are given by (cf.~\eqref{eq:projective_process})
\begin{align*}
	\rmd s &= ( As - \langle s, As \rangle s)\, \rmd t +  (Bs - \langle s, B s \rangle s)\, \circ \rmd W_t \\
	&= \begin{pmatrix}
		- \alpha s_1 - s_1(-\alpha s_1^2 - \alpha s_2^2 + b s_1 s_2) \\ \alpha s_2 + b s_1 - s_2(-\alpha s_1^2 - \alpha s_2^2 + b s_1 s_2)
	\end{pmatrix} \rmd t + \sigma \begin{pmatrix}
		s_2  \\
		- s_1
	\end{pmatrix} \circ \rmd W_t\,.
\end{align*}
Above, for notational clarity we have dropped the `$t$' dependence for $s = s_t$, $s_i = s_{i,t}$, which should be read implicitly throughout. 

Note that, since $B$ is antisymmetric, condition~\eqref{eq:forceByIsom3} is satisfied, i.e.~$q_j = 0$ for all $j\geq 1$, and the Furstenberg-Khasminskii functional $Q$ is given by
\begin{align*}
	Q &= \langle s, As \rangle + \frac{1}{2} \langle \left( B + B^{\top} \right)s, B s \rangle - \langle s, B s \rangle^2\\
	&= - \alpha s_1^2 - \alpha s_2^2 + b s_1 s_2 + 0 = - \alpha + b s_1 s_2.
\end{align*}
In particular, $Q=q_0$ in this case.
By the change of variables to $s = (\cos \phi, \sin \phi)$, the SDE determining the dynamics of the angular component $ \phi \in [0,  \pi)$ on the one-dimensional projective space, based on~\cite[Section 2]{ImkellerLederer}, reads as
\begin{equation} \label{eq:Phi_model2_v2}
	\rmd \phi = - \frac{1}{\sin \phi} \rmd s_1 =  b \cos^2 \phi \rmd t -  \sigma  \circ \rmd W_t, \ \phi_0 \in [0, \pi)\,.
\end{equation}
We can now express the asymptotic Lyapunov exponent by the Furstenberg-Khasminskii formula~\eqref{eq:FK_general}
\begin{equation} \label{eq:Lyap_newmodel_v2}
	\lambda = \int_{ [0,\pi]} q_0(\phi) \eta(\phi) \rmd \phi,
\end{equation}
where $\eta = \frac{d \nu}{d \varphi}$ is the density of the stationary measure $\nu$ for $(\varphi_t)$ and solves the stationary Fokker-Planck equation $L^*\eta =0$; and $q_0(\phi) = - \alpha  + b \cos \phi \sin \phi$.
Observe that the finite time Lyapunov exponents are given by
\begin{equation}
	\label{eq:FTLEs_testproblem}
	\lambda_t(\phi_0) = \frac{1}{t} \int_0^t q_0(\phi_\tau) \rmd \tau,
\end{equation}
and in particular are uniformly negative as long as $b < 2 \alpha$. The values of $q_0$ become sign-indefinite when $b$ crosses $2 \alpha$, and so by Proposition~\ref{prop:LDPsignChange} (a) an FTLE transition occurs here. 

For computation of the moment Lyapunov exponents we consider the tilted generator 
\begin{equation} \label{eq:tilted_gen_model2_v2}
	L_p f(\phi) = \frac{1}{2} \sigma^2 \partial^2_{\phi} f(\phi)+ b \cos^2 \phi \partial_{\phi} f(\phi) 
	+ p \left( b \cos \phi \sin \phi - \alpha  \right) f(\phi),
\end{equation}
with periodic boundary conditions on $[0, \pi]$. Indeed, we recall that the moment Lyapunov exponents $\Lambda(p)$ are exactly the maximal eigenvalues of $L_p$.

\begin{lem}\label{lem:semigroupToGenerator_shear}
	For system~\eqref{eq:specvarEqusim_new_v2}, the moment Lyapunov exponent $\Lambda(p)$ is the principal eigenvalue of the tilted generator $L_p$ as given in~\eqref{eq:tilted_gen_model2_v2}, treated as a closed linear operator on $\ell^1(\Z, \C)$ (see Section \ref{sec:CAP_Toy} for definition and details). That is, $\Lambda(p)$ is a simple, isolated eigenvalue of $L_p$ and has real part strictly larger than that of the rest of $\spec(L_p)$. 
\end{lem}
\begin{proof}
	We observe that $L_p=L+pq_0$ where $L = L_0$ is the generator for~\eqref{eq:Phi_model2_v2}.
	Therefore, we can apply~\cite[Theorem 1.29]{BaxendaleStroock88} (as pointed out in the discussion on page 199 in \cite{BaxendaleStroock88}).
\end{proof}

\section{Rigorous enclosures for moment Lyapunov exponents}
\label{sec:CAP}


In this section, we derive explicit enclosures for the moment Lyapunov exponents $\Lambda(p)$ associated to the two problems described in Section~\ref{sec:test_problems}, using computer-assisted proofs. The obtained information allows us to give precise statements about the rate function $\mathcal{I}$, and also about the asymptotic variance of the Lyapunov exponent for the example of Section~\ref{sec:ToyModel}. 

It should be noted that each of the two examples presents different challenges when it comes to computer-assisted proofs. Regarding the problem of Section~\ref{sec:pitchfork}, the operator~\eqref{eq:Schroedinger_pitchfork} is defined on an unbounded domain, and contains unbounded terms, whereas many existing computer-assisted approaches are restricted to bounded domains. Here we take advantage of the fact that this problem has only real eigenvalues and can be recasted using a self-adjoint operator, which allows us to make use of the homotopy method~\cite[Chapter 10]{NakPluWat19}. With this approach, the fact that the potential in~\eqref{eq:Schroedinger_pitchfork} is unbounded turns out not to be an issue because eigenfunctions are localized. To the best of our knowledge, the only other existing computer-assisted results in the context of unbounded domains and unbounded potentials is~\cite{NagNakWak02}, which deals specifically with coupled harmonic oscillators, and the very recent~\cite{BreChu24}, which does not focus on eigenvalue problems. 

 The problem of Section~\ref{sec:ToyModel} is not self-adjoint (and cannot be made so because $L_p$ has complex eigenvalues), therefore the homotopy method cannot be used to directly enclose the eigenvalues of $L_p$. However, we can consider a zero-finding problem associated to finding an eigenpair (eigenvalue and eigenfunction) of $L_p$, and then use other well-established computer-assisted techniques in order to rigorously enclose a zero of this problem (see e.g.~\cite{BerLes15,NakPluWat19}). Because the operator~\eqref{eq:tilted_gen_model2_v2} is this time defined on a bounded domain, a wider breadth of computer-assisted techniques are available. Furthermore, we make use of the recent development in rigorous continuation techniques introduced in~\cite{Bre23} in order to get, not only enclosures for fixed $p$, but a rigorous enclosure of the map $p\mapsto\Lambda(p)$ over a compact interval $[p_-,p_+]$, in norm strong enough to also controls derivatives, which allows us to directly get an enclosure of the asymptotic variance $\Lambda''(0)$. We note that the homotopy method used in Section 5.1 is a more direct way of getting eigenvalue enclosures, but is does not lends itself so well to rigorous continuation, especially if one is interested in stronger than $L^\infty$ control with respect to the continuation parameter. This is why we do not compute $\Lambda''(0)$ in Section~\ref{sec:homotopy} but only in Section~\ref{sec:CAP_Toy}.
 
As is typical for computer-assisted proofs involving infinite-dimensional objects (here functions), most of the work consists in controlling truncation errors between the infinite dimensional problem we care about and the finite dimensional approximation used to do numerics. However, for some specific steps floating point errors also have to be controlled, which we do using interval arithmetic (see, e.g.,~\cite{Moo79,Tuc11}).

\subsection{Pitchfork bifurcation with additive noise}
\label{sec:homotopy}

The goal of this subsection is to provide an explicit and tight enclosure of the first eigenvalue of the operator $H_p = - \tilde H_p$, with $\tilde H_p$ as in \eqref{eq:Schroedinger_pitchfork}, and to study its dependency with respect $\alpha$. 
For a fixed value of $p$, we first use the so-called homotopy method~\cite[Chapter 10]{NakPluWat19} in order to get an explicit, tight and rigorous enclosure of the first eigenvalue $-\Lambda(p)$ of $H_p$. After getting several such enclosures for different values of $p$, an elementary convexity argument allows us to rigorously enclose the rate function $\Ic_\alpha$ at $0$.

\subsubsection{General strategy}

Let
\begin{align*}
V_p(x) = (\alpha - 3 x^2 ) \left( \frac{1}{2} - p \right) + \frac{(x^3 - \alpha x)^2}{2 \sigma^2} \, , 
\end{align*}
so that
\begin{align*}
H_p f(x) = -\frac{\sigma^2}{2} \partial_{xx} f(x) + V_p(x) f(x).
\end{align*}
We now explain how to get a rigorous enclosure of $\Lambda(p)$ for this problem.
Regarding $H_p$ as an operator on $L^2$ with (dense) domain $\Dc(H_p) = \{ f \in H^2(\R) : \, V_p f \in L^2(\R) \}$, it is known that $H_p$ is (i) an essentially self-adjoint operator\footnote{We recall that an unbounded, densely defined and symmetric operator is said to be essentially
self-adjoint if its closure is self-adjoint.}
; and (ii) that its (unique) self-adjoint extension has purely discrete spectrum (see, e.g.,~\cite{BerShu12}). 

In the sequel, $p$ is fixed and we drop the $p$-dependence from the notation for the moment. As we will require more than the first eigenvalue of $H$ along the way, we denote by $(\lambda_m)_{m\geq 0}$ the sequence of non-decreasing eigenvalues of $H$ counted with multiplicity. 


The Rayleigh-Ritz method gives us a straightforward way of getting rigorous and computable upper-bounds for a finite number of eigenvalues of $H$.

\begin{prop}{\cite[Theorem 10.12]{NakPluWat19}}
\label{prop:RR}
Let $f_0,\ldots,f_M$ in the domain of the self-adjoint extension of $H$ be linearly independent elements, let 
\begin{equation*}
A_0 = \left(\langle f_i,f_j \rangle\right)_{0\leq i,j\leq M} \qquad\text{and}\qquad A_1 = \left(\langle Hf_i,f_j \rangle\right)_{0\leq i,j\leq M},
\end{equation*}
and let
\begin{equation*}
\ol_0 \leq \ldots \leq \ol_M,
\end{equation*}
be the eigenvalues of the generalized eigenvalue problem
\begin{equation}
\label{eq:eig_RR}
A_1 v = \ol A_0 v.
\end{equation}
Then, for all $m\leq M$
\begin{equation*}
\l_m \leq \ol_m.
\end{equation*}
\end{prop}
\begin{rmk}
This result is nothing but a generalization to more than one eigenvalue of the basic estimate $\langle Hf,f \rangle \geq \lambda_0 \langle f,f \rangle$. The quality of the upper-bounds produced by Proposition~\ref{prop:RR} depends on the choice of the elements $f_i$. If we take them to be accurate approximations of the $M+1$ first eigenfunctions of $H$, we expect to get sharp upper-bounds.  

Provided the entries of the matrices $A_0$ and $A_1$ are computed with interval arithmetic, and the \emph{finite-dimensional} eigenvalue problem~\eqref{eq:eig_RR} is also solved rigorously with interval arithmetic (see, e.g.,~\cite{Rum99}), the $\barl_m$ are computable and rigorous upper-bounds for the eigenvalues $\lambda_m$.
\end{rmk}
Regarding lower-bounds, the Lehmann-Maehly method can be used, assuming some a priori knowledge on the next eigenvalue.
\begin{prop}{\cite[Theorem 10.14]{NakPluWat19}}
\label{prop:LM}
Repeat the assumptions of Proposition~\ref{prop:RR}. Assume further that there exists $\nu\in\R$ such that
\begin{equation}
\label{eq:cond_LM}
\ol_M <\nu \leq \l_{M+1},
\end{equation}
and define the matrices
\begin{equation*}
A_2 = \left(\langle Hf_i,Hf_j \rangle\right)_{0\leq i,j\leq M},\qquad B_1=A_1-\nu A_0 \quad\text{and}\quad B_2 = A_2 -2\nu A_1 + \nu^2 A_0.
\end{equation*}
Let
\begin{equation*}
\mu_0 \leq \ldots \leq \mu_M
\end{equation*}
be the eigenvalues of the generalized eigenvalue problem
\begin{equation}
\label{eq:eig_LM}
B_1 v = \mu B_2 v,
\end{equation}
and assume that $\mu_M<0$. Then, for all $m\leq M$
\begin{equation*}
\ul_m \leq \l_m,
\end{equation*}
where
\begin{equation*}
\ul_m := \nu + \frac{1}{\mu_{M-m}}.
\end{equation*}
\end{prop}
\begin{rmk}
Again, in the case $M=0$, this is nothing but the estimate $\langle (H-\nu)^{-1}f,f \rangle \geq (\lambda_1-\nu)^{-1} \langle f,f \rangle$, which, taking $f=(H-\nu)g$, yields $\langle (H-\nu)g,g \rangle \geq (\lambda_1-\nu)^{-1} \langle (H-\nu)g,(H-\nu)g \rangle$.
\end{rmk}
In order to use Proposition~\ref{prop:LM}, we need to find an explicit $\nu$ satisfying~\eqref{eq:cond_LM}. To that end, let us consider $a>0$, $b\in\R$ such that
\begin{align}
\label{eq:a_b_base_prblm}
V^{(0)}(x) := ax^2 + b \leq V(x)\quad \forall~x\in\R,
\end{align}
and the densely defined operator $H^{(0)}$ on $L^2(\R)$ given by
\begin{align*}
H^{(0)} f(x) = -\frac{\sigma^2}{2} \partial_{xx} f(x) + V^{(0)}(x) f(x).
\end{align*}
Denoting by $(\lambda^{(0)}_m)_{m\geq 0}$ the sequence of non-decreasing eigenvalues of $H^{(0)}$, we have that 
\begin{align}
\label{eq:mono_eig_0}
\lambda^{(0)}_m \leq \lambda_m\quad \forall~m\geq 0,
\end{align}
simply because, thanks to~\eqref{eq:a_b_base_prblm},
\begin{align*}
\langle Hf,f \rangle \geq \langle H^{(0)}f,f \rangle\quad \forall~f\in\Cc^\infty_c(\R).
\end{align*}
Moreover, the eigenvalues $\lambda^{(0)}_m$ are known explicitly, as $H^{(0)}$ is nothing but (a rescaled version of) the quantum harmonic oscillator, and we have
\begin{align*}
\lambda^{(0)}_m = \left(m+\frac{1}{2}\right)\sigma\sqrt{2a} + b \quad \forall~m\geq 0.
\end{align*}
In particular, if it happens that $\lambda^{(0)}_{M+1} >\ol_M$, then~\eqref{eq:mono_eig_0} ensures that any $\nu\in(\ol_M,\lambda^{(0)}_{M+1})$ satisfies~\eqref{eq:cond_LM}. 

However, the eigenvalues of $H^{(0)}$ might be \emph{too far away} from those of $H$, and it could be that $\lambda^{(0)}_{M+1} \leq \lambda_M$, in which case we cannot have $\lambda^{(0)}_{M+1} >\ol_M$. In order to make this procedure more general, we introduce a homotopy between $H^{(0)}$ and $H$, and consider
\begin{align*}
H^{(s)} := (1-s)H^{(0)} + sH \quad \forall~s\in[0,1].
\end{align*}
Denoting by $(\lambda^{(s)}_m)_{m\geq 0}$ the sequence of non-decreasing eigenvalues of $H^{(s)}$, we still have that 
\begin{align*}
\lambda^{(0)}_m \leq \lambda^{(s)}_m\quad \forall~m\geq 0,\ \forall~s\in[0,1],
\end{align*}
and this time there will be a positive $s_0\in(0,1]$ such that $\lambda^{(0)}_{M+1} > \lambda^{(s_0)}_{M}$. We can then apply Proposition~\ref{prop:RR} to $H^{(s_0)}$, and, if the obtained upper bound $\ol^{(s_0)}_M\geq \lambda^{(s_0)}_{M}$ is sharp enough so that $\lambda^{(0)}_{M+1} > \ol^{(s_0)}_M$, also use Proposition~\ref{prop:LM}  to obtain rigorous enclosure of the $M+1$ first eigenvalues of $H^{(s_0)}$. In particular, we do now have a rigorous lower bound $\ul^{(s_0)}_M \leq \lambda^{(s_0)}_{M}$ for the $M+1$-th eigenvalue of $H^{(s_0)}$, and can therefore repeat the whole procedure in order to get rigorous enclosures of the $M$ first eigenvalues of $H^{(s_1)}$ for some $s_1>s_0$. We keep iterating this process until we reach $s=1$, for which $H^{(1)} = H$. For further details, and generalizations of this homotopy method, we refer to~\cite[Section 10.2.4]{NakPluWat19}.

\subsubsection{Results}
\label{sec:pitchfork_results}

\begin{thm}
\label{th:single_sol_pitchfork}
Let $\alpha = 1$ and $\sigma = 1$ in~\eqref{eq:pitchfork}. The rate function satisfies\footnote{Below and throughout, decimals in the endpoints of intervals are bolded if the upper and lower limits at those values disagree, e.g., $[-.123\bm{45}, -.123\bm{32}]$.}
\begin{align*}
\Ic_1(0) = - \inf_{p \in \mathbb{R}}  \Lambda(p) \in [0.455\bm{1},0.455\bm{3}].
\end{align*}
\end{thm}
\begin{proof}
For any fixed $p$, we can obtain rigorous enclosures for the moment Lyapunov function $\Lambda(p)$ using the homotopy method described in Section~\ref{sec:homotopy}. In particular, we get
\begin{align*}
\Lambda(0.71646)&\in[-0.455268764\bm{21},-0.455268764\bm{19}], \\ \Lambda(0.71647)&\in[-0.455268764\bm{25},-0.455268764\bm{22}], \\ \Lambda(0.71648)&\in[-0.455268764\bm{19},-0.455268764\bm{16}].
\end{align*}  
These three enclosures, combined with the convexity of $\Lambda$ (Theorem~\ref{thm:mlp} $(iv)$), yield that the infimum of $\Lambda$ is reached in $[0.7164\bm{6},0.7164\bm{8}]$.\footnote{There is a subtlety here due to the fact that the computer stores floating-point values in base 2. 
To be precise, we show using interval arithmetic that the infimum of $\Lambda$ is reached in $[p_1,p_2]$, where $p_1$ is the largest floating point number (in the usual double precision) so that $p_1\leq 0.71646$ and $p_2$ is the smallest floating point number such that $0.71648\leq p_2$. It is then this interval $[p_1,p_2]$ on which we evaluate $\Lambda$ in order to enclose $\Ic(0)$.
}

Then, we evaluate $\Lambda$ on this interval $[0.7164\bm{6},0.7164\bm{8}]$ (this means $H_p$ is now an operator with interval coefficients, but that does not change anything to the homotopy method, for which interval arithmetic was already needed anyway), and the output provides the announced enclosure of $\inf_{p \in \mathbb{R}}  \Lambda(p)$.

The computational parts of the proof can be reproduced by running the Matlab code \texttt{script\_pitchfork\_single.m}, available at~\cite{BreGit23}, together with the Intlab toolbox~\cite{Rum99} for interval arithmetic.
\end{proof}

\begin{rmk}
\label{rem:tight_enclosures}
In principle, once an interval $[p_1,p_2]$ containing the minimizer of $\Lambda$ has been obtained, one can first subdivide this interval into smaller subintervals, evaluate $\Lambda$ on each subinterval separately, and then take the minimum of the obtained enclosures. This is more costly, but typically provides a sharper enclosure of $\Lambda([p_1,p_2])$.
\end{rmk}

There is of course nothing special about the parameter values that were chosen for $\alpha$ and $\sigma$ in Theorem~\ref{th:single_sol_pitchfork}, and one can for instance repeat the same procedure for different values of $\alpha$, leading to Figure~\ref{fig:pitchfork_various_alphas_intro}. The rigorous computations required to produce this picture can be reproduced by running the Matlab code \texttt{script\_pitchfork\_all.m}, available at~\cite{BreGit23}, together with the Intlab toolbox~\cite{Rum99} for interval arithmetic. By narrowing in onto the minimum, we also obtain the following statement.
\begin{thm}
\label{th:pitchfork_minimum}
Let $\sigma = 1$ in~\eqref{eq:pitchfork}. The map $\alpha\mapsto \Ic_\alpha(0)$ has a local minimum at some $\alpha_* \in [1.22\bm{5},1.22\bm{9}]$.
\end{thm}
\begin{proof}
Proceeding as in the proof of Theorem~\ref{th:single_sol_pitchfork}, we rigorously enclose $\Ic_\alpha(0)$ for three different values of $\alpha$. We obtain
\begin{align*}
&\Ic_{1.225}(0) \in [0.43605\bm{413}, 0.43605\bm{513}],\\
&\Ic_{1.227}(0) \in [0.43605\bm{254}, 0.43605\bm{354}],\\
&\Ic_{1.229}(0) \in [0.43605\bm{359}, 0.43605\bm{460}].
\end{align*} 
In particular, $\Ic_{1.227}(0)<\Ic_{1.225}(0)$ and $\Ic_{1.227}(0)<\Ic_{1.229}(0)$, which yields the existence of a local minimum of $\alpha\mapsto \Ic_\alpha(0)$ in $[1.22\bm{5},1.22\bm{9}]$. The rigorous computations required for this proof can be reproduced by running the Matlab code \texttt{script\_pitchfork\_minimum.m}, available at~\cite{BreGit23}, together with the Intlab toolbox~\cite{Rum99} for interval arithmetic.
\end{proof}

\subsection{A two-dimensional toy model}

\label{sec:CAP_Toy}

The goal of this subsection is to provide an explicit and tight enclosure of the first eigenvalue of the operator $L_p$ defined in~\eqref{eq:tilted_gen_model2_v2}. Most of the computer-assisted arguments used in this section are by now standard, except maybe for the rigorous continuation framework allowing us to easily control at once the map $p\mapsto\Lambda(p)$ and its derivative. For the sake of convenience, we rescale the problem to $[0,2\pi]$, and instead consider the operator (still denoted $L_p$) defined by
\begin{align*}
L_p f (\phi) = 2\sigma^2\partial^2_\phi f(\phi) + b(1+\cos(\phi)) \partial_\phi f(\phi) + \frac{p}{2}\left(b\sin(\phi) - 2\alpha\right) f(\phi), \quad \phi\in[0,2\pi],
\end{align*}
with periodic boundary conditions. Keeping $p$ fixed for the moment, and denoting by $(\bar f,\bar \lambda)$ an approximate eigenpair of $L_p$, computed numerically, we introduce the zero finding problem
\begin{align}
\label{eq:def_Fp}
\renewcommand{\arraystretch}{1.5}
F_p(f,\lambda) = \begin{pmatrix}
L_pf - \lambda f \\
\left\langle f,\barf \right\rangle_{L^2} -1
\end{pmatrix},
\end{align}
where
\begin{align*}
\left\langle f,g \right\rangle_{L^2} := \frac{1}{2\pi} \int_0^{2\pi} f(\phi)g(\phi)^* d\phi,
\end{align*}
with $z^*$ denoting the complex conjugate of a complex number $z$. A zero of $F_p$ corresponds to an exact eigenpair of $L_p$, and then second equation in $F_p$ is a normalization condition ensuring local uniqueness of the eigenpair.

%
%

\subsubsection{General strategy}
\label{sec:shear_general}

We use a Newton-Kantorovich type of argument, in an appropriate function space that will be introduced just below,  to prove that there indeed exists a true zero $(f,\lambda)$ of $F_p$ near the approximate eigenpair $(\bar f,\bar\lambda)$. We then also have to prove that the obtained eigenvalue $\lambda$ is really $\Lambda(p)$, and not any other eigenvalue of $L_p$. Since eigenfunctions of $L_p$ are eigenfunctions of $S^t_p, t > 0$ and since $S^t_p$ is nonnegative and admits a simple dominant eigenvalue, it will be enough to check a posteriori that the eigenfunction $f$ is nonnegative to get that $\lambda = \Lambda(p)$. For more on semigroups of compact, positive operators, see, e.g., \cite{birkhoff1957extensions}.

Because of the periodic boundary conditions, we look for eigenfunctions as Fourier series:
\begin{align*}
f(\phi) = \sum_{n\in\Z} f_n e^{in\phi}.
\end{align*}
In the sequel, we always identify such a function with its sequence $(f_n)_{n\in\Z}$ of Fourier coefficients, still denoted $f$. In particular, the Fourier coefficients of $L_p f$ write
\begin{align}
\label{eq:Lpf_Fourier}
\left(L_pf\right)_n = \left(-2\sigma^2n^2 + ibn - \alpha p\right)f_n   + i\left(\frac{b}{2}(n+1) + \frac{pb}{4}\right)f_{n+1} + i\left(\frac{b}{2}(n-1) - \frac{pb}{4}\right)f_{n-1}.
\end{align}

\begin{defn}
We define
\begin{align*}
\ell^1(\Z,\C) = \left\{ f \in \C^\Z ,\ \left\Vert f \right\Vert_{\ell^1} < \infty \right\},
\end{align*}
where
\begin{align*}
\left\Vert f \right\Vert_{\ell^1} =
 \sum_{n\in\Z} \left\vert f_n \right\vert ,
\end{align*}
together with $\Xc = \ell^1(\Z,\C) \times \C $, where $\left\Vert (f,\lambda) \right\Vert_{\Xc} = 
\left\Vert f \right\Vert_{\ell^1} + \left\vert \lambda \right\vert$.
\end{defn}
In practice, we compute the approximate eigenfunction $\bar f$ as a truncated Fourier series, therefore $\bar X = (\bar f,\bar\lambda)$ trivially belongs to $\Xc$. We want to show that, for some (explicit and small) $r>0$, there exists a unique zero $X=(f,\lambda)$ of $F_p$ in $\Xc$ such that $\left\Vert X - \barX\right\Vert_{\Xc} \leq r$. This will be accomplished via the contraction mapping theorem, applied to a Newton-like operator associated to $F_p$, which we define momentarily.

Given a truncation level $N\in\N$, we introduce the corresponding truncation operator $\Pi_N$: For $f=\left(f_n\right)_{n\in\Z}$, 
\begin{align*}
\left(\Pi_N f\right)_n = \left\{
\begin{aligned}
&f_n \qquad & \vert n\vert \leq N,\\
&0 \qquad & \vert n\vert > N.\\
\end{aligned}
\right.
\end{align*}
For $X=(f,\lambda) \in \Xc$,
$\Pi_{N} X = (\Pi_{N} f, \lambda)$. In the sequel, we assume that the approximate eigenpair $\barX=(\barf,\barl)$ belongs to $\Pi_N\Xc$, for some $N$ to be specified later.

We then define the linear operator $A$ by
\begin{align}
\label{eq:def_A}
A X = J \Pi_N X + \frac{1}{2\sigma^2} \partial^{-2}_\phi (I-\Pi_N) X,
\end{align}
where $J$ is a linear operator on $\Pi_{N} \Xc$, which can therefore be identified with an $(2N+2)\times(2N+2)$ matrix, computed numerically as an approximation of $\left(\Pi_N F_p'(\barX)\Pi_N\right)^{-1}$. Here $\partial^{-2}_\phi$ is just the inverse Laplacien operator, i.e. $\left(\partial^{-2}_\phi f\right)_n = -\frac{f_n}{n^2}$, which is well defined on $(I-\Pi_N) X$ since $(I-\Pi_N) X$ only contains mean zero functions. Provided $N$ is large enough, such an $A$ should provide us with an explicit and accurate approximate inverse of the Fr\'echet derivative $F_p'(\bar X)$, which we can use in the following theorem.

\begin{thm}
\label{thm:NK}
Let $F_p$ be as in~\eqref{eq:def_Fp} and $A$ of the form~\eqref{eq:def_A}. Assume there exist $Y,Z_1,Z_2\geq 0$ such that
\begin{align}
\label{eq:cond_NK}
\left\Vert AF_p(\barX) \right\Vert_{\Xc} &\leq Y, \nonumber\\
\left\Vert I - AF_p'(\barX) \right\Vert_{\Xc} &\leq Z_1, \\
\left\Vert A(F_p'(X)-F_p'(\barX)) \right\Vert_{\Xc} &\leq Z_2 \left\Vert X-\barX \right\Vert_{\Xc},\qquad \forall X\in\Xc. \nonumber
\end{align}
If 
\begin{align}
\label{eq:cond_NK_contraction}
Z_1 &< 1, \nonumber\\
2YZ_2 &< (1-Z_1)^2,
\end{align}
then, for any $r$ satisfying 
\begin{align*}
\frac{1-Z_1 - \sqrt{(1-Z_1)^2-2YZ_2}}{Z_2} \leq r < \frac{1-Z_1}{Z_2},
\end{align*}
there exists a unique zero $X$ of $F_p$ in $\Xc$ such that $\left\Vert X - \barX\right\Vert_{\Xc} \leq r$. Furthermore, if $\barX=(\barf,\barl)$ is real, i.e. $\barf_{-n} = \left(\barf_{n}\right)^*$ for all $n\in\Z$ and $\barl\in\R$, then so is $X$. 
\end{thm}
For other versions of the above theorem, and its proof, see e.g.~\cite{AriKocTer05,DayLesMis07,Plu01,BerBreLesVee21,Yam98} and the references therein. This version is particularly close to~\cite[Theorem 2.19]{BerBreLesVee21}, where a proof is included. We note that in general the $Z_2$ estimate only needs to hold in a neighborhood of $\barX$, but here such an estimate will naturally be global because the map $F_p$ is quadratic. We derive below explicit constants $Y$, $Z_1$ and $Z_2$ satisfying~\eqref{eq:cond_NK}, and then check, using interval arithmetic~\cite{Rum99} to control rounding errors, that these constants satisfy~\eqref{eq:cond_NK_contraction} for some given approximate solution $\barX$.

\begin{rmk}
\label{rem:checking_positivity}
Once Theorem~\ref{thm:NK} has been successfully applied, we get an explicit $r$ such that $\left\Vert f-\barf \right\Vert_{\ell^1} \leq r$. Note that the $\ell^1$-norm controls the $C^0$-norm, hence
\begin{align*}
\inf_{\phi\in[0,2\pi]} f(\phi) \geq \inf_{\phi\in[0,2\pi]} \barf(\phi) - r,
\end{align*}
which provides us with an easy way of checking a posteriori that the eigenfunction $f$ is indeed nonnegative.
\end{rmk}

\subsubsection{The bounds}
\label{sec:shear_bounds}

\paragraph{The $Y$ bound.}
As $\barf$ is a truncated Fourier series, $AF_p(\barX)$ also has only finitely many non-zero coefficients. Therefore we can simply take $Y = \left\Vert AF_p(\barX) \right\Vert_{\Xc}$, which is a finite computation that can be carried out on a computer, with interval arithmetic to account for rounding errors.

\paragraph{The $Z_1$ bound.}
In order to derive a suitable $Z_1$ bound, let us introduce $B_p = I - AF_p'(\barX)$. We have to estimate
\begin{align*}
\left\Vert B_p \right\Vert_{\Xc} = \sup_{\substack{X \in \Xc \\ X\neq 0}} \frac{\left\Vert B_pX \right\Vert_{\Xc}}{\left\Vert X \right\Vert_{\Xc}},
\end{align*}
which we do by noticing $\left\Vert B_p \right\Vert_{\Xc}$ can be split as follows (see, e.g., \cite[Lemma 2.13]{Bre23}):
\begin{align}
\label{eq:Z1_split}
\left\Vert B_p \right\Vert_{\Xc} = \max\left(\sup_{\substack{X \in \Pi_{N+1} \Xc \\ X\neq 0}} \frac{\left\Vert B_pX \right\Vert_{\Xc}}{\left\Vert X \right\Vert_{\Xc}},\ \sup_{\substack{X \in \left(I-\Pi_{N+1}\right) \Xc \\ X\neq 0}} \frac{\left\Vert B_pX \right\Vert_{\Xc}}{\left\Vert X \right\Vert_{\Xc}}\right).
\end{align}
The first sup is in fact a maximum:
\begin{align*}
\sup_{\substack{X \in \Pi_{N+1} \Xc \\ X\neq 0}} \frac{\left\Vert B_pX \right\Vert_{\Xc}}{\left\Vert X \right\Vert_{\Xc}} =  \max\left(\left\Vert B_p(0,1) \right\Vert_{\Xc},\max_{0\leq n\leq N} \frac{\left\Vert B_p (e_n,0)\right\Vert_{\Xc}}{\left\Vert e_n \right\Vert_{\ell^1}} \right),
\end{align*} 
where $e_n$ is the sequence having $1$ as $n$-th entry and $0$ otherwise. Similarly to the $Y$ computation, each element $B_p (e_n,0)$ of $\Xc$ only has finitely many non-zero coefficients, and the above estimate can be evaluated on a computer.

In order to deal with the second supremum in~\eqref{eq:Z1_split} we take advantage of the fact that, for any $X \in \left(I-\Pi_{N+1}\right) \Xc$, $B_pX$ does not depend on $J$. Indeed, for an arbitrary $X=(f,\lambda)\in \Xc$, denoting $V = (g,\mu) = F_p'(\barX)X$, we have (taking the Frechet derivative of~\eqref{eq:def_Fp} and looking at~\eqref{eq:Lpf_Fourier})
\begin{align*}
g_n &= \left(-2(\sigma n)^2-\alpha p +i bn\right) f_n + \frac{ib(2(n+1)+p)}{4} f_{n+1} + \frac{ib(2(n-1)-p)}{4} f_{n-1} - \barl f_n - \lambda \barf_n\quad \forall~n\in\Z,\\
\mu &= \sum_{\vert n\vert \leq N} f_n \left(\barf_n\right)^*,
\end{align*}
where we recall that $\barf\in\Pi_N\ell^1$.
In particular, if $X \in \left(I-\Pi_{N+1}\right) \Xc$, 
\begin{align*}
g_n &= 0 \quad \forall~\vert n\vert \leq N, \\
\mu &= 0,
\end{align*}
hence 
\begin{align*}
B_pX = X - AV = \begin{pmatrix}
f - \frac{1}{2\sigma^2}\left(\partial^{-2}_\phi\right) g \\ 0
\end{pmatrix}.
\end{align*}
Then, still assuming $X \in \left(I-\Pi_{N+1}\right) \Xc$,
\begin{align*}
\left\Vert B_pX \right\Vert_{\Xc} &\leq 
\frac{1}{2\sigma^2} \sum_{\vert n\vert \geq N+2} \left( \frac{\left(\vert\alpha p+\barl \vert + \vert bn\vert\right)}{n^2}  + \frac{\vert b\vert\left(2\vert n\vert + \vert p\vert\right)}{4} \left( \frac{1}{(n-1)^2} + \frac{1}{(n+1)^2} \right)  \right)\left\vert f_n\right\vert  \\
&\leq \frac{1}{2\sigma^2} \Bigg( \frac{\left\vert \alpha p + \barl \right\vert + \vert b\vert (N+2)}{(N+2)^2} + \frac{\vert b\vert\left(2 (N+2) + \vert p\vert\right)}{2(N+1)^2} \Bigg) \sum_{\vert n\vert \geq N+2} \vert f_n\vert ,
\end{align*}
and therefore
\begin{align*}
\sup_{\substack{X \in \left(I-\Pi_{N+1}\right) \Xc \\ X\neq 0}} \frac{\left\Vert B_pX \right\Vert_{\Xc}}{\left\Vert X \right\Vert_{\Xc}} \leq \frac{1}{2\sigma^2} \Bigg( \frac{\left\vert \alpha p + \barl \right\vert + \vert b\vert (N+2)}{(N+2)^2} + \frac{\vert b\vert\left(2 (N+2) + \vert p\vert\right)}{2(N+1)^2} \Bigg).
\end{align*}
We thus consider 
\begin{align*}
Z_1^{finite} = \max\left(\left\Vert B_p(0,1) \right\Vert_{\Xc},\max_{0\leq n\leq N} \frac{\left\Vert B_p (e_n,0)\right\Vert_{\Xc}}{\left\Vert e_n \right\Vert_{\ell^1}} \right),
\end{align*}
and 
\begin{align*}
Z_1^{tail} = \frac{1}{2\sigma^2} \Bigg( \frac{\left\vert \alpha p + \barl \right\vert + \vert b\vert (N+2)}{(N+2)^2} + \frac{\vert b\vert\left(2 (N+2) + \vert p\vert\right)}{2(N+1)^2} \Bigg),
\end{align*}
which are both computable, and take $Z_1 = \max(Z_1^{finite},\, Z_1^{tail})$.

\paragraph{The $Z_2$ bound.}

The only nonlinear term in $F_p(f,\lambda)$ being the $-\lambda f$ term, we have, for any $X = (f,\lambda)$ in $\Xc$,
\begin{align*}
\left(F_p'(X)- F_p'(\barX)\right) X = \begin{pmatrix}
(\barl-\lambda)f + \lambda(\barf - f) \\
0
\end{pmatrix},
\end{align*}
therefore
\begin{align*}
\left\Vert \left(F_p'(X)- F_p'(\barX)\right)  X \right\Vert_{\Xc} \leq \left\Vert X - \barX \right\Vert_{\Xc} \left\Vert X \right\Vert_{\Xc},
\end{align*}
and we can take $Z_2 = \left\Vert A \right\Vert_{\Xc}$. Because of the specific structure of $A$ (see~\eqref{eq:def_A}), and of our choice of norm, this operator norm can also be computed explicitly:
\begin{align*}
\left\Vert A \right\Vert_{\Xc} = \max\left(\left\Vert J \right\Vert_{\Xc},\, \frac{1}{2\sigma^2(N+1)^2} \right),
\end{align*}
where the remaining operator norm applies to a finite dimensional operator and therefore amounts to a (weighted) $\ell^1$ matrix norm.

\subsubsection{Extension to a varying $p$}
\label{sec:shear_continuation}

We explained in Sections~\ref{sec:shear_general} and~\ref{sec:shear_bounds} how to get rigorous a posteriori error bounds for an eigenpair $(f,\lambda)$ of $L_p$, for a fixed value of $p$, and how to check that $\lambda = \Lambda(p)$. However, in view of Theorem~\ref{thm:LDPs}, it would be interesting to have a rigorous enclosure of the map $p\mapsto \Lambda(p)$, at least on a compact interval $[p_-,p_+]$, in a strong enough norm so that we could also control quantities likes the asymptotic Lyapunov exponent $\Lambda'(0)$ and the asymptotic variance $\Lambda''(0)$.

The rigorous continuation tools developed in~\cite{Bre23} allows us to do exactly that in a straightforward manner. We expand all functions of $p$ using Chebyshev series, and look for eigenvalues as Chebyshev series, and for eigenfunctions as Fourier-Chebyshev series:
\begin{align*}
\lambda(p) = \lambda_0 + 2\sum_{k=1}^\infty \lambda_{k} \tilde T_{k} (p),\qquad f(p,\phi) = \sum_{n\in\Z} \left(f_{n,0} + 2 \sum_{k=1}^\infty f_{n,k} \tilde T_{k}(p) \right) e^{in\phi},
\end{align*}
where $\tilde T_k$, $k\in\N$, are the Chebyshev polynomials of the first kind, rescaled from $[-1,1]$ to $[p_-,p_+]$: $\tilde T_k(p) = T_k\left(\frac{2p-(p_++p_-)}{p_+-p_-}\right)$. We then define, for $\eta\geq 1$,
\begin{align*}
\ell^1_\eta(\N,\C) = \left\{ u \in \C^\N ,\ \left\Vert u \right\Vert_{\ell^1_\eta} < \infty \right\},
\end{align*}
where
\begin{align*}
\left\Vert u \right\Vert_{\ell^1_\eta} =
 \vert u_0\vert + 2\sum_{k=1}^\infty \left\vert u_k \right\vert \eta^{k} = \sum_{k\in\Z} \left\vert u_{\vert k\vert} \right\vert \eta^{\vert k\vert},
\end{align*}
and
\begin{align*}
\ell^1\left(\Z,\ell^1_\eta(\N,\C)\right) = \left\{ f \in \left(\ell^1_\eta(\N,\C)\right)^\Z ,\ \left\Vert f \right\Vert_{\ell^1,\ell^1_\eta} < \infty \right\},
\end{align*}
where
\begin{align*}
\left\Vert f \right\Vert_{\ell^1,\ell^1_\eta} = \sum_{n\in\Z} \left\Vert f_n \right\Vert_{\ell^1_\eta}  = \sum_{n\in\Z} \left( \left\vert f_{n,0}\right\vert + 2 \sum_{k=1}^\infty \left\vert f_{n,k}\right\vert \eta^k \right) ,
\end{align*}
and look for a zero of~\eqref{eq:def_Fp}  in $\Xc_{\eta} = \ell^1\left(\Z,\ell^1_\eta(\N,\C)\right) \times \ell^1_\eta(\N,\C)$, with norm
\begin{align*}
\left\Vert (f,\lambda)\right\Vert_{\Xc_{\eta}} =  
\left\Vert f \right\Vert_{\ell^1,\ell^1_\eta} + \left\Vert \lambda \right\Vert_{\ell^1_\eta}.
\end{align*}
Our approximate zero $(\barf,\barl)$ is then made of a truncated (in $n$ and $k$) Fourier-Chebyshev series $\barf$ and of a truncated Chebyshev series $\barl$.

\begin{rmk}
For given truncation levels, say $N$ in Fourier and $K$ in Chebyshev, the size of the approximate zero $(\barf,\barl)$ is now of the order of $NK$ (the actual number being $(K+1)(2N+1)+1$). However, because there are only multiplication operators in the $p$ variable, matrices like $J$ can in fact be stored with only order $N^2K$ cost (and not $N^2K^2$ as one might initially think), which makes the cost to the continuation not so prohibitive in practice. For more details, we refer to~\cite{Bre23,BreHen24}.
\end{rmk}

The advantage of this setup compared to more traditional continuation methods is twofold. First, the estimates required for using the Newton-Kantorovich Theorem~\ref{thm:NK} in this new space $\Xc_{\eta}$ are essentially the same as in the fixed $p$ case, up to the fact that real numbers which depend on $p$ becomes elements of $\ell^1_\eta(\N,\C)$, and that absolute values applied to such quantities have to be replaced by $\ell^1_\eta$-norms. In particular, the $Z_1^{tail}$ estimate naturally generalizes to
\begin{align*}
Z_1^{tail} = \frac{1}{2\sigma^2} \Bigg( \frac{\left\Vert \alpha p + \barl \right\Vert_{\ell^1_\eta} + \vert b\vert (N+2)}{(N+2)^2} +   \frac{\vert b\vert\left(2 (N+2) + \left\Vert p\right\Vert_{\ell^1_\eta}\right)}{2(N+1)^2} \Bigg).
\end{align*}

The second reason why doing continuation via an extension in $\ell^1_\eta(\N,\C)$ can be helpful, which is especially relevant in our context, is that the $\ell^1_\eta$ norm allows us to directly control derivatives.

\begin{lem}
\label{lem:ell1eta_derivatives}
Let $\eta>1$, $u\in\ell^1_\eta(\N,\C)$, and also denote by $u$ the map defined on $[-1,1]$ by
\begin{align*}
u:p\mapsto u_0 + 2\sum_{k=1}^\infty u_{k} T_{k} (p).
\end{align*}
For all $x\in[0,1]$, let
\begin{align*}
\delta[x,\eta] = 
\begin{cases}
\frac{\eta+\eta^{-1}}{2} - x & x > \frac{2}{\eta+\eta^{-1}}, \\
\frac{\eta-\eta^{-1}}{2}\sqrt{1-x^2}  & \text{otherwise}.
\end{cases}
\end{align*}
For all $p_1,p_2\in[-1,1]$, $p_1\leq p_2$, define 
\begin{align*}
C^1_\eta(p_1,p_2) &= \min \left(\max_{\substack{ k \in \N \\ 1\leq k\leq \frac{2}{\ln\eta}}} \frac{k^2}{\eta^{k}},\  \frac{1}{\delta[\max(\vert p_1\vert,\vert p_2\vert),\eta]} \right),\\
C^2_\eta(p_1,p_2) &= \min \left(\max_{\substack{ k \in \N \\ 2\leq k\leq \frac{4}{\ln\eta}}} \frac{k^2(k^2-1)}{3\eta^{k}},\  \frac{2}{\left(\delta[\max(\vert p_1\vert,\vert p_2\vert),\eta]\right)^{2}} \right).
\end{align*}
Then, the map $u$ is twice differentiable, and 
\begin{align*}
\sup_{p\in[p_1,p_2]} \left\vert u'(p)\right\vert \leq C^1_\eta(p_1,p_2) \left\Vert u\right\Vert_{\ell^1_\eta},\quad \sup_{p\in[p_1,p_2]} \left\vert u''(p)\right\vert \leq C^2_\eta(p_1,p_2) \left\Vert u\right\Vert_{\ell^1_\eta}.
\end{align*}
\end{lem}
\begin{rmk}
In fact, having $u$ in $\ell^1_\eta$ for $\eta>1$ implies that the associated map is analytic, on a so-called Bernstein ellipse of size $\eta$ containing the segment $[-1,1]$ (see e.g.~\cite{Tre13}). If useful, similar estimates on higher order derivatives could therefore also be obtained.
\end{rmk}
\begin{proof}
Classical properties of the Chebyshev polynomials yield, for all $p\in[-1,1]$,
\begin{align*}
u'(p) = \sum_{l=0}^\infty 2(2l+1) u_{2l+1} + 2 \sum_{k=1}^\infty \sum_{l=0}^\infty 2(k+2l+1)u_{k+2l+1} T_k(p),
\end{align*}
therefore
\begin{align*}
\left\vert u'(p)\right\vert &= 2\sum_{l=0}^\infty \left\vert u_{2l+1} \right\vert \eta^{2l+1} \frac{2l+1}{\eta^{2l+1}} +  2\sum_{k=1}^\infty \sum_{l=0}^\infty \left\vert u_{k+2l+1}\right\vert \eta^{k+2l+1} 2\frac{k+2l+1}{\eta^{k+2l+1}} \\
&\leq \left\Vert u\right\Vert_{\ell^1_\eta} \sup_{k\in \N_{\geq 1}} \frac{k^2}{\eta^{k}} .
\end{align*}
It is straightforward to check that this supremum is in fact a maximum, which must be reached for $k\leq \frac{2}{\ln\eta}$. A similar computation yields
\begin{align*}
\left\vert u''(p)\right\vert \leq \left\Vert u\right\Vert_{\ell^1_\eta} \frac{1}{3}\sup_{k\in \N_{\geq 2}} \frac{k^2(k^2-1)}{\eta^{k}} = \left\Vert u\right\Vert_{\ell^1_\eta} \frac{1}{3}\max_{\substack{ k \in \N \\ 2\leq k\leq \frac{4}{\ln\eta}}} \frac{k^2(k^2-1)}{\eta^{k}}.
\end{align*}
The other part of the two estimates is obtained by using that, since $u\in\ell^1_\eta(\N,\C)$, $\eta>1$, $u$ is analytic on the open Bernstein ellipse $\Ec_\eta$, defined as
\begin{align*}
\Ec_\eta = \left\{ z\in\C,\, \vert z-1\vert + \vert z+1\vert < \eta + \eta^{-1} \right\},
\end{align*}
and bounded on the closure $\overline{\Ec_\eta}$ of $\Ec_\eta$. Furthermore, one can check that, for all $x\in[0,1]$, the ball in $\C$ of center $x$ and radius $\delta[x,\eta]$ is contained in $\overline{\Ec_\eta}$. The second part of the estimates then follows from Cauchy's integral formula and the fact that
\begin{equation*}
\sup_{z\in\overline{\Ec_\eta}} \vert u(z)\vert \leq \left\Vert u\right\Vert_{\ell^1_\eta}. 
\hfill\qedhere
\end{equation*}
\end{proof}

Once we successfully use Theorem~\ref{thm:NK} with the extended space $\Xc_{\eta}$, and check that $\lambda=\Lambda$ (see Remark~\ref{rem:checking_positivity}) we can therefore combine the estimate
\begin{align*}
\left\Vert \Lambda - \barl \right\Vert_{\ell^1_\eta} \leq r,
\end{align*}
with Lemma~\ref{lem:ell1eta_derivatives}, in order to get
\begin{align}
\label{eq:C2_control}
\left\vert \Lambda''(0) - \barl''(0) \right\vert \leq \left(\frac{2}{p_+-p_-}\right)^2 C^2_\eta(0,0) r,
\end{align}
where the $\left(\frac{2}{p_+-p_-}\right)^2$ factor comes from the rescaling between $[-1,1]$ and $[p_-,p_+]$.
Since $\barl$ is merely a polynomial (a truncated Chebyshev series), $\barl''(0)$ can be evaluated explicitly, and~\eqref{eq:C2_control} provides us with a rigorous and explicit enclosure of the asymptotic variance $\Lambda''(0)$.

\subsubsection{Results}
\label{sec:shear_results}

\begin{figure}[h!]
\includegraphics[width=0.49\linewidth]{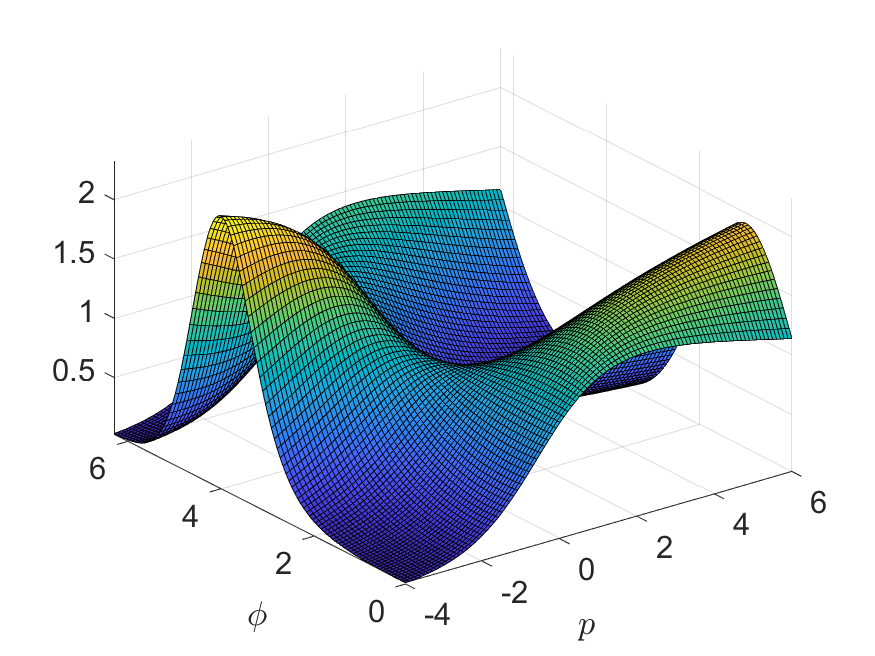}
\hfill
\includegraphics[width=0.49\linewidth]{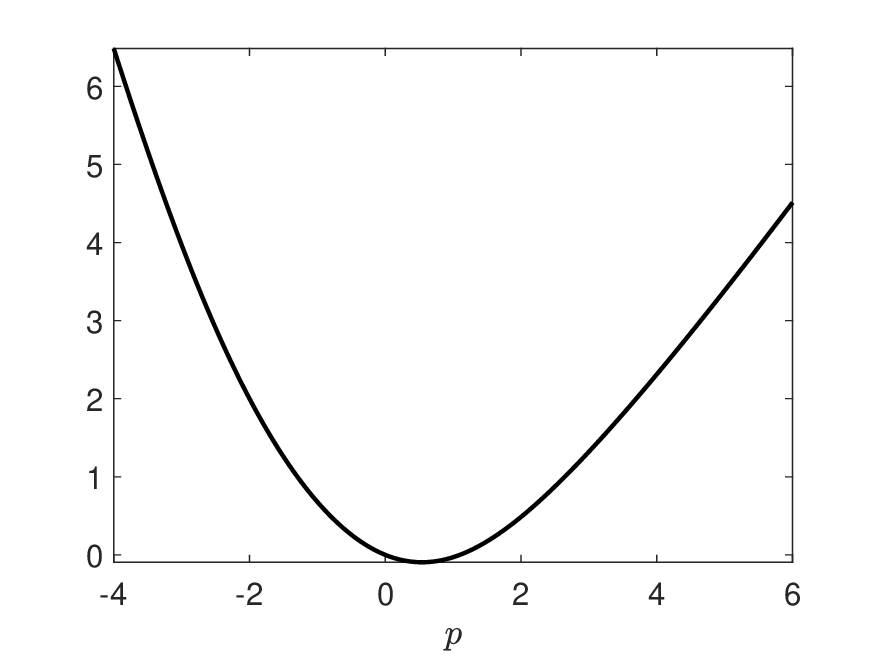}
\caption{Representation of the approximate eigenfunction $\barf = \barf(p,\phi)$ (left) and approximate eigenvalue $\barl = \barl(p)$ (right) used in the proof of Theorem~\ref{th:single_sol_shear}.}
\label{fig:approximate_eigenpair}
\end{figure}

\begin{thm}
\label{th:single_sol_shear}
Let $\alpha = 1$, $b=5$, and $\sigma = 1$ in~\eqref{eq:specvarEqusim_new_v2}. The rate function satisfies
\begin{align*}
\Ic(0) = - \inf_{p \in \mathbb{R}}  \Lambda(p) \in [0.094775\bm{0},0.094775\bm{3}].
\end{align*}
Furthermore, the asymptotic first Lyapunov exponent and the asymptotic variance are given by
\begin{align*}
\Lambda'(0) \in [-0.3523159\bm{8},-0.3523159\bm{4}] \qquad\text{and}\qquad \Lambda''(0) \in [0.65778\bm{7},0.65778\bm{9}].
\end{align*}
\end{thm}
\begin{proof}
We consider $[p_-,p_+] = [-4,6]$ and compute an accurate truncated Fourier-Chebyshev approximation $(\barf,\barl)$ of a zero of~\eqref{eq:def_Fp} over this interval, depicted on Figure~\ref{fig:approximate_eigenpair}. Using the methodology described in Sections~\ref{sec:shear_general}-\ref{sec:shear_continuation}, and in particular Theorem~\ref{thm:NK} with the space $\Xc_{\eta}$, with  $\eta = 1.01$, we obtain that there is an exact eigenpair $(f,\lambda)$ with
\begin{align*}
\left\Vert (f,\lambda) - (\barf,\barl)\right\Vert_{\Xc_{\eta}} \leq 8\times 10^{-10}.
\end{align*}
We then successfully check that $f\geq 0$ (see Remark~\ref{rem:checking_positivity}), which implies that $\lambda = \Lambda$. Using Lemma~\ref{lem:ell1eta_derivatives}, we also get a rigorous control on $\Lambda'$, which allows us to find a narrow interval $[p_1,p_2]$ containing the zero of $\Lambda'$ (using the intermediate value theorem), and then to enclose $-\Lambda([p_1,p_2])$, which provides the announced value of $I(0)$. Using again Lemma~\ref{lem:ell1eta_derivatives}, we also get the announced enclosures of $\Lambda'(0)$ and $\Lambda''(0)$.

The computational parts of the proof can be reproduced by running the Matlab code \texttt{script\_shear\_single.m}, available at~\cite{BreGit23}, together with the Intlab toolbox~\cite{Rum99} for interval arithmetic.
\end{proof}

There is of course nothing special about the parameter values that were chosen for $\alpha$, $b$ and $\sigma$ in Theorem~\ref{th:single_sol_shear}, and one can, for instance, repeat the same procedure for different values of $b$. In Figure~\ref{fig:shear_various_bs}, we show rigorously validated enclosures of $I_b(0)$ for many different values of $b$. The rigorous computations required to produce this picture can be reproduced by running the Matlab code \texttt{script\_shear\_all.m}, available at~\cite{BreGit23}, together with the Intlab toolbox~\cite{Rum99} for interval arithmetic.

\begin{rmk}
\label{rem:I0whenlambdapos}
Note that in this example, we have a transition from negative to positive asymptotic Lyapunov exponent when $b$ increases (see the evolution of $\Lambda'_b(0)$ on Figure~\ref{fig:shear_various_bs}). After that transition, when the asymptotic Lyapunov exponent is positive, the rate function $\Ic_b$ still holds the information about the large deviations principle for the FTLE (see Theorem~\ref{thm:LDPs} and Corollary~\ref{cor:LDP}), but its specific value at $0$ loses its prominent meaning.
\end{rmk}

\begin{figure}[h!]
\begin{center}
\includegraphics[width=0.49\linewidth]{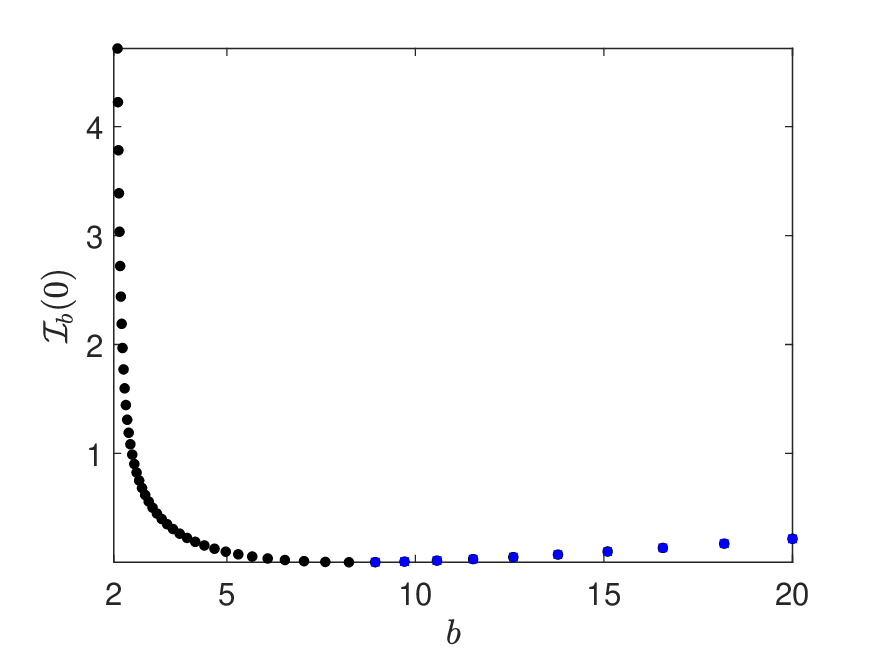}\\
\includegraphics[width=0.49\linewidth]{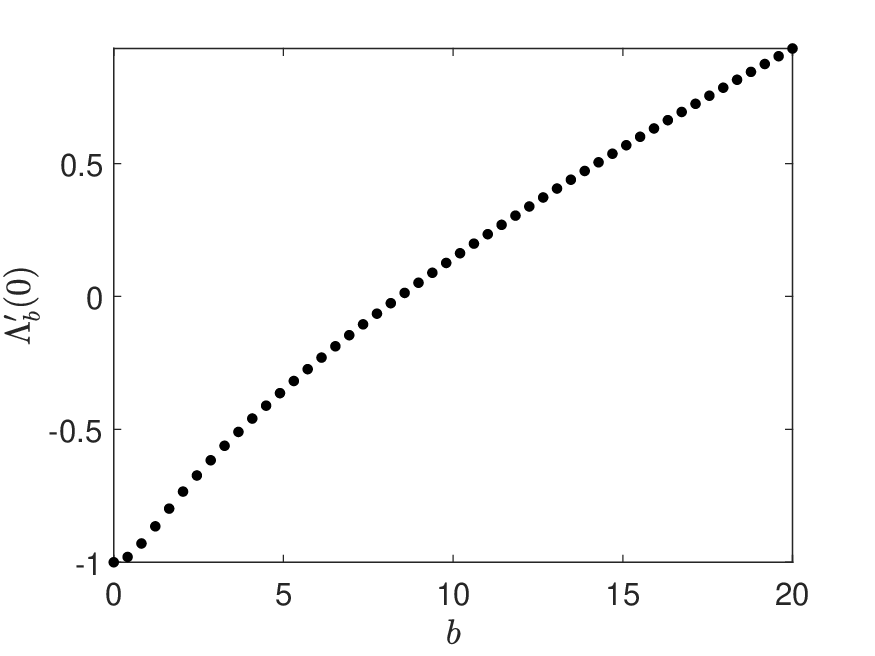}
\hfill
\includegraphics[width=0.49\linewidth]{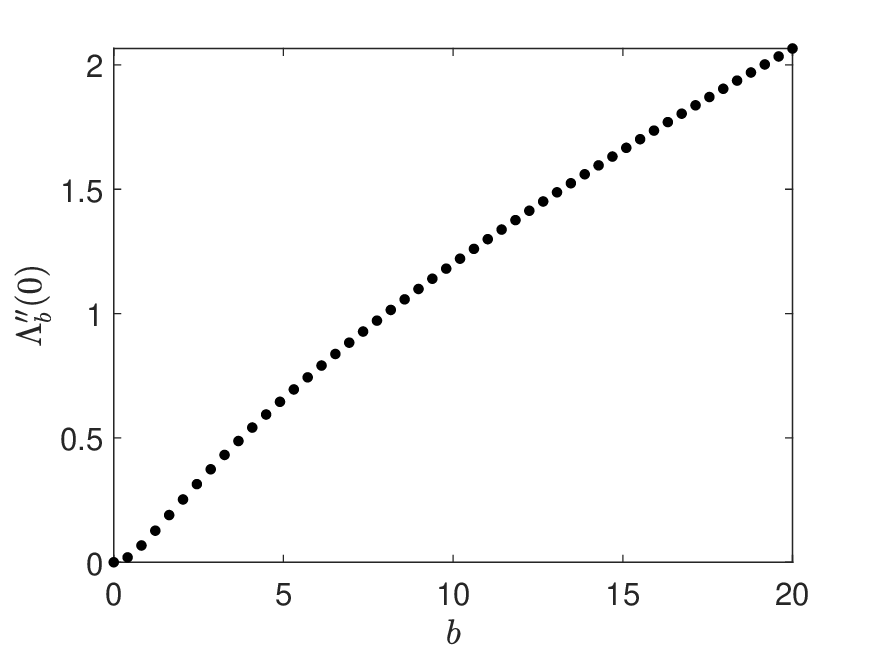}
\end{center}
\caption{Rigorous computations of $\Ic_b(0)$, $\Lambda'_b(0)$ and $\Lambda''_b(0)$ for~\eqref{eq:specvarEqusim_new_v2}, with $\alpha=1$, $\sigma=1$ and several values of $b$. The upper- and lower-bounds are too close to be depicted on the picture: all error bounds are below $9\times 10^{-9}$ for $I_b(0)$, below $2\times 10^{-7}$ for $\Lambda'_b(0)$, and below $6\times 10^{-4}$ for $\Lambda''_b(0)$. Blue square markers are used for $\Ic_b(0)$ when the corresponding asymptotic Lyapunov exponent (given by $\Lambda'_b(0)$) is positive, see Remark~\ref{rem:I0whenlambdapos}.}
\label{fig:shear_various_bs}
\end{figure}


There are some notable differences to the pitchfork example discussed in Section~\ref{sec:homotopy}, in particular comparing $\Ic_{\alpha}(0)$, as depicted in Figure~\ref{fig:pitchfork_various_alphas_intro}, and $\Ic_{b}(0)$, as depicted in Figure~\ref{fig:shear_various_bs}.
In the latter case, there is no local minimum for $\Ic_{b}(0)$ as long as the asymptotic Lyapunov exponent is negative, in contrast to the existence of the local minimum $\alpha_*$ (see Theorem~\ref{th:pitchfork_minimum}). This is due to the fact that increasing the shear factor $b$ globally adds more expansion to the system without being limited to a local effect around some (deterministic) equilibrium. There is then a global minimum at some $b_*$ where the asymptotic Lyapunov exponent crosses $0$ and, hence, $\Ic_{b_*}(0) =0$ (see Section~\ref{sec:prelim}).

\section{Outlook}
\label{sec:outlook}

In the course of this paper we have laid out the connection between positive FTLE and the values taken by the large deviations rate function. In particular, in many circumstances a transition to positive FTLE is accompanied by and detected from a qualitative change in the LDP rate function. 

Many questions arose during the writing of this paper, which we feel provide a compelling program for future work in various directions, including both the treatment of more general systems and alternative characterizations of FTLE transitions. We record some of these below.

\subsubsection*{Treatment of systems on unbounded domains}

As remarked elsewhere, the large deviations principle for unbounded domains we borrow from \cite{FStoltz20} essentially requires that one works with fully elliptic additive noise. This is somewhat unsatisfying as there should be some general LDP for hypoelliptic diffusions as well.

\subsubsection*{Treatment of systems in infinite dimensions}

    It is of natural interest to apply large deviations principles to other circumstances where there is a known FTLE transition, e.g., stochastic PDE such as the Chaffee-Infante equation as it undergoes a series of pitchfork bifurcations \cite{BlumenthalEngelNeamtu2021,BN23}. 
    To our best knowledge, however, several developments will be necessary even to prove a large deviations principle. At time of writing it is an ongoing project to lay out elements of the ergodic theory of the corresponding projective process on an infinite-dimensional domain. For recent advances in this direction we refer the reader to \cite{hairer2023spectral, rosati2024lyapunov}. 

\subsubsection*{Connection with quasi-stationary statistics}

Positive FTLE amid a negative asymptotic exponent, as we have studied in this paper, should be understood as a transient-in-time behavior for typical trajectories. Quasi-stationary ergodic theory is a powerful toolkit for the description of transient behavior one can attach to a stopping time, e.g., the statistics on conditioning remaining in a certain region of phase space \cite{Collet2013Quasi-StationaryDistributions, Pinsky1985AMeasures, Pinsky1985OnProcesses, Yaglom1947CertainProcesses}, with notable recent advances~\cite{Castro2021ExistenceApproach, Champagnat2016ExponentialQ-process, ChampagnatVillemonais2017}. A version of the multiplicative ergodic theorem for quasi-stationary dynamics \cite{castro2022lyapunov} allows to assign Lyapunov exponents in this setting, and it is not hard to show that in many cases, a positive quasi-stationary Lyapunov exponent implies positive FTLE for all sufficiently large times, even when the overall asymptotic Lyapunov exponent is negative. 

It is natural to study the extent of this connection and whether the converse is true: is there a set of assumptions guaranteeing one can assign to a system with positive FTLE a quasi-stationary statistic with a positive Lyapunov exponent?

\subsubsection*{CAP challenges}

As already mentioned in Section~\ref{sec:ToyModel}, models of the form~\eqref{eq:specvarEqusim_new_v2} appear naturally when studying the Hopf normal form, but with $B = \begin{pmatrix}
0 & \sigma \\ 0 & 0
\end{pmatrix}$.
In that case, still describing the projective process $s$ via the angular variable $\phi$, equation~\eqref{eq:Phi_model2_v2} becomes
\begin{align*}
 \rmd \phi =  b \cos^2 \phi \rmd t -  \sigma \sin\phi  \circ \rmd U_t,
\end{align*}
and the corresponding generator $L$ is no-longer uniformly elliptic:
\begin{align*}
L f = \frac{\sigma^2}{2}  \sin^2\phi \partial_\phi \left( \sin^2\phi \partial_\phi f \right)+ b \cos^2 \phi \partial_{\phi} f, \quad \phi \in (0,\pi).
\end{align*}
How to rigorously study the spectrum (of the associated tilted generator $L_p$) in this case seems to be an interesting challenge.

At the moment, it also remains unclear how rigorously control the spectrum of $L_p$ when $M$ is not compact and $L_p$ is not self-adjoint.

\section*{Acknowledgements}

We are indebted to the anonymous referee whose careful reading and insightful comments helped us significantly improve the final version of the manuscript. 

We heartily thank Michael Plum for encouraging discussions regarding the example studied in Sections~\ref{sec:pitchfork} and~\ref{sec:homotopy}.
Furthermore, we thank Sam Punshon-Smith and Dennis Chemnitz for fruitful discussions concerning moment Lyapunov exponents.

M.~Breden was supported by the ANR project CAPPS: ANR-23-CE40-0004-01.

A.~Blessing was supported by the DFG grant 543163250. Furthermore, A. Blessing acknowledges support from DFG CRC/TRR 388 {\em Rough Analysis, Stochastic Dynamics and Related Topics}, project A06 and DFG CRC 1432 {\em Fluctuations and Nonlinearities in Classical and Quantum Matter beyond Equilibrium}, project C10.

A.~Blumenthal was supported by the National Science Foundation under grants DMS-2009431 and DMS-2237360. 

M.~Engel has been supported by Germany’s Excellence Strategy – The Berlin Mathematics Research Center
MATH+ (EXC-2046/1, project ID: 390685689),  via projects AA1-8, AA1-18 and EF45-5. Furthermore, M.~Engel thanks the
DFG CRC 1114,  DFG SPP 2298 and the Einstein Foundation for support.

A.~Blumenthal and M.~Engel additionally thank MATH+ for supporting A.~Blumenthal as a MATH+ Visiting Scholar in the summer of 2023, in particular appreciating the support by Rupert Klein and Nicolas Perkowski.

\appendix


\section{Appendix}
\label{ldp}

Section \ref{subsec:noncompact2} states a large deviations principle (Theorem \ref{thm:noncompactLDP}) for the finite-time Lyapunov exponents of a stochastic flow on $\R^d$ under some conditions; a setting not covered by the classical work (e.g., \cite{BaxendaleStroock88}) on the subject which considers only the case of stochastic diffusions on a compact manifold. The LDP we desire is essentially a version of the results of \cite{FStoltz20}, extending the results for observables of an $\R^d$ diffusion to the present context of observables of a diffusion on $\R^d \times P^{d-1}$ diffusion. 

We state, without proof, a version of the G\"artner-Ellis theorem (Theorem \ref{thm:ge} below), an abstract tool for proving LDPs and briefly summarize how Theorem \ref{thm:ge} is used to prove Theorem \ref{thm:noncompactLDP}, referring to \cite{FStoltz20} for details. Not covered by \cite{FStoltz20} is the analyticity of the map $p \mapsto \Lambda(p)$ as in Theorem \ref{thm:noncompactLDP}(a)(iv), which is checked in Section \ref{subsec:analyticityLambdap}. 

The remainder of this Appendix concerns the proofs of Lemmas \ref{lem:weakStarContSamplePath} (Appendix \ref{app:weakStarCty}) and \ref{lem:specRelationGeneral4} (Appendix \ref{app:proofSpecRelation}). 
 



\subsection{The G\"artner-Ellis Theorem and Proof of Theorem \ref{thm:noncompactLDP}}\label{subsec:provingNCLDP}

Below we present a version of the G\"artner-Ellis theorem for empirical averages of scalar observables. 
In what follows, $(Z_t^\beta)$ is a family of random variables on a probability space $(\Omega, \Fc, \P)$ parametrized by $t \in [0,\infty),~ \b \in B$ for some abstract index set $B$. 
Let 
\[\Lambda_t^\beta(p) = \log \E e^{p Z_t^\beta} \,, \]
where throughout we assume the above expectation is defined and finite for all $t \geq 0, \beta \in B$. 

The following is a version of \cite[Theorem 2.3.6]{DemboZeitouni}. 
\begin{thm}\label{thm:ge}
  Assume that for all $p \in \R$, the limit 
  \begin{align}\label{eq:MLElimExistsApp}
    \Lambda(p) = \lim_{t \to \infty} \frac1t \Lambda_t^\beta(p t)
  \end{align}
  exists and is uniform over $\beta \in B$. Assume moreover that $\Lambda : \R \to \R$ is differentiable. Then, the family of random variables $(Z_t^\beta)$ satisfies a $\beta$-uniform large deviations principle (LDP) with good rate function $\Ic = \Lambda^*$. That is, for all measurable $K \subset \R$, 
  \[- \inf_{r \in K^\circ } \Ic(r) \leq \liminf_{t \to \infty} \inf_{\beta \in B} \frac1t \log \P(Z_t^\beta \in K) \leq \limsup_{t \to \infty} \sup_{\beta \in B} \frac1t \log \P(Z_t^\beta \in K) \leq - \inf_{r \in \overline{K}} \Ic(r) \,. \]
\end{thm}
Here, the \emph{Legendre-Fenchel transform} $\Lambda^*$ of $\Lambda$ is defined as 
\[\Lambda^*(r) = \sup_{p \in \R} \left(p r - \Lambda(p)\right) \,. \]
Observe that $\Lambda^*$ is automatically nonnegative since $\Lambda(0) = 0$. We say that $\Ic = \Lambda^*$ is a \emph{good rate function} if its level sets are compact for the considered topology. Using H\"older's inequality, our assumptions imply that $\Lambda$ is convex and, hence, $\Lambda^*$ is convex and a good rate function \cite[Lemma 2.3.9]{DemboZeitouni}. 

A version of Theorem \ref{thm:ge}, absent the parameter $\beta \in B$, is proved in \cite[Section 2.3]{DemboZeitouni}. The extension to the `$\beta$-uniform' version above is straightforward and omitted.

Let $(X_t, s_t)$ denote the process on $PM \cong \R^d \times P^{d-1}$ in the setting of Section \ref{subsec:noncompact2}, and let Assumption \ref{ass:noncompactLDP2} hold. 
We intend to apply Theorem \ref{thm:ge} to the family of random variables
\begin{align}\label{eq:defnZt}
  Z_t^\beta = \frac1t \int_0^t g(X_\tau, s_\tau) d \tau \, , 
\end{align}
where $B \subset PM$ is some compact subset, and $\beta = (X_0, s_0) \in B$. Here, $g : PM \to \R$ will belong to some appropriate class of observables, our ultimate goal being to set $g(x, s) = Q(x, s) = \langle s, (\rmD f_0)_x s \rangle$ so that $Z_t^\beta = \lambda_t(X_0, s_0)$.

Like in the sketch of the proof of Theorem \ref{thm:LDPs} presented in Section \ref{subsec:compact2}, convergence of $\Lambda(p)$ hinges on the spectral theory of the twisted semigroup $S^t = S^t_g$ acting on $\psi : \R^d \times P^{d-1} \to \R$ by 
\[S^t \psi(X, s) = \E_{(X, s)} \left[ \psi(X_t, s_t) ~ \exp \int_0^t g(X_\tau, s_\tau) \rmd \tau \right] \,. \]
As mentioned in Section \ref{subsec:noncompact2}, we conduct this spectral theory on the weighted uniform space $C_W$ with norm $\| \varphi \|_{C_W} = \sup_{(x, s) \in PM} | \varphi(x, s)| / W(x)$, where
\[W = e^{\theta V}\]
for $\theta \in (0,1)$ fixed, and $V$ is as in Assumption \ref{ass:noncompactLDP2}. The following is a summary of the results of Section 6 of \cite{FStoltz20}

\begin{thm}[\cite{FStoltz20}]\label{thm:specTheoryAPP}
  Let $g \in C_W$ be smooth and assume moreover that $\sup_{s \in P^{d-1}} |g(\cdot, s)| \ll | \sigma^T \nabla V|^2$. 
  \begin{itemize}
    \item[(a)] For all $t > 0$, it holds that $S^t$ is a bounded linear operator on $C_W$. 
    \item[(b)] There exist $\Lambda_g \in \R, r > 0$ such that for all $t > 0$, the operator $S^t : C_W \to C_W$ is compact and admits as a simple eigenvalue $e^{t \Lambda_g} > 0$ with the property that $\sigma(S^t) \setminus \{ e^{t \Lambda_g}\} \subset B_{r^t} (0)$.
    \item[(c)] There exists $\psi_g \in C_W$, $\psi_g > 0$ for which $S^t \psi_g = e^{t \Lambda_g} \psi_g$ for all $t > 0$. 
  \end{itemize}
\end{thm}

\begin{rmk}
    We address briefly two discrepancies between Theorem \ref{thm:specTheoryAPP} and the results in \cite[Section 6]{FStoltz20}. 

    \begin{itemize}
      \item[(a)] The paper \cite{FStoltz20} considers only diffusions on $\R^d$ and not on $P \R^d$. The extension of their methods to $P \R^d = \R^d \times P^{d-1}$ is straightforward, since the factor $P^{d-1}$ is compact and does not need to be controlled by the Lyapunov function $W$. 
      \item[(b)] Another difference is that the results of \cite[Section 6]{FStoltz20} treat the spectral theory of $S^t$ on the space $B^\infty_W$ of measurable functions $f$ for which $f / W$ is bounded, with the same norm as that for $C_W$. One way to proceed is to work, as in \cite{FStoltz20}, by establishing parts (a) -- (c) with $B_W^\infty$ replacing $C_W$. One can then deduce the corresponding statements for $C_W$ on observing that $C_W \subset B_W^\infty$ is a closed subspace. We establish below the continuity of the mapping $(x,s)\mapsto S^t \varphi(x,s)$ for $\varphi\in C_W$.
    \end{itemize}
\end{rmk}

	\begin{lem}
		For $\varphi \in C_W$, it holds that $S^t \varphi \in C_W$, i.e.~in particular the mapping $(x,s)\mapsto S^t \varphi(x,s)$ is continuous. 
	\end{lem}
	\begin{proof}
		Let $\varphi \in C_W$ be fixed. Then $S^t:B^\infty_W\to B^\infty_W$.
		Here we check the continuity of the mapping $(x, s) \mapsto S^t \varphi(x, s)$.~According to~\cite[Lemma 6.2]{FStoltz20} for any $t>0$, $a>0$ we can find a constant $c_{a,t}\geq 0$ and a compact set $K_{a,t}\subset PM$ such that
		\begin{align} \label{eq:controlStpK123} S^t W(x, s) \leq e^{-a t} W(x)+c_{a, t} {\bf 1}_{K_{a, t}}(x, s) \, . \end{align}
		Here, ${\bf 1}_{K_{a, t}}$ is the indicator function for $K_{a, t}$. 
		We set $K_a := K_{a, t/2}$, use the semigroup property and decompose
		\[S^t \varphi = \underbrace{S^{t/2} ( {\bf 1}_{K_{a}} S^{t/2} \varphi)}_{\bar \varphi_a} + S^{t/2} ( {\bf 1}_{K_{a}^c} S^{t/2} \varphi) \, .\]
		Based on~\cite[Lemma 6.4]{FStoltz20}
		we infer that $\bar \varphi_a$ is continuous.
		Using~\eqref{eq:controlStpK123} we derive a bound for $S^{t/2} ({\bf 1}_{K^c_a} S^{t/2}\varphi)$ which implies that
		\[\|S^t \varphi - \bar \varphi_a\|_{C_W} \leq e^{- a t/2} \| \varphi\|_{C_W} \,. \]
		Taking $a \to \infty$, we conclude that $S^t \varphi$ is the $C_W$-limit of continuous functions, hence, in particular continuous. 
	\end{proof}

\subsection{Analyticity of $p \mapsto \Lambda(p)$}\label{subsec:analyticityLambdap}

Throughout this subsection we work under the Assumptions of Theorem~\ref{thm:noncompactLDP}.

    To show the analyticity of $\Lambda(\cdot)$ we use the fact that $\Lambda(p)$ is the principal eigenvalue of the tilted generator of the Feynman-Kac semigroup. Consequently, we have that
    \[ S^p_t \psi_p = e^{t\Lambda(p)}\psi_p, \]
    for an eigenfunction $\psi_p$. 
    Here 
    \[  S^p_t =\E_{(X_0,s_0)} [\psi(X_t,s_t) e^{p\int_0^t Q(X_\tau,s_\tau)~\rmd\tau} ].   \]
    Similarly to~\cite[Lemma 6.6, 6.7]{FStoltz20}, we show that the operator-valued function $p\mapsto S^t_p$ is analytic which implies the analyticity of the principal isolated eigenvalue $\Lambda(p)$ in $p$. ~The following computations are based on the supermartingale property of $W(X_t) e^{-\int_0^t \frac{GW}{W}(X_\tau)\rmd \tau }$. 
We now show the complex differentiability of the map $p\mapsto S^t_p$ in the operator norm on $C_W$. To this aim, for $\psi, Q\in C_W$ 
    $ h\in\C$ with $|h|\leq R$ for $R>0$, and show that there exists a constant $C>0$, such that \begin{align}\label{derivative} \Big\| S^{p+h}_t - S^p_t - h \E_{(X_0,s_0)} \Big[ \psi(X_t,s_t) e^{p\int_0^t Q(X_\tau,S_\tau)\rmd \tau} \int_0^t Q(X_\tau,s_\tau)\rmd \tau \Big]\Big\|_{\mathcal{L}(C_W)}\leq C |h|^2 .
    \end{align}
   
First of all, we note that $\E_{(X_0,s_0)}\Big[ \psi(X_t,s_t) e^{p\int_0^t Q(X_\tau,s_\tau)\rmd \tau} \int_0^t Q(X_\tau,s_\tau)\rmd \tau \Big]$ is a bounded operator on $C_W$ for $Q\in C_W$, since for $(X_0,s_0)\in PM$ we have
   	\begin{align*}
   	&   \E_{(X_0,s_0)}\Big|\Big[ \psi(X_t,s_t) e^{p\int_0^t Q(X_\tau,s_\tau)\rmd \tau} \int_0^t Q(X_\tau,s_\tau)\rmd \tau \Big] \Big| \\& \leq \|\psi\|_{C_W} \E_{(X_0,s_0)} \Big[    W(X_t) e^{|p|\int_0^t |Q(X_\tau,s_\tau)|\rmd \tau} \int_0^t |Q(X_\tau,s_\tau)|\rmd \tau \Big] \\
   	& \leq \|\psi\|_{C_W} \E_{(X_0,s_0)} \Big[ W(X_t) e^{(|p|+1)\int_0^t |Q(X_\tau,s_\tau)|~\rmd \tau }  \Big] 
   	,
   	\end{align*}
   	where we simply used that $x\leq e^x$ for $x\in \R.$
   	Since we are working under the assumptions of Theorem~\ref{thm:noncompactLDP}, we know that $Q=\rmD f_0$ and consequently $|Q| \ll |\sigma^T \nabla V|^2$. Furthermore, for $\eta\in(0,1)$ we have that \[ \frac{GW}{W} \sim -\eta |\sigma^T \nabla V|^2. \]
   	Putting these together we obtain for a constant $c$ that 
   	\begin{align*}
   	&   \E_{(X_0,s_0)}\Big|\Big[ \psi(X_t,s_t) e^{p\int_0^t Q(X_\tau,s_\tau)\rmd \tau} \int_0^t Q(X_\tau,s_\tau)\rmd \tau \Big] \Big| \\
   	& \leq \|\psi\|_{C_W} e^{(|p|+1) ct } \E_{(X_0,s_0)} \Big[ W(X_t) e^{\int_0^t -\frac{GW}{W}(X_\tau)~\rmd \tau} \Big]\\
   	& \leq \|\psi\|_{C_W} e^{(|p|+1) ct } W(X_0),
   	\end{align*}
   	by the supermartingale property of $W(X_t) e^{-\int_0^t \frac{GW}{W}(X_\tau)\rmd \tau }$.\\
   We return to the complex differentiability of the map $p\mapsto S^t_p$ and further compute for $\psi\in C_W$ and $(X_0,s_0)\in PM$
   \begin{align*}
   \begin{split}  
   \Big|  \E_{(X_0,s_0)} \Big[  \psi(X_t,s_t) e^{p\int_0^t Q(X_\tau,s_\tau)\rmd \tau} e^{h\int_0^t Q(X_\tau,s_\tau)\rmd \tau } -\psi(X_t,s_t) e^{p\int_0^t Q(X_\tau,s_\tau)\rmd \tau}\\ -  h \psi(X_t,s_t) e^{p\int_0^t Q(X_\tau,s_\tau)\rmd \tau} \int_0^t Q(X_\tau,s_\tau)\rmd \tau \Big]  \Big|
   \end{split}
   \\
   & \leq \|\psi\|_{C_W} \E_{(X_0,s_0)}\Big[W(X_t)e^{|p|\int_0^t|Q(X_\tau,s_\tau)|\rmd \tau} \Big| 
   e^{h\int_0^t Q(X_\tau,s_\tau)\rmd \tau} -1 -h\int_0^t Q(X_\tau,s_\tau)\rmd \tau \Big| \Big]\\
   &\leq C \frac{|h|^2}{2}\|\psi\|_{C_W} \E_{(X_0,s_0)} \Big[W(X_t) e^{ (|p|+|h|) \int_0^t Q(X_\tau,s_\tau)~\rmd \tau }\Big(\int_0^t Q(X_\tau,s_\tau)\rmd \tau\Big)^2  \Big]\\
   & \leq C |h|^2 \|\psi\|_{C_W} \E_{(X_0,s_0)} [W(X_t) e^{(|p|+|h|+1) \int_0^t Q(X_\tau,s_\tau)\rmd \tau }].
   \end{align*}
   The third line was obtained applying Taylor's formula (with Lagrange remainder) to the complex-valued function $f(h)=e^{h\int_0^t Q(X_\tau,s_\tau)~\rmd\tau}$ and the last line follows due to the inequality $x^2/2\leq e^x$ for $x\geq 0$. 
   As before, using that $|Q|\ll |\sigma^T \nabla V|^2$, $GW/W\sim -\eta |\sigma^T \nabla V|^2$ together with the supermartingale property of $W(X_t) e^{-\int_0^t\frac{GW}{W}(X_\tau)\rmd \tau}$, we can bound the conditional expectation by  $e^{\tilde{c}t}W(X_0)$, for a constant $\tilde{c}>0$, proving~\eqref{derivative}.
   This means that $p\mapsto S^t_p$ is complex differentiable and therefore analytic with respect to the operator norm on $C_W$. 
Since $p\mapsto S^p_t$ is analytic, then so is the principal eigenvalue $p\mapsto e^{t\Lambda(p)}$, and hence so is the mapping $p\mapsto \Lambda(p)$. 
   
\subsection{Proof of Lemma~\ref{lem:weakStarContSamplePath}}\label{app:weakStarCty}
We provide here a proof of the weak-$*$ continuity of the law $(X,s)\mapsto \mathbf{P}_{(X,s)}$ under the assumptions of Lemma~\ref{lem:weakStarContSamplePath}. 
\begin{defn}
Let $\P_z$ denote the law of a stochastic process $Z$ on a separable Banach space $E$ starting in $z\in E$. The laws $(\P_z)_{z\in E}$ are called weakly continuous if for any sequence $(z_n)\subset E$ that converges to $z\in E$, the measures $\P_{z_n}\to \P_{z}$ converge weakly, i.e. 
\[ \int \phi \rmd{\P_{z_n}} \to \int \phi \rmd {\P_z},\] 
for every bounded, continuous functional $\phi:E\to \R$. 
	\end{defn}
In our case we work on the path space $E:=C([0,1];\R^N)$ and analyze the convergence of the laws $\P_{z_n}\to \P_z$ with respect to the weak-$*$ topology on the spaces of measures on $E$. This is referred to as the weak-$*$ continuity of $z\mapsto \P_z$.  

\medskip 

\noindent {\bf Setting (1): $M$ is compact}

\medskip

Using the Nash embedding theorem,  isometrically embed $PM$ into $\R^N$ for some $N$ sufficiently large.  Construct extensions $F_i$ of the original vector fields $\tilde{f}_i$ in such a way so as to satisfy the global Lipschitz conditions
\begin{gather} \label{eq:appRegcondsCpct}\begin{gathered}
  \sum_{i = 0}^m |F_i(z_1) - F_i(z_2)| + |H(z_1) - H(z_2)|  \leq B |z_1 - z_2| \,, \quad \text {and} \\ 
  \sum_{i = 0}^{m} |F_i(z)| + |H(z)| \leq B (1 + |z|) \, , 
\end{gathered}
\end{gather}
for a constant $B>0$, where $H= (H_j)_{j=1}^n$ is the vector field given by
 $$H_j(x) = \sum_{i=1}^m \sum_{k=1}^n \frac{\partial (F_i)_j(x)}{\partial x_k}(x) (F_i)_k(x) .$$
This translates the problem of controlling the law of $(X_\bullet, s_\bullet) = (X_t, s_t)_{t \in [0,1]}$ to a problem concerning the law $(Z_\bullet)$ for a diffusion $(Z_t)$ of the form
\begin{align} \label{eq:appSDE}\rmd Z_t = F_0(Z_t) \rmd t + \sum_{i = 1}^m F_i(Z_t) \circ \rmd W_t^i\end{align}
on $\R^N, N \geq 1$, for standard Brownian motions $(W_t^i)$.
In this case, the following suffices for concluding Lemma \ref{lem:weakStarContSamplePath}. 
\begin{prop}
  Suppose \eqref{eq:appRegcondsCpct} holds.  
  Then, \eqref{eq:appSDE} admits unique strong solutions for all initial $Z_0$, and moreover, $z \mapsto {\bf P}_{z}$ varies continuously in the weak-$*$ topology on the space of measures on $C([0,1], \R^N)$.  
\end{prop}
Here,  ${\bf P}_z$ denotes the law of $(Z_\bullet) = (Z_t)_{[0,1]}$ on $C([0,1], \R^N)$ conditioned on $Z_0 = z$. 

\begin{proof}[Proof Sketch]
  Below, $C > 0$ is a constant depending only on $B$ in \eqref{eq:appRegcondsCpct}, the value of which can change from line to line. 

  That \eqref{eq:appRegcondsCpct} implies unique existence of strong solutions is standard and omitted. To prove weak-$*$ continuity, 
  let $z^i, i =1,2$ be two fixed initial conditions and $(Z_t^i), i = 1,2$ be the corresponding diffusions with $Z_0^i = z^i$.
  By the Portmanteau theorem it suffices to show 
  \[\left| \int \Phi d {\bf P}_{z^1} - \int \Phi d {\bf P}_{z^2} \right| = \left| \E \Phi(Z_\bullet^1) - \E \Phi(Z_\bullet^2) \right| \leq C |z^1 - z^2| \]
  for Lipschitz continuous $\Phi : C([0,1], \R^N) \to \R$. Since $\Phi$ is Lipschitz, it suffices to establish \begin{align}\label{eq:LipestAPP} \E \sup_{t \in [0,1]} |Z_t^1 - Z_t^2| \leq C |z^1 - z^2| \,. \end{align} 

  To start, translate \eqref{eq:appSDE} into It\^o form and note that the corresponding vector fields $(\tilde F_i)$ are smooth and also satisfy \eqref{eq:appRegcondsCpct}.  Using the integral form of~\eqref{eq:appSDE} in the sense of It\^o, a standard argument using It\^o isometry and Gronwall's inequality implies
  \begin{align}\label{eq:appGronwall}
    \E |Z_t^1 - Z_t^2|^2 \leq C |z^1 - z^2|^2 \,. 
  \end{align}
  
Next, we derive
  	\begin{align*}
  	\sup_{t \in [0,1]} |Z_t^1 - Z_t^2|^2 \leq C \int_0^1 |Z_s^1 - Z_s^2|^2 \rmd s + C \sum_i \sup_{t \in [0,1]} \underbrace{\left| \int_0^t \left( {F}_i(Z_t^1) - {F}_i(Z_t^2)\right) \rmd W_t^i \right|^2}_{Y_t^i} \, . 
  	\end{align*}
  	Under expectation, the first RHS term is bounded by $C |z^1 - z^2|^2$. The processes $Y_t^i$ are nonnegative submartingales, so using the $L^2$ version of Doob's maximal inequality (\cite[Theorem 1.3.8]{karatzas1991brownian}) we get 
  	\[\E \sup_{t \in [0,1]}| Y_t^i|^2 \leq  4 \E |Y_1^i|^2 \,. \] 
  	Therefore by It\^o isometry and \eqref{eq:appRegcondsCpct}, the RHS is bounded like \begin{align*} & \E \sup_{t \in [0,1]} \left| \int_0^t \left( {F}_i(Z_t^1) - {F}_i(Z_t^2)\right) \rmd W_t^i \right|^2 \leq C \E \left| \int_0^1 \left( {F}_i(Z_t^1) - {F}_i(Z_t^2)\right) \rmd W_t^i   \right |^2  \\
  	& \leq C\E \int_0^1 |F_i(Z^1_t) - F_i(Z^2_t)|^2\rmd{t}  \leq C \E \int_0^1 |Z_t^1 - Z_t^2|^2 \rmd t \leq C |z^1 - z^2|^2.
  	\end{align*}
  	Adding up all these contributions, \eqref{eq:LipestAPP} follows and the proof is complete.

\end{proof}

\medskip 

\noindent {\bf Setting (2): $M$ is noncompact}

\medskip

We can no longer hope for the extension vector fields $(F_i)$ to satisfy \eqref{eq:appRegcondsCpct}, and must rely instead on a drift-type condition as in Assumption \ref{ass:drift}. To this end, fix an isometric embedding of $P^{d-1}$ to $\hat P \subset \R^{N'}$ for some $N' \geq 1$, and set $N = d + N'$. For $i \geq 1$ we can let $F_i = \begin{pmatrix}
  f_i \\ 0 
\end{pmatrix}$ be constant vector fields, while the drift term $F_0$ is built so as to agree with $\tilde f_0$ on $\R^d \times \hat P$ and so that $F_0 = \begin{pmatrix}
  f_0 \\ 0 
\end{pmatrix}$ outside a ball of radius 1 of $\R^d \times \hat P$. Writing $\Lc$ for the infinitesimal generator associated to the resulting diffusion $(Z_t)$, the following result will imply our desired version of Lemma \ref{lem:weakStarContSamplePath}. 



\begin{prop} \label{prop:weakStarCtyApp}
  Assume there exists a $C^2$ function $\mathcal{V} : \R^N \to [1,\infty)$ with the following properties. 
  \begin{itemize}
    \item[(a)] For some constant $c > 0$, it holds that 
  \[\Lc \mathcal{V} \leq c \mathcal{V} \,. \]
  \item[(b)] For all $C \geq 1$, the sublevel set $\mathcal{V}^{-1} [1,C]$ is compact. 
  \end{itemize}
  Then, \eqref{eq:appSDE} admits unique strong solutions for all initial $Z_0$, and moreover, $z \mapsto {\bf P}_{z}$ varies continuously in the weak-$*$ topology on the space of measures defined on $C([0,1], \R^N)$. 
\end{prop}

\begin{proof}
  Fix a family of vector fields $F_i^R$ so that (i) $F_i^R = F_i$ on $B_{2R}(0)$ and (ii) $F_i^R$ satisfy \eqref{eq:appRegcondsCpct} (with constants possibly depending on $R$). 
  
  Let $v_0 \geq 1$ be fixed, and throughout we impose that $z^1, z^2 \in \mathcal{V}^{-1}[1,v_0]$. We write $(Z_t^i), i =1 ,2$ for the corresponding paths with $Z_0^i = z^i$. As before, $\Phi : C([0,1], \R^N) \to \R$ is Lipschitz and bounded. Set $Z_t^{i, R}$ to be diffusion path with $(F_i^R)$ replacing $(F_i)$, and observe that $Z_t^{i, R} = Z_t^i$ for all $t \leq \tau^{i,R}$, where
  \[\tau^{i, R} = \inf\{ t > 0 : |Z^i_t| \geq R\} \,. \] 
  We estimate 
  \begin{align*}
   \left| \int \Phi d {\bf P}_{z^1} - \int \Phi d {\bf P}_{z^2}\right|  = & \left|\mathbb{E} \Phi (Z_\bullet^1) -\mathbb{E} \Phi(Z_\bullet^2)\right| \\ 
     \leq &  \left|\mathbb{E} \Phi(Z_\bullet^1) -\mathbb{E} \Phi(Z_\bullet^{1, R})\right| + \left|\mathbb{E} \Phi(Z_\bullet^{1, R}) -\mathbb{E} \Phi(Z_\bullet^{2, R})\right| \\
    & +\left|\mathbb{E} \Phi(Z_\bullet^{2, R}) -\mathbb{E} \Phi(Z_\bullet^{2}) \right|\,. 
\end{align*}

Fix $\epsilon > 0$. Arguing as in the proof of \cite[Theorem 3.5]{khasminskii2012stochastic}, it follows that there exists $R$, depending only on $v_0$, such that 
\begin{align}
  \P(\tau^{i, R} < 1) \leq \frac{\epsilon}{6 K},  
\end{align}
where $K$ is an upper bound for $\Phi$ and $i = 1, 2$. For this value of $R$,  \eqref{eq:LipestAPP} applies with a constant $C = C_R$ depending on $R$, and 
\[| \E \Phi(Z_\bullet^i) - \E \Phi(Z_{\bullet}^{i, R})| \leq \E |(\Phi(Z_\bullet^i) - \Phi(Z_\bullet^{i, R})) {\bf 1}_{\{ \tau^{i, R} < 1\}} | \leq \frac{\epsilon}{3} \,.  \]
Set $ \delta = \epsilon / (3 C_R L)$, where $L$ is a Lipschitz constant for $\Phi$. We conclude that $|\int \Phi d {\bf P}_{z^1} - \int \Phi d {\bf P}_{z^2} | \leq \epsilon$ for all $z^1, z^2 \in \mathcal{V}^{-1}[1,v_0]$ with $|z^1 - z^2| < \delta$. This completes the proof. 
\end{proof}

\subsection{Proof of Lemma \ref{lem:specRelationGeneral4}}\label{app:proofSpecRelation}

We conclude with a proof of Lemma \ref{lem:specRelationGeneral4}, restated below for convenience. 

\begin{lem}
  Let $(\Bc, |\cdot|)$ be a Banach space, and let $\Bc_0 \subset \Bc$ be a dense subspace equipped with its own norm $|\cdot|_0$ with respect to which $(\Bc_0, |\cdot|_0)$ is also Banach, and for which $|x| \leq C |x|_0$ for all $x \in \Bc_0$ and for some $C > 0$. 
  
  Let $A : \Bc \to \Bc$ be a bounded linear operator for which $A(\Bc_0) \subset \Bc_0$ and $A_0 := A|_{\Bc_0} : (\Bc_0, |\cdot|_0) \circlearrowleft$ is a bounded linear operator. Finally, assume $A : (\Bc, |\cdot|) \circlearrowleft$ and $A_0 : (\Bc_0, |\cdot|_0) \circlearrowleft$ are both compact on their respective spaces. Then, 
  \[\rho(A : \Bc \to \Bc) = \rho(A_0 : \Bc_0 \to \Bc_0) \, , \]
  i.e., the spectral radius of $A$ regarded on $\Bc$ coincides with that of $A_0 = A|_{\Bc_0}$ regarded on $\Bc_0$. 
\end{lem}

\begin{proof}
Let $r = \rho(A : \Bc \to \Bc), r_0 = \rho(A_0 : \Bc_0 \to \Bc_0)$. For now, let us assume $r_0	> 0$; the case $r_0 = 0$ is treated at the end. 

\medskip 

\noindent {\bf Step 1}: We first confirm that when $r_0 > 0$, it holds that $r \geq r_0$. For this, let $\lambda_0 \in \sigma(A_0 : \Bc_0 \to \Bc_0)$ such that $|\lambda_0| = r_0$, and let $f_0 \in \Bc_0 \setminus \{ 0 \}$ such that $A_0 f_0 = \lambda_0 f_0$. Then, $A f_0 = \lambda_0 f_0$, hence $\lambda_0 \in \sigma(A : \Bc \to \Bc)$ and so $r_0 = |\lambda_0| \leq r$. 

\medskip 

\noindent {\bf Step 2}: Conversely, we check now that $r \leq r_0$. For this, note that under our assumption $r_0 > 0$, it holds that $r \geq r_0 > 0$ by Step 1. It follows from the spectral theorem for compact operators that for an open and dense set of $f \in \Bc$, $|A^n f|^{1/n} \to r$ as $n \to \infty$. This open and dense set meets $\Bc_0 \setminus \{ 0 \}$ at some vector $f_0$. We estimate
\[|A^n f_0| \leq C |A^n f_0|_0 = C |A^n_0 f_0|_0 \leq C |A^n_0|_0 |f_0|_0 \, . \]
Taking $n$-th roots and $n \to \infty$, the LHS converges to $r$ by construction, while the RHS converges to $r_0$ by the Gelfan'd formula \cite[Lemma IX.1.8]{dunford1964linear}. We conclude that $r \leq r_0$. 

\medskip

\noindent {\bf Step 3}: It remains to address the case $r_0 = 0$, for which it suffices to show that $r = 0$. For the sake of contradiction, suppose $r > 0$. Then, one can follow the argument in Step 2, which implies in particular that $r \leq r_0 = 0$. This is a contradiction. 
\end{proof}

\bibliographystyle{abbrv}
\bibliography{Bibliography.bib}

\begin{thebibliography}{10}

\bibitem{AriGazKoc21}
G.~Arioli, F.~Gazzola, and H.~Koch.
\newblock Uniqueness and bifurcation branches for planar steady
  {N}avier--{S}tokes equations under {N}avier boundary conditions.
\newblock {\em Journal of Mathematical Fluid Mechanics}, 23(3):1--20, 2021.

\bibitem{AriKocTer05}
G.~Arioli, H.~Koch, and S.~Terracini.
\newblock Two novel methods and multi-mode periodic solutions for the
  {F}ermi-{P}asta-{U}lam model.
\newblock {\em Communications in mathematical physics}, 255(1):1--19, 2005.

\bibitem{arnold1984formula}
L.~Arnold.
\newblock A formula connecting sample and moment stability of linear stochastic
  systems.
\newblock {\em SIAM Journal on Applied Mathematics}, 44(4):793--802, 1984.

\bibitem{Arnoldbook}
L.~Arnold.
\newblock {\em Random dynamical systems}.
\newblock Springer Monographs in Mathematics. Springer-Verlag, Berlin, 1998.

\bibitem{ArnoldKliemann1987}
L.~Arnold and W.~Kliemann.
\newblock Large deviations of linear stochastic differential equations.
\newblock In H.~J. Engelbert and W.~Schmidt, editors, {\em Stochastic
  Differential Systems}, pages 115--151, Berlin, Heidelberg, 1987. Springer
  Berlin Heidelberg.

\bibitem{ArnoldKliemannOel}
L.~Arnold, W.~Kliemann, and E.~Oeljeklaus.
\newblock Lyapunov exponents of linear stochastic systems.
\newblock In {\em Lyapunov exponents ({B}remen, 1984)}, volume 1186 of {\em
  Lecture Notes in Math.}, pages 85--125. Springer, Berlin, 1986.

\bibitem{ArnoldOelPardoux}
L.~Arnold, E.~Oeljeklaus, and E.~Pardoux.
\newblock Almost sure and moment stability for linear {I}t\^{o} equations.
\newblock In {\em Lyapunov exponents ({B}remen, 1984)}, volume 1186 of {\em
  Lecture Notes in Math.}, pages 129--159. Springer, Berlin, 1986.

\bibitem{arnold1996toward}
L.~Arnold, N.~Sri~Namachchivaya, and K.~R. Schenk-Hopp{\'e}.
\newblock Toward an understanding of stochastic hopf bifurcation: a case study.
\newblock {\em International Journal of Bifurcation and Chaos},
  6(11):1947--1975, 1996.

\bibitem{BaxendaleStroock88}
P.~H. Baxendale and D.~W. Stroock.
\newblock Large deviations and stochastic flows of diffeomorphisms.
\newblock {\em Probab. Theory Related Fields}, 80(2):169--215, 1988.

\bibitem{bedrossian2022almost}
J.~Bedrossian, A.~Blumenthal, and S.~Punshon-Smith.
\newblock {Almost-sure exponential mixing of passive scalars by the stochastic
  Navier--Stokes equations}.
\newblock {\em The Annals of Probability}, 50(1):241--303, 2022.

\bibitem{bellet2006ergodic}
L.~R. Bellet.
\newblock Ergodic properties of markov processes.
\newblock In {\em Open Quantum Systems II: The Markovian Approach}, pages
  1--39. Springer, 2006.

\bibitem{BerShu12}
F.~A. Berezin and M.~Shubin.
\newblock {\em The Schr{\"o}dinger Equation}, volume~66.
\newblock Springer Science \& Business Media, 2012.

\bibitem{berglund2002metastability}
N.~Berglund and B.~Gentz.
\newblock Metastability in simple climate models: pathwise analysis of slowly
  driven langevin equations.
\newblock {\em Stochastics and Dynamics}, 2(03):327--356, 2002.

\bibitem{berglund2002pathwise}
N.~Berglund and B.~Gentz.
\newblock Pathwise description of dynamic pitchfork bifurcations with additive
  noise.
\newblock {\em Probability theory and related fields}, 122(3):341--388, 2002.

\bibitem{berglund2002sample}
N.~Berglund and B.~Gentz.
\newblock A sample-paths approach to noise-induced synchronization: Stochastic
  resonance in a double-well potential.
\newblock {\em The Annals of Applied Probability}, 12(4):1419--1470, 2002.

\bibitem{berglund2006noise}
N.~Berglund and B.~Gentz.
\newblock {\em Noise-induced phenomena in slow-fast dynamical systems: a
  sample-paths approach}.
\newblock Springer Science \& Business Media, 2006.

\bibitem{birkhoff1957extensions}
G.~Birkhoff.
\newblock {Extensions of Jentzsch’s theorem}.
\newblock {\em Transactions of the American Mathematical Society},
  85(1):219--227, 1957.

\bibitem{BN23}
D.~Bl\"omker and A.~Neam{\c t}u.
\newblock Bifurcation theory for {SPDE}s: finite-time {L}yapunov exponents and
  amplitude equations.
\newblock {\em SIAM J. Appl. Dyn. Syst}, 2:2150--2179, 2023.

\bibitem{BreGit23}
A.~Blumenthal, M.~Breden, M.~Engel, and A.~Neam\c{t}u.
\newblock Matlab code associated to the paper {``Detecting random bifurcations
  via rigorous enclosures of large deviations rate functions''}.
\newblock \url{https://github.com/MaximeBreden/LDPforFTLE}, 2024.

\bibitem{BlumenthalEngelNeamtu2021}
A.~Blumenthal, M.~Engel, and A.~Neam\c{t}u.
\newblock On the pitchfork bifurcation for the {C}hafee-{I}nfante equation with
  additive noise.
\newblock {\em Probab.~Theory.~Rel.~Fields}, 187:603--627, 2023.

\bibitem{blumenthal2022positive}
A.~Blumenthal and Y.~Yang.
\newblock Positive lyapunov exponent for random perturbations of predominantly
  expanding multimodal circle maps.
\newblock {\em Annales de l'Institut Henri Poincar{\'e} C}, 39(2):419--455,
  2022.

\bibitem{Bre23}
M.~Breden.
\newblock A posteriori validation of generalized polynomial chaos expansions.
\newblock {\em SIAM Journal on Applied Dynamical Systems}, 22(2):765--801,
  2023.

\bibitem{BreChu24}
M.~Breden and H.~Chu.
\newblock Constructive proofs for some semilinear {PDE}s on
  {$H^2(e^{|x|^2/4},\mathbb{R}^d)$}.
\newblock {\em arXiv preprint arXiv:2404.04054}, 2024.

\bibitem{BreChuLamRas24}
M.~Breden, H.~Chu, J.~S. Lamb, and M.~Rasmussen.
\newblock Rigorous enclosure of {L}yapunov exponents of stochastic flows.
\newblock {\em arXiv preprint arXiv:2411.07064}, 2024.

\bibitem{BreEng23}
M.~Breden and M.~Engel.
\newblock Computer-assisted proof of shear-induced chaos in stochastically
  perturbed hopf systems.
\newblock {\em Ann. Appl. Probab.}, 33(2):1052--1094, 2023.

\bibitem{BreHen24}
M.~Breden and O.~Hénot.
\newblock Efficient rigorous continuation via {C}hebyshev series expansion.
\newblock {\em in preparation}, 2024.

\bibitem{Callawayetal}
M.~Callaway, T.~S. Doan, J.~S.~W. Lamb, and M.~Rasmussen.
\newblock The dichotomy spectrum for random dynamical systems and pitchfork
  bifurcations with additive noise.
\newblock {\em Ann. Inst. Henri Poincar\'{e} Probab. Stat.}, 53(4):1548--1574,
  2017.

\bibitem{castro2022lyapunov}
M.~M. Castro, D.~Chemnitz, H.~Chu, M.~Engel, J.~S. Lamb, and M.~Rasmussen.
\newblock {The Lyapunov spectrum for conditioned random dynamical systems}.
\newblock {\em arXiv preprint arXiv:2204.04129}, 2022.

\bibitem{Castro2021ExistenceApproach}
M.~M. Castro, J.~S.~W. Lamb, G.~Olic\'{o}n-M\'{e}ndez, and M.~Rasmussen.
\newblock Existence and uniqueness of quasi-stationary and quasi-ergodic
  measures for absorbing {M}arkov chains: a {B}anach lattice approach.
\newblock {\em Stochastic Process. Appl.}, 173:Paper No. 104364, 19, 2024.

\bibitem{Champagnat2016ExponentialQ-process}
N.~Champagnat and D.~Villemonais.
\newblock {Exponential convergence to quasi-stationary distribution and
  Q-process}.
\newblock {\em Probability Theory and Related Fields}, 164(1):243--283, 2016.

\bibitem{ChampagnatVillemonais2017}
N.~Champagnat and D.~Villemonais.
\newblock Uniform convergence to the {$Q$}-process.
\newblock {\em Electron. Commun. Probab.}, 22:Paper No. 33, 7, 2017.

\bibitem{ChemnitzEngel}
D.~Chemnitz and M.~Engel.
\newblock Positive {L}yapunov exponent in the {H}opf normal form with additive
  noise.
\newblock {\em Comm. Math. Phys.}, 2023.

\bibitem{Collet2013Quasi-StationaryDistributions}
P.~Collet, S.~Mart{\'{i}}nez, and J.~San~Mart{\'{i}}n.
\newblock {\em {Quasi-Stationary Distributions}}.
\newblock Probability and Its Applications. Springer Berlin Heidelberg, Berlin,
  Heidelberg, 2013.

\bibitem{conway2019course}
J.~B. Conway.
\newblock {\em A course in functional analysis}, volume~96.
\newblock Springer, 2019.

\bibitem{CrauelFlandoli}
H.~Crauel and F.~Flandoli.
\newblock Additive noise destroys a pitchfork bifurcation.
\newblock {\em J. Dynam. Differential Equations}, 10(2):259--274, 1998.

\bibitem{DayLesMis07}
S.~Day, J.-P. Lessard, and K.~Mischaikow.
\newblock Validated continuation for equilibria of {PDEs}.
\newblock {\em SIAM J. Numer. Anal.}, 45(4):1398--1424, 2007.

\bibitem{DemboZeitouni}
A.~Dembo and O.~Zeitouni.
\newblock {\em Large Deviations Techniques and Applications}.
\newblock Springer, 1998.

\bibitem{DoanEngeletal}
T.~S. Doan, M.~Engel, J.~S.~W. Lamb, and M.~Rasmussen.
\newblock Hopf bifurcation with additive noise.
\newblock {\em Nonlinearity}, 31(10):4567--4601, 2018.

\bibitem{dunford1964linear}
N.~Dunford and J.~T. Schwartz.
\newblock Linear operators ii: spectral theory.
\newblock 1964.

\bibitem{EngelNagel}
K.-J. Engel and R.~Nagel.
\newblock {\em One-parameter semigroups for linear evolution equations}, volume
  194 of {\em Graduate Texts in Mathematics}.
\newblock Springer-Verlag, New York, 2000.
\newblock With contributions by S. Brendle, M. Campiti, T. Hahn, G. Metafune,
  G. Nickel, D. Pallara, C. Perazzoli, A. Rhandi, S. Romanelli and R.
  Schnaubelt.

\bibitem{EngelLambRasmussen}
M.~Engel, J.~S.~W. Lamb, and M.~Rasmussen.
\newblock Bifurcation analysis of a stochastically driven limit cycle.
\newblock {\em Comm. Math. Phys.}, 365(3):935--942, 2019.

\bibitem{FStoltz20}
G.~Ferr\'e and G.~Stoltz.
\newblock Large deviations of empirical measures of diffusions in weighted
  topologies.
\newblock {\em Electron. J. Probab.}, 25(121):1--52, 2020.

\bibitem{freidlin1998random}
M.~Freidlin and A.~Wentzell.
\newblock {\em Random perturbations}.
\newblock Springer, 1998.

\bibitem{GalMonNis20}
S.~Galatolo, M.~Monge, and I.~Nisoli.
\newblock Existence of noise induced order, a computer aided proof.
\newblock {\em Nonlinearity}, 33(9):4237, 2020.

\bibitem{ginelli2007characterizing}
F.~Ginelli, P.~Poggi, A.~Turchi, H.~Chat{\'e}, R.~Livi, and A.~Politi.
\newblock Characterizing dynamics with covariant {Lyapunov} vectors.
\newblock {\em Physical review letters}, 99(13):130601, 2007.

\bibitem{Hairer2011}
M.~Hairer.
\newblock On {M}alliavin's proof of {H}\"{o}rmander's theorem.
\newblock {\em Bull. Sci. Math.}, 135(6-7):650--666, 2011.

\bibitem{hairer2011yet}
M.~Hairer and J.~C. Mattingly.
\newblock {Yet another look at Harris’ ergodic theorem for Markov chains}.
\newblock In {\em Seminar on Stochastic Analysis, Random Fields and
  Applications VI: Centro Stefano Franscini, Ascona, May 2008}, pages 109--117.
  Springer, 2011.

\bibitem{hairer2023spectral}
M.~Hairer and T.~Rosati.
\newblock Spectral gap for projective processes of linear spdes.
\newblock {\em arXiv preprint arXiv:2307.07472}, 2023.

\bibitem{hormander1967hypoelliptic}
L.~H{\"o}rmander.
\newblock Hypoelliptic second order differential equations.
\newblock 1967.

\bibitem{ImkellerLederer}
P.~Imkeler and C.~Lederer.
\newblock Some formulas for lyapunov exponents and rotation numbers in two
  dimensions and the stability of the harmonic oscillator and the inverted
  pendulum.
\newblock {\em Dynamical Systems: An International Journal}, 16:29--61, 2001.

\bibitem{karatzas1991brownian}
I.~Karatzas and S.~Shreve.
\newblock {\em Brownian motion and stochastic calculus}, volume 113.
\newblock Springer Science \& Business Media, 1991.

\bibitem{khasminskii2012stochastic}
R.~Khasminskii.
\newblock {\em Stochastic stability of differential equations}.
\newblock Springer, 2012.

\bibitem{kontoyiannis2005large}
I.~Kontoyiannis and S.~Meyn.
\newblock {Large deviations asymptotics and the spectral theory of
  multiplicatively regular Markov processes}.
\newblock 2005.

\bibitem{lamb2015topological}
J.~Lamb, M.~Rasmussen, and C.~Rodrigues.
\newblock Topological bifurcations of minimal invariant sets for set-valued
  dynamical systems.
\newblock {\em Proceedings of the American Mathematical Society},
  143(9):3927--3937, 2015.

\bibitem{Lan82}
O.~E. Lanford~III.
\newblock A computer-assisted proof of the {F}eigenbaum conjectures.
\newblock {\em Bulletin of the American Mathematical Society}, 6(3):427--434,
  1982.

\bibitem{LinYoung}
K.~K. Lin and L.-S. Young.
\newblock Shear-induced chaos.
\newblock {\em Nonlinearity}, 21(5):899--922, 2008.

\bibitem{Liu15}
X.~Liu.
\newblock A framework of verified eigenvalue bounds for self-adjoint
  differential operators.
\newblock {\em Applied Mathematics and Computation}, 267:341--355, 2015.

\bibitem{lucarini2020new}
V.~Lucarini and A.~Gritsun.
\newblock A new mathematical framework for atmospheric blocking events.
\newblock {\em Climate Dynamics}, 54(1):575--598, 2020.

\bibitem{marshall1993toward}
J.~Marshall and F.~Molteni.
\newblock Toward a dynamical understanding of atmospheric weather regimes.
\newblock {\em J. Atmos. Sci}, 50:1792--1818, 1993.

\bibitem{metzner2012analysis}
P.~Metzner, L.~Putzig, and I.~Horenko.
\newblock Analysis of persistent nonstationary time series and applications.
\newblock {\em Communications in Applied Mathematics and Computational
  Science}, 7(2):175--229, 2012.

\bibitem{mitsui2014dynamics}
T.~Mitsui and K.~Aihara.
\newblock Dynamics between order and chaos in conceptual models of glacial
  cycles.
\newblock {\em Climate dynamics}, 42:3087--3099, 2014.

\bibitem{mitsui2016effects}
T.~Mitsui and M.~Crucifix.
\newblock Effects of additive noise on the stability of glacial cycles.
\newblock {\em Mathematical Paradigms of Climate Science}, pages 93--113, 2016.

\bibitem{Moo79}
R.~E. Moore.
\newblock {\em Methods and applications of interval analysis}.
\newblock SIAM, 1979.

\bibitem{NagNakWak02}
K.~Nagatou, M.~Nakao, and M.~Wakayama.
\newblock Verified numerical computations for eigenvalues of non-commutative
  harmonic oscillators.
\newblock {\em Numerical Functional Analysis and Optimization},
  23(5-6):633--650, 2002.

\bibitem{NakPluWat19}
M.~T. Nakao, M.~Plum, and Y.~Watanabe.
\newblock {\em Numerical Verification Methods and Computer-Assisted Proofs for
  Partial Differential Equations}, volume~53 of {\em Springer Series in
  Computational Mathematics}.
\newblock Springer Singapore, 2019.

\bibitem{Oksendal}
B.~Oksendal.
\newblock {\em Stochastic differential equations}.
\newblock Springer, 2013.

\bibitem{Pinsky1985AMeasures}
R.~Pinsky.
\newblock A classification of diffusion processes with boundaries by their
  invariant measures.
\newblock {\em Ann. Probab.}, 13(3):693--697, 1985.

\bibitem{Pinsky1985OnProcesses}
R.~G. Pinsky.
\newblock On the convergence of diffusion processes conditioned to remain in a
  bounded region for large time to limiting positive recurrent diffusion
  processes.
\newblock {\em Ann. Probab.}, 13(2):363--378, 1985.

\bibitem{Plu90}
M.~Plum.
\newblock Eigenvalue inclusions for second-order ordinary differential
  operators by a numerical homotopy method.
\newblock {\em Zeitschrift f{\"u}r angewandte Mathematik und Physik ZAMP},
  41(2):205--226, 1990.

\bibitem{Plu01}
M.~Plum.
\newblock Computer-assisted enclosure methods for elliptic differential
  equations.
\newblock {\em Linear Algebra Appl.}, 324(1-3):147--187, 2001.

\bibitem{quinn2021dynamical}
C.~Quinn, D.~Harries, and T.~J. O’Kane.
\newblock Dynamical analysis of a reduced model for the north atlantic
  oscillation.
\newblock {\em Journal of the Atmospheric Sciences}, 78(5):1647--1671, 2021.

\bibitem{reed1978iv}
M.~Reed and B.~Simon.
\newblock {\em Methods of Mathematical Physics IV: Analysis of Operators},
  volume~4.
\newblock Elsevier, 1978.

\bibitem{rosati2024lyapunov}
T.~Rosati.
\newblock Lyapunov exponents in a slow environment.
\newblock {\em Stochastic Processes and their Applications}, 170:104296, 2024.

\bibitem{Rum99}
S.~M. Rump.
\newblock {INTLAB - INTerval LABoratory}.
\newblock {\em Developments in Reliable Computing, Kluwer Academic Publishers,
  Dordrecht, pp}, pages 77--104, 1999.

\bibitem{sato2018dynamical}
Y.~Sato, T.~S. Doan, J.~S. Lamb, and M.~Rasmussen.
\newblock Dynamical characterization of stochastic bifurcations in a random
  logistic map.
\newblock {\em arXiv preprint arXiv:1811.03994}, 2018.

\bibitem{Stroock86}
D.~W. Stroock.
\newblock On the rate at which a homogeneous diffusion approaches a limit, an
  application of large deviation theory to certain stochastic integrals.
\newblock {\em Ann. Probab.}, 14(3):840--859, 1986.

\bibitem{teramae2004robustness}
J.-n. Teramae and D.~Tanaka.
\newblock Robustness of the noise-induced phase synchronization in a general
  class of limit cycle oscillators.
\newblock {\em Physical review letters}, 93(20):204103, 2004.

\bibitem{Touchette}
H.~Touchette.
\newblock Introduction to dynamical large deviations of {M}arkov processes.
\newblock {\em arXiv:1705.06492v6}, 2022.

\bibitem{Tre13}
L.~N. Trefethen.
\newblock {\em Approximation theory and approximation practice}, volume 128.
\newblock Siam, 2013.

\bibitem{Tuc11}
W.~Tucker.
\newblock {\em Validated numerics: a short introduction to rigorous
  computations}.
\newblock Princeton University Press, 2011.

\bibitem{BerBreLesVee21}
J.~B. van~den Berg, M.~Breden, J.-P. Lessard, and L.~van Veen.
\newblock Spontaneous periodic orbits in the {N}avier--{S}tokes flow.
\newblock {\em Journal of Nonlinear Science}, 31(2):1--64, 2021.

\bibitem{BerLes15}
J.~B. van~den Berg and J.-P. Lessard.
\newblock Rigorous numerics in dynamics.
\newblock {\em Notices Amer. Math. Soc.}, 62(9), 2015.

\bibitem{viennet2022guidelines}
A.~Viennet, N.~Vercauteren, M.~Engel, and D.~Faranda.
\newblock Guidelines for data-driven approaches to study transitions in
  multiscale systems: The case of lyapunov vectors.
\newblock {\em Chaos: An Interdisciplinary Journal of Nonlinear Science},
  32(11), 2022.

\bibitem{Yaglom1947CertainProcesses}
A.~M. Yaglom.
\newblock {Certain limit theorems of the theory of branching random processes}.
\newblock In {\em Doklady Akad. Nauk SSSR (NS)}, volume~56, 1947.

\bibitem{Yam98}
N.~Yamamoto.
\newblock A numerical verification method for solutions of boundary value
  problems with local uniqueness by {B}anach's fixed-point theorem.
\newblock {\em SIAM J. Numer. Anal.}, 35:2004--2013, 1998.

\bibitem{zmarrou2007bifurcations}
H.~Zmarrou and A.~J. Homburg.
\newblock Bifurcations of stationary measures of random diffeomorphisms.
\newblock {\em Ergodic Theory and Dynamical Systems}, 27(5):1651--1692, 2007.

\end{thebibliography}

\end{document}